\documentclass[11pt]{amsart}
\usepackage[a4paper,top=4cm,bottom=3cm,left=3cm,right=3cm]{geometry}
\usepackage[T1]{fontenc}
\usepackage{bm}
\usepackage{amsmath,amssymb,amsfonts,amsthm} 
\usepackage{hyperref}
\usepackage{graphicx}
\usepackage{dirtytalk}
\usepackage{pgfplots}
\usepackage{yfonts}
\usepackage{url}
\usepackage{enumitem}
\usepackage{esint}
\usepackage{multirow}
\usepackage{blindtext}
\usepackage{tensor}
\usepackage{mathrsfs}
\newcommand{\numberset}{\mathbb}

\newcommand{\R}{\numberset{R}}
\newcommand{\M}{\textup M}

\newcommand{\cin}{\textup{c}_{\textup{in}}}

\newcommand{\hp}{{\textup{hyp}}}
\newtheorem{thm}{Theorem}[section]
\newtheorem{rem}[thm]{Remark}
\newtheorem{lem}[thm]{Lemma}
\newtheorem{dfn}[thm]{Definition}
\newtheorem{cor}[thm]{Corollary}
\newtheorem{prp}[thm]{Proposition}

\title[Foliations by STCMCs for asymptotically hyperboloidal initial data sets]{Foliations by constant spacetime mean curvature surfaces for asymptotically hyperboloidal initial data sets}

\author{Jacopo Tenan}
\address{(Jacopo Tenan) Dipartimento di Matematica, Universit\`a degli Studi di Roma "Tor Vergata", Via della Ricerca Scientifica, 00133, Roma, Italy}

\date{}
\numberwithin{equation}{section}
\pgfplotsset{compat=1.18}
\begin{document}

\begin{abstract}
We construct an exhaustive family of constant spacetime mean curvature (STCMC) surfaces for initial data sets close to the anti-de Sitter-Schwarzschild hyperboloid. In particular, we obtain such a foliation as the long time limit of the volume preserving spacetime mean curvature flow starting from the constant mean curvature foliation constructed by Neves-Tian (Geom. Funct. Anal., 2009). As an application, inspired by the definition of STCMC center of mass for initial data sets proposed in the asymptotically Euclidean setting by Cederbaum-Sakovich (Calc. Var. PDE, 2021), we study the center of mass of an asymptotically hyperboloidal initial data set.
\end{abstract}

\maketitle
%
%
%
\section{Introduction}
In Mathematical General Relativity, there are two standard pictures where hyperboloidal initial data sets occur, as pointed out in \cite{chruscieldelay}. One picture is drawn by considering the Riemannian slice $\{t=0\}$ in the anti-de Sitter spacetime
\begin{equation*}
\boldsymbol g_l=\left(1+\frac{s^2}{l^2}\right)^{-1} ds^2+s^2 g_0-\left(1+\frac{s^2}{l^2}\right) dt^2,
\end{equation*}
equipped with induced metric $\overline g_{\textup{adS}}=\boldsymbol g_l\big|_{\{t=0\}}$ and second fundamental form $\overline K=0$. Above, $l>0$ is a parameter typically called \textit{hyperbolic radius} and $g_0$ is the round metric on the unit sphere $\mathbb S^2$. In this case, the \textit{density constraint equation} takes the form 
\begin{equation*}
\Lambda=\frac{\overline R_{\overline g_{\textup{adS}}}}2\equiv-\frac3{l^2},
\end{equation*}
which is the so called (hyperbolic) \textit{cosmological constant}. This time-symmetric picture of the anti-de Sitter spacetime is considered interesting by some physicists because of the so called AdS-CFT correspondence, see for example \cite{horowitzmyers}. More generally, if a Riemannian manifold $(M,\overline g)$ is asymptotic to the hyperbolic space and $\overline K$ is close to zero, we say that $(M,\overline g,\overline K)$ is an \textit{asymptotically hyperbolic initial data set.}\\
\indent On the other hand, one can model the same situation as above by requiring $\Lambda=0$ and considering a non-time symmetric initial data set. In particular, from now on, we will refer to a \textit{hyperboloidal initial data set} as a triple $(M,\overline g,\overline K)$ where $\overline K=\overline\lambda\overline g$, for some $\overline\lambda\in\R$. Models of such kind are physically relevant since these initial data sets expand at infinity as gravitational waves do. In this paper, we will in particular focus our attention on the so called \textit{anti-de Sitter-Schwarzschild spacetime} (of mass $m>0$ and hyperbolic radius $l>0$), that is
\begin{equation*}
\boldsymbol g_{m,l}=\left(1+\frac{s^2}{l^2}-\frac{m}{s}\right)^{-1}ds^2+s^2g_0-\left(1+\frac{s^2}{l^2}-\frac{m}{s}\right) dt^2.
\end{equation*}
In this case, choosing for convenience $l=1$, we consider the Riemannian slice $M=\{t=0\}$ as the hyperboloidal initial data set with metric and second fundamental form given by
\begin{equation}\label{hyperboloidalids11}
\overline{g}_{\textup{adS}}^{m}=\left(1+s^2-\frac{m}{s}\right)^{-1}ds^2+s^2g_0, \qquad \overline K_{\textup{adS}}^{m}=\overline g_{\textup{adS}}^{m}.
\end{equation}
More precisely, we will be interested in considering initial data sets which are asymptotic to $(M,\overline g_{\textup{adS}}^m,\overline K_{\textup{adS}}^m)$ in a precise sense, see Definitions \ref{asymptoticallyhyp} and \ref{asymptoticallyhypK} below. Asymptotically hyperboloidal initial data set have been recently considered in \cite{dahlsakovich}, \cite{hirschhyunzhang}, \cite{maerten}, \cite{sancassanivelu}. See also \cite{chakhuri} for a definition when $\overline K$ vanishes at infinity.
\\
\indent In this paper, our main goal will be that of defining a suitable notion of center of mass for asymptotically hyperboloidal initial data sets. Such a goal, for metrics asymptotic to some reference metric, has been an interesting object of study in Geometric Analysis since the breakthrough paper by Huisken-Yau \cite{huiskenyau}, where a constant mean curvature (CMC) foliation of asymptotically Schwarzschildean manifold has been constructed and used as a geometrically meaningful coordinate system. It also turned out that a (geometric) definition of center of mass is closely related to such a foliation. In the asymptotically Euclidean case, the result by Huisken-Yau has been followed by several works which generalized the asymptotic assumptions and the reference model. See \cite{metzger}, \cite{huang}, \cite{nerz}, \cite{eichmairkoerber} for a non-exhaustive list of results. All of these generalization of the Huisken-Yau result are based on methods other than curvature flows, while in the former paper \cite{huiskenyau} the main tool to construct the foliation is the \textit{volume preserving mean curvature flow} (VPMCF). Moreover, in \cite{corvinowu} it has been showed that the geometric center of mass defined in \cite{huiskenyau} agrees in many cases with the center of mass defined by the ADM-formulation of the initial value problem for the Einstein’s equation. Recently, Sinestrari and the author studied the VPMCF in the context of asymptotically Euclidean spaces, recovering in some cases the foliation of \cite{nerz} and \cite{eichmairkoerber}, see \cite{vpmcf}.\\
\indent The foliation has been generalized to the case of a non-time symmetric asymptotically Euclidean initial data set by Cederbaum-Sakovich \cite{cederbaumsakovich}, which constructed an exhaustive family of constant spacetime mean curvature surfaces (STCMC) for $(M,\overline g,\overline K)$ close to $(\R^3,\overline g_{\textup{Eucl}},0)$. We recall that the spacetime mean curvature for a surface $\Sigma\subset M$ is defined as 
\begin{equation*}
\mathcal H=\sqrt{H^2-P^2}, \qquad P=\textup{tr}_g(\overline K),
\end{equation*}
where $H$ is the mean curvature of $\Sigma$ and $g$ is the metric induced on $\Sigma$ by $\overline g$. The existence of this foliation, in the case of positive ADM-energy, have been reproved by the author using a spacetime modification of the VPMCF, called \textit{volume preserving spacetime mean curvature flow} (VPSTMCF), see \cite{vpstmcf}.\\
\indent In the hyperbolic context, the VPMCF has been used by Rigger \cite{rigger} in order to generalize the construction by Huisken-Yau to asymptotically anti-de Sitter-Schwarzschild spaces. In particular Rigger requires $\overline g-\overline g_{\textup{adS}}^m$ to be $O(\rho^{-5})$, together with its derivatives up to the fourth order, where $\rho>1$ is the hyperbolic distance from the origin of the coordinates. This result has been later improved and widely generalized by Neves-Tian \cite{nevestian}. They considered the anti-de Sitter-Schwarzschild metric in \eqref{hyperboloidalids11} which, after a suitable change of coordinate $r=r(s)>0$, can be written as
\begin{equation}\label{metricnevestian}
\overline g_{\textup{adS}}^m=dr^2+\left(\sinh^2 r+\frac{m}{3\sinh r}+O(e^{-3r})\right)g_0,
\end{equation}
and foliated this Riemannian 3-manifold by spheres of constant mean curvature. We remark that here $g_0=O(e^{-2r})$. This CMC-foliation by Neves-Tian is constructed by a careful analysis of the stability operator 
\begin{equation}\label{stabilityintro}
L^\Sigma f=-\Delta f+\left(|A|^2+\overline{\textup{Ric}}(\nu,\nu)\right)f,
\end{equation}
and relies on elliptic methods applied to \eqref{stabilityintro}. This problem has been also later addressed by Nerz in \cite{nerzhyp}, with different hypotheses.\\
\indent In the hyperbolic context, the VPMCF has also been studied by Cabezas Rivas-Miquel in \cite{cabezas-rivasmiquel}, proving that long time existence for the flow is assured when starting from convex (by horospheres) surfaces in the hyperbolic space form. Moreover, they also proved that the flow converges to geodesic spheres in the hyperbolic space. A similar problem in a different setting has been recently addressed by Albert-Niclòs and Cabezas-Rivas in \cite{albert-nicloscabezas-rivas}.\\
\indent In this paper, we aim to three goals. First of all, we analyze the VPSTMCF in the context of asymptotically hyperboloidal initial data set, and we show that the flow exists for every positive time and converges to STCMC surfaces, provided that the initial datum is suitably round, see Definition \ref{dfnroundenss36}. The Riemannian manifold where the flow is studied here is the one considered by Neves-Tian, see the metric \eqref{metricnevestian}. Secondly, we use these long time limits in order to show that asymptotically initial data sets can  be foliated by STCMC surfaces. Finally, we generalize the notion of \textit{spacetime center of mass} (STCoM) introduced in \cite{cederbaumsakovich} to this context and prove some result on the STCoM.\\
\indent In the following, we will always assume that $(M,\overline g,\overline K)$ is asymptotically hyperboloidal in the sense that 
\begin{equation}\label{hypinitialdatasetperturbed}
\overline g=\overline g_{\textup{adS}}^m+O(e^{-5r}), \qquad \overline K=\overline g+O(e^{-5r}).
\end{equation}
Moreover, we assume this decay to hold up to two derivatives for the metric and up to one derivative for $\overline K$. Notice that, since we require in \eqref{hypinitialdatasetperturbed} the same decay for the metric and the spacetime second fundamental form, this is equivalent to require that also $\overline K=\overline g_{\textup{adS}}^m+O(e^{-5r})$.\\
\indent As briefly recalled above, the VPSTMCF has been introduced in \cite{vpstmcf} and it is a modification of the VPMCF which takes into account the non-time symmetric nature of the initial data set. In particular, we say that $\Sigma_t$ evolves by VPSTMCF if there exists a surface $\Sigma$ and a smooth family of immersions $F_t:\Sigma\to M$ such that 
\begin{equation}\label{vpstmcfintro}
\frac{\partial F}{\partial t}=-(\mathcal H-\hbar) \nu,
\end{equation}
and $\Sigma_t=F_t(\Sigma)$, where $\mathcal H=\mathcal H(t)$ is the spacetime mean curvature of $\Sigma_t$, $\hbar=\hbar(t)$ is its integral mean and $\nu=\nu_t$ the unit normal to $\Sigma_t$. We remark that, unlike the VPMCF, the flow in \eqref{vpstmcfintro} is not area-decreasing, and this introduces some difficulties in proving the convergence of the flow once long time existence has been established. In order to prove both existence and convergence, we will compute the evolution of integral and Sobolev norms of the speed $\mathcal H-\hbar$. Here, the decay assumptions \eqref{hypinitialdatasetperturbed} play a crucial role, since they allow a careful estimate on the evolution of the norm $\|\mathcal H-\hbar\|_{L^2(\Sigma_t)}$, which finally turns out to be exponentially decaying. See Lemma \ref{expdecay413} and Corollary \ref{cor417} for more details. The analysis on the evolution of the speed $\mathcal H-\hbar$ will also allow us to deduce estimates on the mean curvature $H$ without estimating its evolution. Note that, differently from the VPSTMCF in the Euclidean case, in the hyperbolic case the speed of the flow is not close to its VPMCF counterpart, since $P$ is asymptotic to a non-zero constant, and not decaying to zero at infinity. As a consequence, the evolution equation for the traceless second fundamental form $|\overset{\circ}{A}|$ along the flow contains some undesirable reaction terms, which are bounded but not suitably small. For this reason, we will deduce a bound on this quantity by estimating the Sobolev norm of the speed $\mathcal H-\hbar$, using bootstrapping methods as in \cite{nerz}. Moreover, a graph representation of the evolving surface $\Sigma_t$ will be constructed, generalizing the \textit{à-la-Schauder} analysis of the stability operator on hyperbolic coordinate spheres of \cite{nevestian} to the non-CMC case.\\
\indent At the core of the paper, we will show that \textit{suitably round} initial surfaces converge to STCMC surfaces along the VPSTMCF. This (now vague) notion of roundness is strictly related to being \textit{almost CMC-surfaces}, i.e. surfaces whose mean curvature is constant modulo a small controlled error, and it will be made more precise in Definitions \ref{almostCMC215} and \ref{dfnroundenss36}. In particular, the initial data for our flow will be the CMC-surfaces constructed by Neves-Tian in \cite{nevestian}. They proved the existence of an exhaustive family of CMC surfaces, say $\left\{\Sigma^\sigma\right\}_{\sigma\geq\sigma_0}$ for some $\sigma_0>1$ large, such that $\sigma=\sqrt{\frac{|\Sigma^\sigma|}{4\pi}}$ and $\Sigma^\sigma$ is close to a coordinate sphere of radius $\hat r=\hat r(\sigma)>0$, with $|\sinh \hat r-\sigma|=O(1)$. See Section \ref{thmnevestian} for more details. The whole construction by Neves-Tian in \cite{nevestian} is based on the fact that metrics asymptotic to \eqref{metricnevestian} are written in \textit{balanced coordinates}. Furthermore, Neves-Tian ask the mass aspect function to be diagonal. To be precise, the metrics under consideration are part of a bigger family of metrics asymptotic to 
\begin{equation}\label{asymptnevestianII}
dr^2+\sinh^2 r g_0+\frac{\boldsymbol m}{3\sinh r},
\end{equation}
where $\boldsymbol m$ is a symmetric 2-tensor on $\mathbb S^2$. In our case, the tensor $\boldsymbol m$ satisfies both 
\begin{equation}
\int_{\mathbb S^2} \textup{tr}_{g_0}\boldsymbol m \ d\mu_{g_0}=0, \qquad \boldsymbol m=mg_0,
\end{equation}
with $m>0$. The first condition is the definition of balanced coordinates, while the fact that $\boldsymbol m$ is diagonal is a very restrictive hypothesis. More in general, this can be replaced by requiring that the \textit{mass aspect function} is strictly positive, i.e. $\textup{tr}_{g_0}\boldsymbol m>0$. To this aim, Neves-Tian devoted a second work \cite{nevestian2}. Furthermore, the balancedness of the coordinates has been removed by Cederbaum-Cortier-Sakovich in \cite{cederbaumcortiersakovich}.\\
\indent In the rest of the paper, we will always suppose that our mass aspect function is diagonal, with positive (scalar) mass $m>0$. We finally state the main theorem and some consequences. Starting from a specific initial datum $\Sigma$ belonging to the CMC family $\left\{\Sigma^\sigma\right\}_{\sigma\geq\sigma_0}$ recalled above for some $\sigma\geq \sigma_0$, we define $\Sigma_t^\sigma:=F_t^\sigma(\Sigma^\sigma)$, with $F_t^\sigma$ solution of \eqref{vpstmcfintro} with $F_0^\sigma(\Sigma)=\Sigma^\sigma$, and show that the flow exists for every $t>0$ and converges to a STCMC. 
\begin{thm}\label{mainthmintro}
Let $(M,\overline g,\overline K)$ be a an asymptotically hyperboloidal initial data set, with $m>0$. Let $\Sigma$ be a CMC leaf of the foliation by Neves-Tian recalled above or equivalently a CMC surface belonging to the family of surfaces described in Section \ref{thmnevestian} with area radius $\sigma:=\sqrt{|\Sigma|/4\pi}$ and approximate radius $\hat r=\hat r(\sigma)>1$ such that 
\begin{equation*}
H_\Sigma=2\frac{\cosh \hat r}{\sinh \hat r}-\frac{m}{\sinh^3 \hat r}+O(e^{-4\hat r}).
\end{equation*}
Then the volume preserving spacetime mean curvature flow \eqref{vpstmcfintro} has a solution $\Sigma_t$, with $\Sigma_0=\Sigma$, that exists for every positive time. Moreover, $\Sigma_t$ remains close to a coordinate sphere for every $t>0$ and finally converges to a STCMC-surface as $t\to\infty$.
\end{thm}
We expect that the computation carried on in the following Sections could be adapted to the positive mass aspect function case, with some additional hypotheses on $\textup{tr}_{g_0}\boldsymbol m$ such as in \cite{nevestian2}. Once this has been done, also the balancedness condition can be removed exactly as in \cite[Thm. 3.9]{cederbaumcortiersakovich}. We remark that, since in the case of \cite{cederbaumcortiersakovich} the asymptotics on the metric are preserved by the performed isometries, the same must hold for the spacetime second fundamental form since it has the same asymptotics of the metric.\\
\indent Once one sets $\Sigma_{\textup{st}}^\sigma:=\lim_{t\to\infty} F_t^\sigma(\Sigma^\sigma)$, which exists because of Theorem \ref{mainthmintro}, the following Corollary is a standard consequence of the main theorem. Its proof is sketched in Section \ref{thefoliationproof}.
\begin{cor}\label{foliationcor12} Let $(M,\overline g,\overline K)$ be a an asymptotically hyperboloidal initial data set, with $m>0$. There exists a family $\left\{\Sigma_{\textup{st}}^\sigma\right\}_{\sigma\geq\sigma_0}$, with $\sigma_0>1$ large, such that 
\begin{enumerate}[label=\textup{(\roman*)}]
\item $\Sigma_{\textup{st}}^\sigma$ is a STCMC surface;
\item There exists a map $\Psi:[\sigma_0,\infty)\times \mathbb S^2\to M$ such that $\Psi(\sigma,\mathbb S^2)=\Sigma_{\textup{st}}^\sigma$ and 
\begin{equation*}
\Psi([\sigma_0,\infty)\times\mathbb S^2)=M\setminus B,
\end{equation*}
where $B\subset M$ is a compact set;
\item If $\sigma\neq \sigma'$, then $\Sigma_{\textup{st}}^\sigma\cap \Sigma_{\textup{st}}^{\sigma'}=\emptyset$;
\item Each leaf is stable in the sense of the definition in \textup{\eqref{spacetimestability}} and, in a certain "roundess neighborhood", it is the unique surface with a given value of the (constant) spacetime mean curvature.
\end{enumerate}
\end{cor}
We recall that, together with the work by Cederbaum-Sakovich \cite{cederbaumsakovich} which introduced STCMC foliations, an asymptotic foliation by STCMC surfaces has been constructed for asymptotic Schwarzchildean lightcones by Kröncke-Wolff in \cite{kronckewolff}, using a modification of mean curvature flow called area preserving null mean curvature flow. The author recently learned\footnote{Personal communication with A. Ellithy, June 2026.} that the construction of a STCMC foliation in the case of asymptotically hyperbolic initial data sets is part of an upcoming work by A. Ellithy, D. Lundberg and A. Sakovich. Furthermore, local STCMC foliations have been recently constructed by Metzger-Peñuela Diaz in \cite{metzgerpenueladiaz} using a Lyapunov Schmidt reduction.\\
\indent Once the foliation in Corollary \ref{foliationcor12} is constructed, the standard way to define a center of mass, which goes back to \cite{huiskenyau}, is to compute the barycenter of each leaf of the foliation and then to take the limit, provided this exists. In the Euclidean case, the barycenter has the standard definition of the integral of the position on the surface weighted by the area. In the hyperbolic case, which is not an affine space, Cederbaum-Cortier-Sakovich in \cite{cederbaumcortiersakovich} proposed a definition of barycenter using the canonical isometric immersion of $\mathbb H^3$ into the hyperboloid in the Minkowski space. In particular, they defined the barycenter of a surface $\Sigma\subset M$ as 
\begin{equation*}
\boldsymbol z_\Sigma:=I^{-1}(\hat C_\Sigma),
\end{equation*}
where $I:\mathbb H^3\to \mathscr{H}:=\{w\in \mathbb R^{1,3}: \ |w|_{\mathbb R^{1,3}}=-1\}$ is the canonical isometric immersion and $\hat C_\Sigma\in \mathscr{H}$ is the barycenter of $I(\Sigma)\subset \mathscr{H}$. See Section \ref{conclusion5} and Appendix \ref{appendixisoimm} for the exact definitions. They proved that the foliation by Neves-Tian admits limit of the barycenter of the limits and this coincide with the origin of the coordinates of the hyperbolic space. After our STCMC foliation is constructed, we adopt their definition to show the following Corollary.
\begin{cor} Let $(M,\overline g,\overline K)$ be a an asymptotically hyperboloidal initial data set, with $m>0$, and let $\Sigma_t^\sigma$ be the solution of the flow constructed in Theorem \ref{mainthmintro}, for some $\sigma\geq\sigma_0$. Then, setting 
\begin{equation*}
\Sigma_{\textup{st}}^\sigma:=\lim_{t\to\infty} \Sigma_t^\sigma,
\end{equation*}
the family $\{\Sigma_{\textup{st}}^\sigma\}_{\sigma\geq\sigma_0}$ is a STCMC foliation. Moreover the family of barycenters $\boldsymbol z_{\Sigma_{st}^{\sigma}}$ converges to $\boldsymbol 0\in\mathbb H^3$ as $\sigma\to\infty$.
\end{cor}
Notice that this is not too surprising at the light of \cite[Cor. 2]{wolff}, since in the exact anti-de Sitter-Schwarzschild case STCMC and CMC surfaces (i.e. coordinate slices) coincide. We finally also remark that in the general case of balanced coordinates as in \eqref{asymptnevestianII}, with $\boldsymbol m$ strictly positive as above but not diagonal, the approach used in \cite[Thm. 3.9]{cederbaumcortiersakovich} should prove that barycenters are approaching the mass vector associated to the corresponding mass aspect function as defined in \cite{cederbaumcortiersakovich}. 
\section*{Acknowledgements}
\ \ \ \ \ \ \ The author acknowledges support from the MUR Excellence Department Project MatMod@TOV, CUP E83C23000330006, awarded to the Department of Mathematics, University of Rome Tor Vergata, and from the MUR Prin 2022 Project "\textit{Contemporary perspectives on geometry and gravity}" CUP E53D23005750006. The author is a member of the group GNAMPA of INdAM (Istituto Nazionale di Alta Matematica).\\
\indent The author would like to thank Carla Cederbaum for suggesting the topic and for useful discussions, Gerhard Huisken and Markus Wolff for helpful conversations and Carlo Sinestrari for insightful comments on the initial draft of the manuscript. The author is also grateful to Yuqiao Li for carefully reading the first draft of this paper and for many useful comments.\\
\indent Part of this work has been written during a visit to the "Geometric Analysis, Differential Geometry and Relativity Theory" group at Eberhard Karls Universität in Tübingen, to which the author is very grateful for hospitality.
\section{Preliminaries and notations}
\subsection{Main definitions and notations}\label{maindefn11}
The hyperbolic space can be written through several (isometric) representations. We will identify the hyperbolic space $(\mathbb H^3,\overline g_{\hp})$ with Euclidean space $\R^3$ equipped with the metric 
\begin{equation*}
\overline g_{\hp}=dr^2+\sinh^2(r) g_0,
\end{equation*}
where $r=|\vec x|$ is the Euclidean distance from the origin and $g_0$ is the round metric on $\mathbb S^2$. Coordinate spheres described by the equation $r=\hat r$ thus inherit the metric $g_{\mathbb S_{\hat r}}$ given by
\begin{equation*}
g_{\mathbb S_{\hat r}}=\sinh^2(\hat r)g_0,
\end{equation*}
which allows to treat them as round spheres of radius $\sigma=\sinh(\hat r)$. Note that we will always indicate the metric on a 3-manifold with a bar, while the induced metric on a surface $\Sigma$ will simply be $g_\Sigma$. In the following, the radius $\sigma$ will have another meaning, but for this introductory section we define the parameter as the hyperbolic sine of the coordinate radius of a sphere. Note that the Euclidean radius coincides with the coordinate radius, and the hyperbolic area is $4\pi\sigma$. The (intrinsic) diameter of $\mathbb S_{\hat r}$ is $\pi\sinh \hat r$.\\
\indent As a reference metric, we will use the so called \textit{anti-de Sitter-Schwarzschild metric} of mass $m$, i.e.
\begin{equation*}
\overline g_{m}=dr^2+\left(\sinh^2\left(r\right)+\frac{m}{3\sinh\left(r\right)}\right) g_0,
\end{equation*}
where $m>0$ is the \textit{mass parameter}. With respect to this metric, the coordinate spheres $\mathbb S_{\hat r}$ are umbilical and the mean curvature of a coordinate sphere is given by
\begin{equation}\label{eq12}
H^{m}_{\mathbb S_{\hat r}}=2\frac{\cosh \hat r}{\sinh \hat r}-\frac{m}{\sinh^3 \hat r}.
\end{equation}
From now on, we will consider Riemannian metrics $(M,\overline g)$ satisfying the following definition, which has been studied by Neves-Tian in \cite{nevestian}.
\begin{dfn}[Neves-Tian]\label{asymptoticallyhyp} A Riemannian \textup{3-}manifold $(M,\overline g)$ is \textup{asymptotically anti-de Sitter Schwarzschild} if there exists a compact set $C\subset M$ and a diffeomorphism $\vec x:M\setminus C\to \mathbb H^3\setminus \overline B_1$ such that, setting $r=|\vec x|$, there exist $R_0>1$ and $\overline c>0$ such that 
\begin{equation*}
\left|\partial^\beta\left(\left(\vec x\right)_*\overline g-\overline g_{m}\right)\right|\leq \overline c e^{-5r},
\end{equation*}
for every multi-index $\beta: \ |\beta|\leq 2$ and $r>R_0$.
\end{dfn}
In this paper, we will be particularly interested in initial data sets. For this reason, we give the following definition 
\begin{dfn}\label{asymptoticallyhypK} An initial data set $(M,\overline g,\overline K)$ is \textup{asymptotically hyperboloidal} if $(M,\overline g)$ is as in \textup{Definition \ref{asymptoticallyhyp}} and there exist $R_0>0$ and $\overline c>0$ such that
\begin{equation*}
\left|\partial^\beta\left(\left(\vec x\right)_*\overline K-\left(\vec x\right)_*\overline g\right)\right|\leq \overline c e^{-5r},
\end{equation*}
for every multi-index $|\beta|\leq 1$ and $r>R_0$.
\end{dfn}
\subsection{Surfaces}
In the following, $\Sigma\stackrel{\iota}{\hookrightarrow} M$ will be a hypersurface embedded in $(M,\overline g)$, with induced metric $g:=\iota^*\overline g$. We will use latin letters for coordinates on $\Sigma$, and greek letters for those on $M$. We call $\nu$, $A=\{h_{ij}\}$ and $H$, respectively, the unit normal vector, the second fundamental form and the mean curvature of $\Sigma$. Moreover, $\kappa_1$ and $\kappa_2$ will be the principal curvatures. When needed, we will set $\Sigma^\hp:=\vec x(\Sigma)$, which is a surface embedded in the hyperbolic space. Moreover, $A^\hp$, $H^\hp$, etc. refers to a surface $\Sigma^\hp$ in the hyperbolic space $\mathbb H^3$. \\
\indent We will use the notation $e_r=\partial_r$ meaning the radial direction in $\overline g^{m}$, while $e_\theta$ and $e_\phi$ will be the angular directions. Moreover, in the following we will set  
\begin{equation}\label{perpeprep16}
|\nu^\top|^2:=\langle \nu,e_\theta \rangle^2+\langle \nu,e_\phi\rangle^2, \qquad |\nu^\bot|^2:=\langle \nu,\partial_r\rangle^2=1-|\nu^\top|^2,
\end{equation}
where here $|\cdot|=|\cdot|_{\overline g^{m}}$. The following lemma collects some well-known properties of the metric restricted to the surface $\Sigma$.
\begin{lem}[Lemma 3.1-3.2, \cite{nevestian}]\label{lemma14ii} Let $(M,\overline g)$ be as above, and let $\Sigma\hookrightarrow M$ be a surface. Set as above $r=|\vec x|$. Then 
\begin{enumerate}[label=\textup{(\roman*)}]
\item The Gaussian curvature of $\Sigma$ is given by
\begin{equation}\label{gaussianid24}
\kappa_g=-1+\frac{H^2}{4}+\frac{m}{\sinh^3 r}-\frac{3m|\nu^\top|^2}{2\sinh^3r}-\frac{|\overset{\circ}{A}|^2}2+O(e^{-5r});
\end{equation}
\item The Ricci curvature of $M$ restricted to $\Sigma$ in direction $\nu$ is given by
\begin{equation}\label{riccinunuonsurf}
\overline{\textup{Ric}}(\nu,\nu)=-2-\frac{m}{\sinh^3r}+\frac{3m|\nu^\top|^2}{2\sinh^3r}+O(e^{-5r}).
\end{equation}
\end{enumerate}
\end{lem}
\begin{proof}
We first prove \eqref{riccinunuonsurf}. In the following, we will write $e_\phi$ and $e_\theta$ for the normalization of $\partial_\phi$ and $\partial_\theta$. In these coordinates
\begin{equation*}
\begin{aligned}
&\overline{\textup{Ric}}(e_\theta,e_\theta)=-2+\frac{m}{2\sinh^3r}+O(e^{-5r}), \\
&\overline{\textup{Ric}}(e_\phi,e_\phi)=-2+\frac{m}{2\sinh^3r}+O(e^{-5r}), \\
&\overline{\textup{Ric}}(e_r,e_r)=-2-\frac{m}{\sinh^3r}+O(e^{-5r}),
\end{aligned}
\end{equation*}
\begin{equation}\label{Riccimisto27}
\overline{\textup{Ric}}(e_\phi,e_\theta)=O(e^{-5r}), \qquad \overline{\textup{Ric}}(e_r,e_\phi)=O(e^{-5r}), \qquad \overline{\textup{Ric}}(e_r,e_\theta)=O(e^{-5r}).
\end{equation}
as proved in \cite[Lemma 3.1]{nevestian}. Writing $\nu=\langle e_\theta,\nu\rangle e_\theta+\langle e_\phi,\nu\rangle e_\phi+\langle e_r,\nu\rangle e_r$, since \eqref{perpeprep16} holds with an error $O(e^{-5r})$ when $|\cdot|$ is replaced with $|\cdot|_{\overline g}$, we find
\begin{equation*}
\begin{aligned}
\overline{\textup{Ric}}(\nu,\nu)=&\left(\langle e_\theta,\nu\rangle^2+\langle e_\phi,\nu\rangle^2\right)\left(-2+\frac{m}{2\sinh^3r}+O(e^{-5r})\right)\\
&+|\nu^\bot|^2\left(-2-\frac{m}{\sinh^3r}+O(e^{-5r})\right)\\
=&-2-\frac{m}{\sin^3r}+\frac{3m|\nu^\top|^2}{2\sinh^3r}+O(e^{-5r}).
\end{aligned}
\end{equation*}
Equation \eqref{gaussianid24} follows combining \eqref{riccinunuonsurf} with the Gauss equation.
\end{proof}
\begin{rem}
In the following, we will consider round surfaces. In those cases, $\nu$ will be close to $\partial_r$, and thus $|\nu^\top|$ will be small.
\end{rem}
\begin{lem}\label{ric(nu)38} Let $\Sigma\hookrightarrow M$ be a surface such that for some $C>0$ and $\sigma>1$
\begin{equation*}
\frac{\sigma}{C}\leq \sinh r\leq C\sigma, \qquad |\nu-\partial_r|\leq C\sigma^{-2}.
\end{equation*}
Then, there exists $C'>0$ and $\sigma_0>1$ such that, if $\sigma>\sigma_0$, 
\begin{equation}\label{supRicnuX29}
\sup_{|X|=1} \left|\overline{\textup{Ric}}(\nu,X)\right|\leq C'\sigma^{-5},
\end{equation}
where $X$ belongs to $T\Sigma$.
\end{lem}
\begin{proof}
Note that, by definition, $\overline g(\nu,X)=0$ and thus, writing $\nu$ and $X$ in the orthgonal basis $\{e_r,e_\phi,e_\theta\}$,
\begin{equation*}
\langle X,e_r\rangle\langle \nu,e_r\rangle=-\langle X,e_\phi\rangle\langle \nu,e_\phi\rangle-\langle X,e_\theta\rangle\langle \nu,e_\theta\rangle,
\end{equation*}
where $\langle\cdot,\cdot\rangle=\langle\cdot,\cdot\rangle_{\overline g}$. Thus, combining the fact that $|\langle X,e_\alpha\rangle|\leq C$ together with \eqref{Riccimisto27},
\begin{equation*}
\begin{aligned}
\overline{\textup{Ric}}(\nu,X)=\ & \langle X,e_r\rangle\langle \nu,e_r\rangle \overline{\textup{Ric}}(e_r,e_r)+\langle X,e_\phi\rangle\langle \nu,e_\phi\rangle \overline{\textup{Ric}}(e_\phi,e_\phi)\\
&+\langle X,e_\theta\rangle\langle \nu,e_\theta\rangle \overline{\textup{Ric}}(e_\theta,e_\theta)+O(\sigma^{-5})\\
=\ & -\frac{2m}{\sinh^3 r} \langle X,e_r\rangle\langle \nu,e_r\rangle+O(\sigma^{-5})=O(\sigma^{-5}),
\end{aligned}
\end{equation*}
where we used that $|\langle X,e_r\rangle|=|\langle X,e_r-\nu\rangle|\leq C\sigma^{-2}$, since $e_r=\partial_r$, and $\sinh r\geq \frac{\sigma}{C}$.
\end{proof}
\begin{rem}\label{nuTOP}
Note that, if again $|\nu-\partial_r|$ is small, then also $|\nu^\top-\partial_r^\top|$ is small, choosing an orthonormal basis on $\Sigma$.
\end{rem}
In the following, we will indicate by $C>0$ every universal constant, only depending on the ambient, possibly changing from line to line.
\begin{lem}\label{asymptoticofspheres} Let $(M,\overline g)$ be as above, and consider the coordinate sphere $\mathbb S_{\hat r}:=\{r=\hat r\}\hookrightarrow M$. Then 
\begin{equation*}
|H-h|\leq Ce^{-5\hat r}, \qquad |\nabla H|\leq Ce^{-5\hat r}, \qquad \left|\overset{\circ}{A}\right|\leq C e^{-5\hat r}.
\end{equation*}
\end{lem}
\begin{proof}
Since \eqref{eq12} holds on the reference metric $\overline g_m$, by the asymptotics we find 
\begin{equation*}
\left|H-H_{S_{\hat r}}^m\right|\leq Ce^{-5\hat r},
\end{equation*}
where $H_{S_{\hat r}}^m$ is a constant. Thus, 
\begin{equation*}
\begin{aligned}
|H-h|\leq |H-H_{S_{\hat r}}^m|+|h-H_{S_{\hat r}}^m|&\leq |H-H_{S_{\hat r}}^m|+\fint_\Sigma |H-H_{S_{\hat r}}^m| \ d\mu\\
&\leq Ce^{-5\hat r}.
\end{aligned}
\end{equation*}
By the asymptotics and equation \eqref{eq12} we also have 
\begin{equation*}
|\nabla H|\leq Ce^{-5\hat r}.
\end{equation*}
The decay of $\overset{\circ}{A}$ follows similarly.
\end{proof}
\begin{rem}\label{remarkhyp16}
Note that, if $\Sigma$ is such that $\inf_\Sigma r\geq \frac{\hat r}{2}$, then there exists $C>0$ such that 
\begin{equation*}
\left|A-A^\hp\right|\leq Ce^{-3\hat r}.
\end{equation*}
This is clear writing the second fundamental forms of the two warped metrics and taking the difference. An immediate consequence is that the same estimate holds for the mean curvature and the traceless second fundamental form.
\end{rem}
Inspired by the remark on Euclidean spheres of Section \ref{maindefn11}, we give the following definition.
\begin{dfn} For every $\Sigma\hookrightarrow M$, we set 
\begin{equation*}
r_\Sigma:=\min_\Sigma r, \qquad R_\Sigma:=\max_\Sigma r, \qquad \sigma_\Sigma:=\sqrt{\frac{|\Sigma|}{4\pi}}.
\end{equation*}
\end{dfn}
\subsection{Neves-Tian's foliation}\label{thmnevestian} In their work \cite{nevestian}, Neves-Tian proved existence and uniqueness of CMC surfaces in the anti-de Sitter–Schwarzschild metric with positive mass and balanced coordinates, see in particular \cite[Sect. 8]{nevestian}. In particular, they showed the existence of a family of CMC-surfaces, say $\left\{\Sigma^\sigma\right\}_{\sigma\geq\sigma_0}$, for some $\sigma_0>1$ and $\sigma=\sqrt{\frac{|\Sigma^\sigma|}{4\pi}}$ such that, for each $\sigma$ there exists $\hat r=\hat r(\sigma)>0$ such that 
\begin{enumerate}[label=(\roman*)]
\item $\Sigma^\sigma$ is a CMC-surface;
\item $|r-\hat r|\leq Ce^{-R_{\Sigma^\sigma}}$ for some universal $C>0$;
\item $\left\|\overset{\circ}{A}\right\|_{L^\infty(\Sigma^\sigma)}\leq Ce^{-2\hat r}$.
\end{enumerate}
Moreover, they showed that it is possible to choose $\hat r$ so that the mean curvature $H_{\Sigma^\sigma}$ of $\Sigma^\sigma$ satisfies 
\begin{equation*}
H_{\Sigma^\sigma}=2\frac{\cosh \hat r}{\sinh \hat r}-\frac{m}{\sinh^3 \hat r}+O(e^{-4\hat r}),
\end{equation*}
and, in this case, the graph function $w_{\hat r}:=r-\hat r$ satisfies 
\begin{equation}\label{graphfs1120}
\|w_{\hat r}\|_{C^{2,\alpha}}\leq Ce^{-R_{\Sigma^\sigma}},
\end{equation}
for some $C>0$. Combining this with the fact that, since $\left|\overline {\textup{g}}-\overline{\textup g}_{\mathbb H^3}\right|=O(e^{-3r_{\Sigma^\sigma}})$,
\begin{equation*}
4\pi\sigma^2=|\Sigma^\sigma|=|\Sigma^\sigma|_{\mathbb H^3}+O(e^{-r_{\Sigma^\sigma}}).
\end{equation*}
On the other hand, for graphs on coordinate spheres the area is given by
\begin{equation*}
|\Sigma^\sigma|_{\mathbb H^3}=4\pi \sinh^2 \hat r+O(e^{r_{\Sigma^\sigma}}),
\end{equation*}
because of \eqref{graphfs1120} and also $\hat r\geq r_{\Sigma^\sigma}-C$. Thus $\left|\sinh \hat r-\sigma\right|\leq C$. Furthermore, 
\begin{equation*}
|\sinh r-\sinh \hat r|\leq |\cosh \xi||r-\hat r|\leq C,
\end{equation*}
since $\xi=r+\theta (r_s-r)\leq R_{\Sigma^\sigma}+C$ for some $\theta\in[0,1]$. As a consequence 
\begin{equation*}
H_{\Sigma^\sigma}=2\frac{\cosh \hat r}{\sinh \hat r}-\frac{m}{\sigma^3}+O(\sigma^{-4}).
\end{equation*}
In the following, we will work, except for the final Section \ref{conclusion5}, with a single leaf of the foliation constructed by Neves-Tian. For this reason, we will simply write $\Sigma$ instead of $\Sigma^\sigma$, with $\sigma$ fixed and equal to the area radius of $\Sigma$. 
\begin{rem}[On a notation from Section \ref{theflow}] From Section \textup{\ref{theflow}} on, this CMC surface will be the initial datum of a curvature flow which we will show exists for every positive time and converges. For this reason, we will refer to the approximate radius of this selected CMC surface as $\hat r_0$, in order to distinguish it from the radius of an arbitrary round surface, which will be indicated with $\hat r$.
\end{rem}
\subsection{Spacetime curvatures} We conclude this introduction by defining the spacetime mean curvature and analyzing some preliminary properties. See \cite{cederbaumsakovich} for equivalent definitions in the Euclidean setting. 
\begin{dfn} Let $(M,\overline g,\overline K)$ be an asympotically hyperboloidal initial data set. Let $\Sigma\hookrightarrow M$. Then we define 
\begin{equation*}
P:=\textup{tr}_g \overline K,
\end{equation*}
and define the \textup{spacetime mean curvature} as
\begin{equation*}
\mathcal H:=\sqrt{H^2-P^2},
\end{equation*}
where $H$ is again the mean curvature of $\Sigma$.
\end{dfn}
With the assumption of Definition \ref{asymptoticallyhypK} in mind, we have the following immediate consequence.
\begin{lem}\label{asymptP} Let $(M,\overline g,\overline K)$ be a $C_5^2$-asymptotically hyberboloidal initial data set, and let $\Sigma\hookrightarrow M$ be a closed surface. If $r_\Sigma \geq \frac{\hat r}2$, then, setting $\overline P:=\overline K-\overline g$,
\begin{equation*}
P=g^{ij}\overline K_{ij}=g^{ij}\left(\overline g_{ij}+\overline P_{ij}\right)=2+\textup{tr}_g(\overline P)=2+O\left(e^{-5\hat r}\right),
\end{equation*}
since $\left|\textup{tr}_g\left(\overline P\right)\right|\leq |\overline P|\leq \overline ce^{-5\hat r}$.
\end{lem}
\begin{lem}\label{asymptnablaP}
Let $(M,\overline g,\overline K)$ be a $C_5^2$-asymptotically hypberoloidal initial data set, and let $\Sigma\hookrightarrow M$ be a surface with $|A|\leq C$, for some universal $C>0$. If $r_\Sigma \geq \frac{\hat r}2$, then 
\begin{equation*}
|\nabla P|\leq Ce^{-5\hat r},
\end{equation*}
for some $C>0$ only depending on the ambient.
\end{lem}
\begin{proof} Using normal coordinates on $\Sigma$, we find 
\begin{equation*}
\nabla_k P=\nabla_k\left(2+\textup{tr}_g(\overline P)\right)=\nabla_k\left(g^{ij}\overline P_{ij}\right)=g^{ij}\nabla_k\overline P_{ij}=g^{ij}\left(\overline\nabla_k \overline P_{ij}-A*\overline P(\nu,\cdot)\right),
\end{equation*}
and thus the conclusion follows from the asymptotic assumptions.
\end{proof}
\begin{lem}\label{lemma213} Let $(M,\overline g,\overline K)$ be an asymptotically hyperboloidal initial data set and consider a \textup{CMC-surface} $\Sigma\hookrightarrow M$ such that $H^2\geq P^2$, $h\geq 2$ and there exist $\hat r, c>0$ such that 
\begin{enumerate}[label=(\roman*)]
\item $\frac12 \hat r\leq r_\Sigma\leq R_\Sigma\leq \frac32 \hat r$;
\item $\left|\sigma_\Sigma-\sinh(\hat r)\right|\leq c$;
\item $\mathcal H\geq \frac{3}{2\sinh \hat r}$.
\end{enumerate}
Then there exists $\sigma_0=\sigma_0(c,\overline c)$ and $C>0$ such that 
\begin{equation*}
|\mathcal H-\hbar|\leq C\sigma^{-4}, \qquad |\nabla\mathcal H|\leq C\sigma^{-4}.
\end{equation*}
\end{lem}
\begin{proof} Notice that 
\begin{equation*}
\left|\mathcal H-\sqrt{h^2-4}\right|\equiv \left|\sqrt{h^2-P^2}-\sqrt{h^2-4}\right|\leq \frac{|P^2-4|}{\mathcal H+\sqrt{h^2-4}}\leq C\sigma^{-4}.
\end{equation*}
Thus, $|\mathcal H-\hbar|\leq C\sigma^{-4}$. Moreover, taking the derivative of $\mathcal H^2=H^2-P^2$, 
\begin{equation*}
\mathcal H\nabla\mathcal H=H\nabla H-P\nabla P=-P\nabla P.
\end{equation*}
Since $\mathcal H\geq\frac{3}{2\sinh \hat r}$, the conclusion follows from Lemma \ref{asymptnablaP}. 
\end{proof}
\subsection{Almost CMC-surfaces} For the rest of the Section, we will consider surfaces which have constant mean curvature, modulo an error. 
\begin{dfn}\label{almostCMC215} Let $(M,\overline g,\overline K)$ be an asymptotically hyperboloid initial data set. We say that $\Sigma\hookrightarrow M$ is an \textit{almost CMC}-surface if there exist $\hat r$, $C>0$ and $\delta>0$ such that 
\begin{enumerate}[label=(\roman*)]
\item $\frac12 \hat r\leq r_\Sigma\leq R_\Sigma\leq \frac32 \hat r$;
\item $\left|\sigma_\Sigma-\sinh(\hat r)\right|\leq C$;
\item $\|H-h\|_{L^\infty(\Sigma)}\leq C\sigma^{-3-\delta}$;
\item the following estimate holds
\begin{equation}\label{hclosedness}
\left|h-2\sqrt{1+\frac1{\sinh^2 \hat r}}+\frac{m}{\sinh^3 \hat r}\right|\leq C\sigma^{-3-\delta}.
\end{equation}
\end{enumerate}
\end{dfn}
\begin{rem}
As remarked in \textup{Section \ref{thmnevestian}}, the CMC-surfaces constructed by Neves-Tian in \cite{nevestian} are clearly in this class.
\end{rem}
We now analyze the behavior of these surfaces from a "spacetime" point of view, i.e. we compute the spacetime curvature of the almost-CMC surfaces and how this quantity behaves with respect to its integral mean.
\begin{lem}\label{lem111} Let $(M,\overline g,\overline K)$ be an asymptotically hyperboloidal initial data set, and consider an almost CMC-surface $\Sigma\hookrightarrow M$ in the sense of Definition \textup{\ref{almostCMC215}}. Then, for every $\epsilon>0$ there exists $\sigma_0=\sigma_0(\epsilon,c)>1$ such that, if $\sigma>\sigma_0$,
\begin{equation*}\frac{2-\epsilon}{\sinh \hat r}\leq \mathcal H\leq \frac{2+\epsilon}{\sinh \hat r}.
\end{equation*}
\end{lem}
\begin{proof} Since $P=2+O\left(e^{-5\hat r}\right)$, we find
\begin{equation*}
\mathcal H^2=H^2-P^2=h^2-4+O(\sigma^{-3-\delta})=\frac{4}{\sinh^2 \hat r}+O(\sigma^{-2-\delta}),
\end{equation*}
for some $\delta>0$, and thus the conclusion follows.
\end{proof}
\begin{rem}\label{usefulremark} Thanks to \ref{lem111}, hypothesis (iii) in Lemma \ref{lemma213} can be replaced by \eqref{hclosedness}.
\end{rem}
\begin{cor}\label{corlemma213} Let $(M,\overline g,\overline K)$ be an asymptotically hyperboloidal initial data set and consider a \textup{CMC-surface} $\Sigma\hookrightarrow M$ such that there exist $\hat r, c$ and $\alpha>0$ such that 
\begin{enumerate}[label=\textup{(\roman*)}]
\item $\frac12 \hat r\leq r_\Sigma\leq R_\Sigma\leq \frac32 \hat r$;
\item $\left|\sigma_\Sigma-\sinh(\hat r)\right|\leq c$;
\item \eqref{hclosedness} holds.
\end{enumerate}
Then there exists $\sigma_0=\sigma_0(c,\overline c)$ and $C>0$ such that 
\begin{equation*}
|\mathcal H-\hbar|\leq C\sigma^{-4}, \qquad |\nabla\mathcal H|\leq C\sigma^{-4}.
\end{equation*}
\end{cor}
We conclude the Section with the following Lemma. This shows that the oscillation of the mean curvature has a better decay then the spacetime one.
\begin{lem}\label{improvementH220} Let $(M,\overline g,\overline K)$ be an asymptotically hyperboloidal initial data set and consider a surface $\Sigma\hookrightarrow M$ such that there exist $\hat r$, $C$, $B>0$ such that 
\begin{enumerate}[label=\textup{(\roman*)}]
\item $\frac12 \hat r\leq r_\Sigma\leq R_\Sigma\leq \frac32 \hat r$;
\item $\left|\sigma_\Sigma-\sinh(\hat r)\right|\leq B$;
\item $\|\mathcal H-\hbar\|_{L^\infty(\Sigma)}\leq B\sigma^{-\beta}$ for some $\beta\leq 4$;
\item $H>0$ and $\mathcal H\leq \frac{3}{\sinh \hat r}$.
\end{enumerate}
Then there exists $c=c(B,C)>0$ such that 
\begin{equation*}
\|H-h\|_{L^\infty(\Sigma)}\leq c\sigma^{-\beta-1}.
\end{equation*}
If moreover $H\geq 1$ and $\|\nabla\mathcal H\|_{L^\infty(\Sigma)}\leq B\sigma^{-\beta}$ then
\begin{equation}\label{improvgrad219}
\|\nabla H\|_{L^\infty(\Sigma)}\leq c\sigma^{-\beta-1}.
\end{equation}
\end{lem}
\begin{proof}
The hypotheses imply that $\left|\mathcal H^2-\hbar^2\right|\leq c\sigma^{-\beta-1}$. Thus, by definition of $P$,
\begin{equation*}
\left|H^2-(4+\hbar^2)+O(\sigma^{-5})\right|\leq c\sigma^{-\beta-1}.
\end{equation*}
Setting $\gamma:=\sqrt{4+\hbar^2}\geq 2$, we find $|H-\gamma|\leq c\sigma^{-\beta-1}$, since $H>0$. Thus the first conclusion holds. Suppose now that $H\geq 1$. We obtain \eqref{improvgrad219} by
\begin{equation*}
\left|\nabla H\right|=H^{-1}\left|\mathcal H\nabla \mathcal H+P\nabla P\right|\leq C\left(\sigma^{-1}|\nabla \mathcal H|+\sigma^{-5}\right)\leq c\sigma^{-\beta-1}.
\end{equation*}
\end{proof}
\begin{lem}\label{lemma33} Let $\Sigma\hookrightarrow M$ be such that (i) and (ii) in Lemma \ref{improvementH220} hold. Then
\begin{equation}\label{ineqA+Ric}
|A|^2+\overline{\textup{Ric}}(\nu,\nu)=\frac{\mathcal H^2}{2}+|\overset{\circ}{A}|^2+O(\sigma^{-3}).
\end{equation}
\end{lem}
\begin{proof}
By definition, we find 
\begin{equation}\label{eq36}
|A|^2+\overline{\textup{Ric}}(\nu,\nu)=\frac{H^2}2+\overline{\textup{Ric}}(\nu,\nu)+|\overset{\circ}{A}|^2=\frac{\mathcal H^2}{2}+\frac{P^2}{2}+\overline{\textup{Ric}}(\nu,\nu)+|\overset{\circ}A|^2,
\end{equation}
and since $P=2+g^{ij}\overline P=2+\textup O(\sigma^{-5})$ and $\overline{\textup{Ric}}(\nu,\nu)=-2+\textup O(\sigma^{-3})$, it follows
\begin{equation}\label{estimate|A|+Ric}
\left|\frac{P^2}{2}+\overline{\textup{Ric}}(\nu,\nu)+|\overset{\circ}A|^2\right|\leq |\overset{\circ}{A}|^2+C\sigma^{-3}.
\end{equation}
\end{proof}
\section{Round surfaces}\label{sect3round}
We start defining the notion of roundeness we will use in the rest of the paper. This definition involves the curvature estimates obtained in the previous Section together with some bounds on the integral norms of the main geometric quantities of a surface $\Sigma$. For this reason, we start this Section reviewing some well-known inequalities on spheres and surfaces.
\subsection{Sobolev norms and Sobolev inequalities}
In their paper \cite{hoffmanspruck}, Hoffman and Spruck obtained a generalization of the well-known Simon-Sobolev inequality. In particular, they remarked that, in the case of the hyperbolic space $\mathbb H^3$, the Simon-Sobolev inequality continues to hold, i.e.  
\begin{equation*}
\left(\int_\Sigma \psi^2 \ d\mu\right)^\frac12\leq C\left(\int_\Sigma |\nabla\psi|+H|\psi| \ d\mu\right),
\end{equation*}
for every $\Sigma$ with $H\geq0$ and every $\psi\in W^{1,1}(\Sigma)$. If one replaces $\Sigma$ with the family of coordinate spheres $\left\{\mathbb S_\sigma(\vec 0)\right\}_\sigma$, which have uniformly bounded mean curvature, the inequality takes the form 
\begin{equation}\label{notoptineq}
\left(\int_{\mathbb S_\sigma} \psi^2 \ d\mu\right)^\frac12\leq C\left(\int_{\mathbb S_\sigma} |\nabla\psi|+|\psi| \ d\mu\right),
\end{equation}
for some universal constant $C>0$. This clearly holds also for an arbitrary family of surfaces with uniformly bounded mean curvature.\\
\indent It was then observed, see in example \cite{nerzhyp}, that inequality \eqref{notoptineq} is not optimal, due to the scaling of the $L^1$ and $L^2$ norm associated to  the induced metric. In particular, setting $\tilde g=a g$ for some $a>0$, i.e. rescaling the metric, one finds
\begin{equation}\label{nerzrescal}
\|\psi\|_{L^p(\tilde\Sigma)}=\left(\frac{|\tilde \Sigma|}{|\Sigma|}\right)^\frac1p\|\psi\|_{L^p(\Sigma)}, \qquad \|\tilde \nabla \psi\|_{L^p(\tilde \Sigma)}=\left(\frac{|\tilde \Sigma|}{|\Sigma|}\right)^{\frac1p-\frac12}\|\nabla \psi\|_{L^p(\Sigma)},
\end{equation}
where in the case of coordinate spheres (i.e. simply choosing $g$ equal to the round metric $g_0$ on $\mathbb S^2$) the ratio $\frac{|\tilde \Sigma|}{|\Sigma|}$ is proportional to $a^2$. It is then clear that the inequality is correctly scaled in the case of coordinate spheres by 
\begin{equation}\label{optineqspheres}
\left(\int_{\mathbb S_\sigma} \psi^2 \ d\mu\right)^\frac12\leq C\left(\int_{\mathbb S_\sigma}|\nabla\psi|+{(\sinh\sigma)}^{-1}|\psi| \ d\mu\right),
\end{equation}
keeping in mind that $|\mathbb S_\sigma|_{\hp}=4\pi\sinh^2 \sigma$.\\
\indent As we will explain in the following, we will work with surfaces which are close to coordinate spheres. Thus a generalized version of \eqref{optineqspheres} will hold. With this in mind, we give the following definitions.
\begin{dfn} Let $\Sigma\hookrightarrow M$ be a surface. For every $p\in [1,\infty]$, we set 
\begin{equation*}
\|\psi\|_{W^{1,p}(\Sigma)}=\|\psi\|_{L^p(\Sigma)}+\sigma_\Sigma \|\nabla\psi\|_{L^p(\Sigma)},
\end{equation*}
for every $\psi\in W^{1,p}(\Sigma)$.
\end{dfn}
\begin{lem}\label{lemmagraphsob} Let $\Sigma$ be a surface in $\mathbb H^3$ such that there exists $\hat r\geq 1$ and a function $f:\mathbb S_{\hat r}(0)\to \R$ such that $\Sigma=\textup{graph}(f,\mathbb S_{\hat r})$ and $|f|_{C^{2,\alpha}}\leq B\sigma^{-1}$, for some $B>0$ and $\alpha\in(0,1)$. Then there exist $C_S>0$ and $\vartheta>0$, only depending on universal constants, and a radius $\sigma_0=\sigma_0(C_S,\vartheta,B,\alpha)$ such that, if $\sigma>\sigma_0$,
\begin{equation}\label{sob+nondeg}
\|\psi\|_{L^2(\Sigma)}\leq \frac{C_S}{\sigma_\Sigma}\|\psi\|_{W^{1,1}(\Sigma)}, \qquad \int_{B_g(x_0,\epsilon)} \ d\mu_g\geq \vartheta \epsilon^2,
\end{equation}
for every $\psi\in W^{1,1}(\Sigma)$ and $0<\epsilon<\textup{diam}(\Sigma)$, where $B_g(x_0,\epsilon)=\{y\in \Sigma: \ d_g(x_0,y)<\epsilon\}$, with $g$ the induced metric on $\Sigma$ and $d_g$ the distance induced by the metric.
\end{lem}
\begin{rem} In turn, the difference $g-g_\hp$ is $O(\sigma^{-3})$, and the metric $g_\hp$ in the previous Lemma can be replaced by $g$.
\end{rem}
\begin{proof}
Since $\Sigma$ is $C^{2,\alpha}$-close to the sphere $\mathbb S_{\hat r}(\boldsymbol 0)$, we have that $g$ is $C^{1,\alpha}$ close to the round metric on $\mathbb S_{\hat r}(\boldsymbol 0)$, i.e. $\left(\sinh^2 \hat r\right) g_0$. For the round metric, we already remarked that the Sobolev inequality in \eqref{sob+nondeg} holds. On the other hand, with respect to the round metric $g_{\hat r}$ on the sphere $\mathbb S_{\hat r}(\boldsymbol 0)$, we can write
\begin{equation*}
\begin{aligned}
\int_{B_{g_{\hat r}}(x_0,\epsilon)} \ d\mu_{g_{\hat r}}&=\sinh^2 \hat r \int_0^{2\pi}\int_0^{\frac{\epsilon}{\sinh \hat r}} \sin\alpha \ d\alpha \ d\theta\\
&=2\pi\sinh^2 \hat r\left[1-\cos \left(\frac{\epsilon}{\sinh \hat r}\right)\right]\\
&\geq \frac{2\pi}{5}\sinh^2 \hat r\left(\frac{\epsilon}{\sinh \hat r}\right)^2=\frac{2\pi}{5}\epsilon^2,
\end{aligned}
\end{equation*}
for $\epsilon\leq\pi\sinh \hat r=\textup{diam}\left(\mathbb S_{\hat r}(\boldsymbol 0)\right)$. Since now $|g-g_{\hat r}|=O(\sigma^{-1})$, the Sobolev inequality easily generalizes to $\Sigma$. Moreover, by the decay for $f$ and its derivative we have that small geodesic segments on the sphere, when projected to $\Sigma$, have the same length with an error of order $O\left(\frac{\epsilon}{\sinh^2 \hat r}\right)$, as one can see computing the length of $\sinh \hat r\nu_{\gamma(s)}+f(\gamma(s))\nu_{\gamma(s)}$, when $\gamma$ is parametrized by arc-length. This concludes the proof of the Lemma.
\end{proof}
We conclude this subsection with a well-known consequence of Sobolev's inequality in \eqref{sob+nondeg}. The proof of \eqref{SobLp27} and \eqref{SobLinfty28} is standard, see for example \cite{nerzthesis}.
\begin{lem}
Let $\Sigma\hookrightarrow M$ be a surface such that 
\begin{equation*}
\|\psi\|_{L^2(\Sigma)}\leq \frac{C_S}{\sigma_\Sigma}\|\psi\|_{W^{1,1}(\Sigma)}
\end{equation*}
for every $\psi\in W^{1,1}(\Sigma)$. Then, for every $p>2$, there exists a possibly different (universal) $C_S>0$ such that 
\begin{equation}\label{SobLp27}
\|\psi\|_{L^p(\Sigma)}\leq \frac{C_S}{2}p\sigma^{\frac2p-1}\|\psi\|_{W^{1,2}(\Sigma)}, \qquad  \forall \psi\in W^{1,2}(\Sigma),
\end{equation}
\begin{equation}\label{SobLinfty28}
\quad \quad \|\psi\|_{L^\infty(\Sigma)}\leq 2^{\frac{2(p-1)}{p-2}}C_S\sigma^{-\frac2p}\|\psi\|_{W^{1,p}(\Sigma)}, \qquad \forall \psi \in W^{1,p}(\Sigma).
\end{equation}
\end{lem}
\subsection{Definition of round surfaces} Before defining the roundness of a surface, we briefly review the definition of Hölder norm.
\begin{dfn}\label{dfn35holder} Let $\Sigma\hookrightarrow M$ be a surface of area $|\Sigma|_g=4\pi\sigma^2$. We then set
\begin{equation}\label{holdernorm}
\|\psi\|_{C^{k,\alpha}(\Sigma)}:=\max_{|j|\leq k}\sigma^{|j|}\|\nabla^j\psi\|_{L^\infty(\Sigma)}+\max_{|j|=k}\sigma^{k+\alpha}\sup_{x\neq y}\frac{|\nabla^j\psi(x)-\nabla^j\psi(y)|}{d_g(x,y)^\alpha}
\end{equation}
\end{dfn}
Setting $\hat \Sigma:=\sigma^{-1}\Sigma$, we remark that 
\begin{equation}\label{holderscaling}
|\nabla^j\psi|=\sigma^{-|j|}|\hat \nabla^j \psi|, \qquad d_g=\sigma d_{\hat g}.
\end{equation}
Thus, \eqref{holdernorm} is well-posed with respect to scaling the surface.
\begin{dfn}\label{dfnroundenss36}
Let $(\M,\overline g,\overline K)$ be an asymptotically hyperbolic initial data set and let $\iota:\Sigma\to M$ be a surface. Fix an area radius $\sigma>1$, $\hat r_0>1$ and parameters $B_1$, $B_2$, $B_3>0$. We say that $(\Sigma,g)$ is a \textup{round surface of approximate radius $\hat r$} in $(M,\overline g,\overline  K)$, and we write quantitatively $\Sigma\in \mathcal{B}_{\sigma,\hat r}(B_1,B_2,B_3)$, if there exists $\hat r\in [r_\Sigma,R_\Sigma]$ such that 
\begin{equation}\label{eq312}
|\sinh \hat r-\sigma|<B_1,\qquad |r-\hat r|<B_1\sigma^{-1}, \qquad \frac72\pi\sigma^2<|\Sigma|_g<5\pi\sigma^2,
\end{equation}
\begin{equation}\label{eq313}
\|\mathcal H-\hbar\|_{W^{1,4}(\Sigma)}<B_2 \sigma^{-\frac52},
\end{equation}
\begin{equation}\label{eq314b}
\|\overset{\circ}{A}\|_{L^\infty(\Sigma)}<B_3\sigma^{-2}, \qquad \|w_{\hat r}\|_{C^{2,\frac12}(\Sigma)}<B_3\sigma^{-1},
\end{equation}
where $w_{\hat r}:=r-\hat r$ restricted to $\Sigma$. In the following, whenever it will be clear from the context, we will simply write $w$ instead of $w_{\hat r}$.\\
\indent If \eqref{eq312}, \eqref{eq313} or \eqref{eq314b} holds with one $<$ replaced by $\leq$, then we write $\Sigma\in\overline{\mathcal B}_{\sigma,\hat r}(B_1,B_2,B_3)$.
\end{dfn}
\begin{rem}\label{remark37i-ii} \textup{(i)} We remark that the CMC-surfaces constructed in \textup{Section \ref{thmnevestian}} belong to this class, see the computations therein.\\
\indent \textup{(ii)} This class is more specified than the roundness class defined in \cite{vpmcf} or \cite{vpstmcf}. This is because we need here to take into account which approximate sphere we are considering in order to say that $\Sigma$ is round. To relax this condition, we simply say that $\Sigma$ is round, i.e. $\Sigma\in \mathcal B_\sigma(B_1,B_2,B_3)$ if 
\begin{equation*}
\begin{aligned}
&|\sinh r-\sigma|<B_1, \qquad \frac72\pi\sigma^2<|\Sigma|_g<5\pi\sigma^2,\\
&\|\mathcal H-\hbar\|_{W^{1,4}(\Sigma)}<B_2 \sigma^{-\frac52}, \qquad \|\overset{\circ}{A}\|_{L^\infty(\Sigma)}<B_3\sigma^{-2}.
\end{aligned}
\end{equation*}
This allows to say that surfaces with different approximate radius are in the same class of roundness, say $\Sigma,\Sigma'\in \mathcal B_\sigma(B_1,B_2,B_3)$. Notice that in this case one finds $|\hat r-\hat r'|\leq c(B_1)\sigma^{-1}$.
\end{rem}
At the light of the scaling properties \eqref{holderscaling}, hypothesis \eqref{eq314b} says that
\begin{equation*}
\|\nabla w_{\hat r}\|_{L^\infty(\Sigma)}<B_3\sigma^{-2},
\end{equation*}
and thus, since $\textup{diam}(\Sigma)\sim \sigma$,
\begin{equation}\label{CalphaW1infty}
\|w_{\hat r}\|_{C^{0,\alpha}(\Sigma)}=\|w_{\hat r}\|_{L^\infty(\Sigma)}+\sigma^\alpha \sup_{x\neq y}\frac{|w_{\hat r}(x)-w_{\hat r}(y)|}{d(x,y)^\alpha}<B_3\sigma^{-1}.
\end{equation}
Note moreover that 
\begin{equation*}
\nabla w_{\hat r}=\nabla r=\partial_r^\top,
\end{equation*}
since $\partial_r^\top$ is the restriction to $T\Sigma$ of $\partial_r$. Similarly, setting $\partial_r=\langle \nu,\partial_r\rangle \nu+\nabla w_{\hat r}$, we find $1+O(e^{-5r})=\langle \nu,\partial_r\rangle^2+|\nabla w_{\hat r}|^2$ so that 
\begin{equation*}
\langle \nu,\partial_r\rangle=\sqrt{1-|\nabla w_{\hat r}|^2+O(e^{-5r})}=1-\frac{|\nabla w_{\hat r}|^2}{2}+O(e^{-5r}).
\end{equation*}
We conclude that 
\begin{equation}\label{stimapartialrnu}
|\partial_r-\nu|^2=2-2\langle \nu,\partial_r\rangle+O(e^{-5r})=|\nabla w_{\hat r}|^2+O(e^{-5r})=O(\sigma^{-4}).
\end{equation}
\subsection{Umbilicality of round surfaces} In this Section, we will require $\Sigma$ to satisfy some conditions which, ultimately, will be consequence of the definition of roundness. At first, we show that the smallness of $\|\nabla H\|_{L^2(\Sigma)}$ implies an $H^1$-estimate on $\overset{\circ}{A}$.
\begin{lem}\label{H1boot} Let $\Sigma\hookrightarrow M$ be a surface such that there exists $B>0$, a universal constant $C>0$ and a small constant $\varepsilon>0$ such that
\begin{equation*}
\frac12 \hat r\leq r_\Sigma\leq R_\Sigma\leq \frac32 \hat r, \qquad \frac\sigma2\leq \sinh \hat r\leq \frac32\sigma, \qquad \frac72\pi\sigma^2<|\Sigma|_g<5\pi\sigma^2,
\end{equation*}
\begin{equation*}
\mathcal H\geq \frac3{2\sigma}, \qquad |\overset{\circ}{A}|^2\leq \varepsilon \sigma^{-2}, \qquad \|\nabla H\|_{L^2(\Sigma)}\leq B\sigma^{-4}, \qquad \left\|\overline{\textup{Ric}}(\cdot,\nu)\right\|_{L^2(\Sigma)}\leq C\sigma^{-4}.
\end{equation*}
Then, there exist a small universal constant $\varepsilon_0>0$, a constant $c=c(B)>0$ and a large radius $\sigma_0=\sigma_0(B,\varepsilon_0)$ such that, if $\varepsilon<\varepsilon_0$ and $\sigma>\sigma_0$,
\begin{equation*}
\|\overset{\circ}{A}\|_{H^1(\Sigma)}\leq c(B)\sigma^{-3}.
\end{equation*}
\end{lem}
Note that the assumption on $\overline{\textup{Ric}}$ is satisfied by our setting provided that $\nu$ is close to $\partial_r$ and $r_\Sigma$ is suitably large, see Lemma \ref{ric(nu)38}.
\begin{proof} Consider the following (multiplied by $\overset{\circ}{A}$) version of Simons' identity
\begin{equation}\label{simonsbis}
\begin{aligned}
\langle \Delta \overset{\circ}{A},\overset{\circ}{A}\rangle = \ & \langle\textup{Hess}(H),\overset{\circ}{A}\rangle + H \langle A^2,\overset{\circ}{A}\rangle - |A|^2 \langle A,\overset{\circ}{A}\rangle \\
& + \left(h^{il} \overline{\textup{Rm}}\indices{_k^j_{kl}}+h^{lk}\overline{\textup{Rm}}\indices{_l^{ij}_k}\right)\overset{\circ}{A}_{ij}+ \left(\nabla^j\left(\overline{\textup{Ric}}\indices{^i_\varepsilon}\nu^\varepsilon\right)+\nabla_l \left(\overline{\textup{Rm}}\indices{_\varepsilon^{ij}_l}\nu^\varepsilon \right)\right)\overset{\circ}{A}_{ij},
\end{aligned}
\end{equation}
see \cite{metzger}. Using normal coordinates, 
\begin{equation*}
Hh_i^lh_{lj}-|A|^2h_{ij}=Hh_{il}\overset{\circ}{h}_{lj}-|\overset{\circ}{A}|^2h_{ij},
\end{equation*} 
and moreover
\begin{equation*}
\begin{aligned}
Hh_{il}\overset{\circ}{h}_{lj}\overset{\circ}{h}_{ij}-|\overset{\circ}{A}|^2h_{ij}\overset{\circ}{h}_{ij}&=H\textup{tr}\left(\overset{\circ}{A}^3\right)+\frac{H^2}{2}|\overset{\circ}{A}|^2-|\overset{\circ}{A}|^4\\
&=\frac{H^2}{2}|\overset{\circ}{A}|^2-|\overset{\circ}{A}|^4.
\end{aligned}
\end{equation*}
By the symmetries of the Riemannian tensor, see in example \cite[Lemma 3.8]{huiskenyau},
\begin{equation*}
\left(h^{il} \overline{\textup{Rm}}\indices{_k^j_{kl}}+h^{lk}\overline{\textup{Rm}}\indices{_l^{ij}_k}\right)\overset{\circ}{A}_{ij}=2\overline{\textup{Rm}}_{1212}|\overset{\circ}{A}|^2=|\overset{\circ}{A}|^2\left(\overline{\textup{Ric}}(\nu,\nu)+O(\sigma^{-3})\right),
\end{equation*}
where the latter identity is true because $\left|\overline{\textup{Rm}}-\overline{\textup{Rm}}_{\hp}\right|=O(\sigma^{-3})$ by the asymptotic of the metric and the identity $2\overline{\textup{Rm}}_{1212}^\hp=\overline{\textup{Ric}}^\hp(\nu,\nu)$, which holds in the hyperbolic space because of the constancy of the sectional curvature. Integrating by parts, again by the techniques of \cite{metzger}, we find
\begin{equation*}
\begin{aligned}
\|\nabla\overset{\circ}{A}\|_{L^2}^2&=\int_\Sigma |\nabla\overset{\circ}{A}|^2 \ d\mu=-\int_\Sigma \langle \Delta \overset{\circ}{A},\overset{\circ}{A}\rangle \ d\mu\\
&\leq-\int_\Sigma \left(\frac{H^2}{2}+\overline{\textup{Ric}}(\nu,\nu)\right)|\overset{\circ}{A}|^2 \ d\mu+\int_\Sigma \left(|\overset{\circ}{A}|^2+O(\sigma^{-3})\right)|\overset{\circ}{A}|^2 \ d\mu\\
& \quad +C\|\nabla H\|_{L^2(\Sigma)}\|\nabla\overset{\circ}{A}\|_{L^2(\Sigma)}+C\left\|\overline{\textup{Ric}}(\cdot,\nu)\right\|_{L^2(\Sigma)}\|\nabla\overset{\circ}{A}\|_{L^2(\Sigma)},
\end{aligned}
\end{equation*}
where the fact that the Riemann tensor in \eqref{simonsbis} is fully replaced by $\overline{\textup{Ric}}(\cdot,\nu)$ is a consequence of the fact that, in normal coordinates,
\begin{equation*}
\nabla_l\left(\overline{\textup{Rm}}_{\alpha iil}\nu^\alpha\right)=-\nabla_l\left(\overline{\textup{Rm}}_{li\alpha i}\nu^\alpha\right)=\nabla_l\left(\left(-\overline{\textup{Ric}}_{l\alpha}+\overline{\textup{Rm}}_{l\nu \alpha \nu}\right)\nu^\alpha\right)=-\nabla_l\left(\overline{\textup{Ric}}_{l\alpha}\nu^\alpha\right).
\end{equation*}
By the estimate \eqref{estimate|A|+Ric} and the lower bound for $\frac{\mathcal H^2}{2}$, 
\begin{equation*}
\frac{H^2}{2}+\overline{\textup{Ric}}(\nu,\nu)=\frac{\mathcal H^2}{2}+O(\sigma^{-3})\geq \frac{C}{\sigma^2}.
\end{equation*}
Thus, the Young's inequality, the estimate $|\overset{\circ}{A}|^2\leq \varepsilon\sigma^{-2}$, together with $\|\nabla H\|_{L^2(\Sigma)}\leq B\sigma^{-4}$ and $\left\|\overline{\textup{Ric}}(\cdot,\nu)\right\|_{L^2(\Sigma)}\leq C\sigma^{-4}$, imply
\begin{equation*}
\frac{C}{\sigma^2}\|\overset{\circ}{A}\|_{L^2(\Sigma)}^2+\|\nabla\overset{\circ}{A}\|_{L^2(\Sigma)}^2\leq c(B)\sigma^{-8}.
\end{equation*}
This means that $\|\overset{\circ}{A}\|_{H^1(\Sigma)}\leq c(B)\sigma^{-3}$.
\end{proof}
\begin{lem}\label{operatornormDelta} Let $\Sigma\hookrightarrow M$ be a surface such that $|A|\leq 10$ and
\begin{enumerate}[label=\textup{(\roman*)}]
\item $\frac12 \hat r\leq r_\Sigma\leq R_\Sigma\leq \frac32 \hat r$,
\item $\frac\sigma2\leq \sinh \hat r\leq \frac32\sigma$,
\end{enumerate}
and the Sobolev inequality in Lemma \textup{\ref{lemmagraphsob}} holds. Suppose moreover that \eqref{supRicnuX29} holds. Then, for every $p>2$, there exists a universal constant $C_p>0$, only depending on $p$, and a radius $\sigma_0=\sigma_0(C)$ such that, if $\sigma>\sigma_0$,
\begin{equation*}
\sup_{T: \ \|T\|_{W^{1,q}(\Sigma)}\leq 1}\left|\int_\Sigma \langle\Delta \overset{\circ}{A},T\rangle \ d\mu\right|\leq  C_p\left(\sigma^{-2}\|H-h\|_{W^{1,p}(\Sigma)}+\sigma^{\frac{2}p-1}\|\overset{\circ}{A}\|_{H^1}+C\sigma^{-6+\frac2p}\right),
\end{equation*}
where $q$ is the Hölder conjugate of $p$.
\end{lem}
\begin{proof} By Simons' identity, see again \cite{metzger}, we have
\begin{equation}\label{simons}
\begin{aligned}
\Delta A = \ & \textup{Hess}(H) + H h_i^l h_{lj} - |A|^2 A + h_i^l \overline{\textup{Rm}}_{kjkl}+h^{lk}\overline{\textup{Rm}}_{lijk} \nonumber \\
& + \nabla_j\left(\overline{\textup{Ric}}_{i\varepsilon}\nu^\varepsilon\right)+\nabla_l \left(\overline{\textup{Rm}}_{\varepsilon ijl}\nu^\varepsilon \right),
\end{aligned}
\end{equation}
and since moreover $\Delta \left(\frac{H}{2}g\right)=\frac12\left(\Delta H\right)g$, we find
\begin{equation*}
\begin{aligned}
\Delta \overset{\circ}{A} = \ & \textup{Hess}(H) -\frac12\left(\Delta H\right)g+ H h_i^l h_{lj} - |A|^2 A + h_i^l \overline{\textup{Rm}}_{kjkl}+h^{lk}\overline{\textup{Rm}}_{lijk} \nonumber \\
& + \nabla_j\left(\overline{\textup{Ric}}_{i\varepsilon}\nu^\varepsilon\right)+\nabla_l \left(\overline{\textup{Rm}}_{\varepsilon ijl}\nu^\varepsilon \right).
\end{aligned}
\end{equation*}
Thus 
\begin{equation}\label{express230}
\begin{aligned}
\langle \Delta \overset{\circ}{A},T\rangle = \ & \langle\textup{Hess}(H),T\rangle -\frac12\left(\Delta H\right)\textup{tr}_g(T)+ H \langle A^2,T\rangle - |A|^2 \langle A,T\rangle \\
& + \left(h^{il} \overline{\textup{Rm}}\indices{_k^j_{kl}}+h^{lk}\overline{\textup{Rm}}\indices{_l^{ij}_k}\right)T_{ij}+ \left(\nabla^j\left(\overline{\textup{Ric}}\indices{^i_\varepsilon}\nu^\varepsilon\right)+\nabla_l \left(\overline{\textup{Rm}}\indices{_\varepsilon^{ij}_l}\nu^\varepsilon \right)\right)T_{ij}
\end{aligned}
\end{equation}
Notice that, using normal coordinates, 
\begin{equation*}
Hh_i^lh_{lj}-|A|^2h_{ij}=Hh_{il}\overset{\circ}{h}_{lj}-|\overset{\circ}{A}|^2h_{ij}.
\end{equation*} 
Moreover, using again normal coordinates, and by definition of traceless second fundamental form, 
\begin{equation*}
\left|h_{il} \overline{\textup{Rm}}_{kjkl}+h_{lk}\overline{\textup{Rm}}_{lijk}\right|=\left|\overset{\circ}{h}_{il}\overline{\textup{Rm}}_{kjkl}+\overset{\circ}{h}_{lk}\overline{\textup{Rm}}_{lijk}\right|\leq C|\overset{\circ}{A}|,
\end{equation*}
since the additional terms $\frac{H}2\overline{\textup{Rm}}_{kjki}$ and $\frac{H}2\overline{\textup{Rm}}_{kijk}$ simplifies and the Riemann tensor is bounded. Furthermore, writing also $T=\overset{\circ}{T}+\textup{tr}_g(T)g$, 
\begin{equation*}
\begin{aligned}
\left(\nabla^j\left(\overline{\textup{Ric}}_{i\varepsilon}\nu^\varepsilon\right)+\nabla_l\left(\overline{\textup{Rm}}_{\varepsilon ijl}\nu^\varepsilon\right)\right)T_{ij}=&\left(\nabla^i\left(\overline{\textup{Ric}}_{i\varepsilon}\nu^\varepsilon\right)+\nabla_l\left(\overline{\textup{Rm}}_{\varepsilon iil}\nu^\varepsilon\right)\right)\overset{\circ}{T}_{ij}\\
=&\left(\nabla^i\left(\overline{\textup{Ric}}_{i\varepsilon}\nu^\varepsilon\right)-\nabla_l\left(\overline{\textup{Ric}}_{l\varepsilon}\nu^\varepsilon\right)\right)\overset{\circ}{T}_{ij}.
\end{aligned}
\end{equation*}
Using the hypothesis \eqref{supRicnuX29}, we have $\left|\overline{\textup{Ric}}_{i\varepsilon}\nu^\varepsilon\right|\left|\nabla_k\overset{\circ}{T}_{lj}\right|\leq C\sigma^{-5}\left|\nabla\overset{\circ}{T}\right|$. Applying integration by parts to \eqref{express230}, we find that, for every $T$ s.t. $\|T\|_{W^{1,q}}\leq 1$,
\begin{equation*}
\begin{aligned}
\left|\int_\Sigma \langle\Delta \overset{\circ}{A} ,T\rangle \ d\mu\right| &\leq C\left(\|\nabla H\|_{L^p(\Sigma)}\|\nabla T\|_{L^q(\Sigma)}+\|\overset{\circ}{A}\|_{L^p(\Sigma)}\|T\|_{L^q(\Sigma)}\right)+C\sigma^{-5+\frac{2}{p}}\|\nabla \overset{\circ}{T}\|_{L^q(\Sigma)}\\
&\leq C\left(\sigma^{-1}\|\nabla H\|_{L^p(\Sigma)}\|T\|_{W^{1,q}(\Sigma)}+\sigma^{\frac2p-1}\|\overset{\circ}{A}\|_{H^1(\Sigma)}\right)+C\sigma^{-6+\frac2p}\|T\|_{W^{1,q}(\Sigma)}\\
&\leq C\left(\sigma^{-1}\|\nabla H\|_{L^p(\Sigma)}+\sigma^{\frac{2}p-1}\|\overset{\circ}{A}\|_{H^1}+C\sigma^{-6+\frac2p}\right).
\end{aligned}
\end{equation*}
\end{proof}
\begin{prp}\label{prp310}
Let $\Sigma\hookrightarrow M$ be a surface such that there exists $B>0$, a universal constant $C>0$ and a small constant $\varepsilon>0$ such that
\begin{equation}\label{prp3101}
\frac12 \hat r\leq r_\Sigma\leq R_\Sigma\leq \frac32 \hat r, \qquad \frac\sigma2\leq \sinh \hat r\leq \frac32\sigma, \qquad \frac72\pi\sigma^2<|\Sigma|_g<5\pi\sigma^2,
\end{equation}
\begin{equation}\label{prp3102}
\mathcal H\geq \frac3{2\sigma}, \qquad |\overset{\circ}{A}|^2\leq \varepsilon \sigma^{-2}, \qquad \|\nabla H\|_{L^2(\Sigma)}\leq B\sigma^{-4}, \qquad \|\nabla H\|_{L^4(\Sigma)}\leq B\sigma^{-\frac52},
\end{equation}
and the Sobolev inequality in Lemma \textup{\ref{lemmagraphsob}} holds. Suppose moreover that \eqref{supRicnuX29} holds. Then, there exist a small universal constant $\varepsilon_0>0$, a constant $c=c(B)>0$ and a large radius $\sigma_0=\sigma_0(B,\varepsilon_0)$ such that, if $\varepsilon<\varepsilon_0$ and $\sigma>\sigma_0$,
\begin{equation*}
\|\overset{\circ}{A}\|_{L^\infty(\Sigma)}\leq c(B)\sigma^{-2}.
\end{equation*}
\end{prp}
\begin{proof}
By Lemma \ref{H1boot} we have $\|\overset{\circ}{A}\|_{H^1(\Sigma)}\leq c(B)\sigma^{-3}$. Moreover, Lemma \ref{operatornormDelta} with $p=4$ implies
\begin{equation*}
\sup_{T: \ \|T\|_{W^{1,\frac43}}\leq 1}\left|\int_\Sigma \langle\Delta \overset{\circ}{A},T\rangle \ d\mu\right|\leq C\left(\sigma^{-1}\|\nabla H\|_{L^4(\Sigma)}+\sigma^{-\frac12}\|\overset{\circ}{A}\|_{H^1(\Sigma)}+C\sigma^{-\frac{11}2}\right).
\end{equation*}
Since $\|\nabla H\|_{L^4(\Sigma)}\leq B\sigma^{-\frac52}$ by the hypothesis, we find $\|\Delta \overset{\circ}{A}\|_{W^{-1,4}(\Sigma)}\leq c(B)\sigma^{-\frac72}$, for $\sigma$ suitably large with respect to $B$. Note also that, again by Lemma \ref{H1boot}, $\|\overset{\circ}{A}\|_{L^2(\Sigma)}\leq c(B)\sigma^{-3}$. Combining this with \cite[Prop. D.2]{nerz}, we get $\|\overset{\circ}{A}\|_{L^\infty(\Sigma)}\leq c(B)\sigma^{-2}$, as desired.
\end{proof}
\subsection{Graph regularity of round surfaces}\label{subsectgraphs} The Section above showed that a relaxed version of hypotheses \eqref{eq312} and \eqref{eq313} in Definition \ref{dfnroundenss36}, i.e. \eqref{prp3101} and \eqref{prp3102}, allows to prove that the first inequality in \eqref{eq314b} holds. We now dedicate this Subsection to understanding the graph regularity of $\Sigma$ with respect to $\mathbb S_{\hat r}$.\\
\indent This proposition is a generalization of the strong approximation carried out in \cite[Section 7]{nevestian}, where it is assumed to work with CMC-surfaces. We weaken this assumption introducing hypothesis (iii) and (iv) of the following Proposition. However, the ingredients used in the proof are those involved in \cite{nevestian}.
\begin{prp}[Part I]\label{prp311partI} Let $\Sigma$ be a surface immersed in an asymptotically hyperboloidal initial data set. Suppose that the following holds. There exist $\hat r_0>0$ and $B>0$ such that, if $\sigma=\sqrt{|\Sigma|/4\pi}$,
\begin{enumerate}[label=\textup{(\roman*)}]
\item $\frac\sigma2\leq \sinh r_\Sigma\leq \sinh R_\Sigma \leq \frac32\sigma$;
\item $\|\overset{\circ}{A}\|_{L^\infty(\Sigma)}^2+\|H-h\|_{L^\infty(\Sigma)}\leq B\sigma^{-4}$ and $\|\nabla H\|_{L^4(\Sigma)}\leq B\sigma^{-4}$;
\item $|\partial_r^\top|\leq B\sigma^{-1}$ on $\Sigma$;
\item $|\sinh \hat r_0-\sigma|\leq B$ and 
\begin{equation*}
h=2+\frac1{\sinh^2 \hat r_0}+o(\sigma^{-2}).
\end{equation*}
\end{enumerate}
Then, provided $\sigma$ is suitably large, there exist an \textup{approximate radius} $\hat r=\hat r(\sigma)>0$ and a universal constant $C>0$ such that $C^{-1}\sigma\leq \sinh \hat r\leq C\sigma$ and 
\begin{equation}\label{hexpldelta}
h=2+\frac1{\sinh^2 \hat r}-\frac{m}{\sinh^3 \hat r}+O(\sigma^{-3-\delta}).
\end{equation}
for some $\delta>0$.
\end{prp}
\begin{rem}
The existence of a "reference" radius $\hat r_0$ is only needed to have 
\begin{equation*}
h=2+\frac1{\sigma^2}+o(\sigma^{-2}).
\end{equation*}
\end{rem}
\begin{proof}
The proof is identical to that in \cite[Thm. 7.1]{nevestian}. However, for sake of completeness we sketch it. The identity in \cite[Prop. 3.4]{nevestian} says that, on $\Sigma$,
\begin{equation*}
\Delta r=2\frac{\cosh r}{\sinh r}-\frac{m}{\sinh^3 r}-H+F,
\end{equation*}
where
\begin{equation*}
F=\underbrace{\left(H-2\right)\left(1-\langle \partial_r,\nu\rangle\right)}_{F_1}+\underbrace{\left(1-\langle \partial_r,\nu\rangle\right)^2-2|\partial_r^\top|^2e^{-2r}+|\partial_r^\top|^2O(e^{-3r})+O(e^{-5r})}_{F_2}.
\end{equation*}
Note that, apart from the first addend, the expression for $F_2$ is independent of $H$ and satisfies $|F_2|\leq c\sigma^{-4}$ and $|\nabla F_2|\leq c\sigma^{-5}$. In fact, by the hypothesis $|\partial_r^\top|=O(\sigma^{-1})$, and by easy computations $1-\langle \partial_r,\nu\rangle^2$ is close to $|\partial_r^\top|^2$. For estimate on the derivative of $|\partial_r^\top|^2$, we postpone the computations to Lemma \ref{appendixA}. Keep also in mind \eqref{stimapartialrnu}. We now briefly remark that the remaining part of $F$, i.e. $F_1=\left(H-2\right)\left(1-\langle \partial_r,\nu\rangle\right)$, can be estimated in $C^{0,\frac12}$. In fact, by Lemma \ref{appendixA} and \eqref{stimapartialrnu} we have 
\begin{equation*}
\|1-\langle \partial_r,\nu\rangle\|_{W^{1,\infty}(\mathbb S^2)}=O(\sigma^{-2}),
\end{equation*}
and thus, using Morrey's inequality as below,
\begin{equation*}
\begin{aligned}
\|\left(H-2\right)\left(1-\langle \partial_r,\nu\rangle\right)\|_{C^{0,\frac12}(\mathbb S^2)}&\leq C\|1-\langle \partial_r,\nu\rangle\|_{W^{1,\infty}(\mathbb S^2)}\|H-2\|_{C^{0,\frac12}(\mathbb S^2)}\\
&\leq C\sigma^{-1}\|\nabla H\|_{L^4(\mathbb S^2)}\leq C\sigma^{-\frac12}\|\nabla H\|_{L^4(\Sigma)}\leq c\sigma^{-\frac{9}2}.
\end{aligned}
\end{equation*}
Thus, the same equation for $\Delta r$ as in the proof of \cite[Thm. 7.1]{nevestian} holds, with $H$ replaced by $h$ (also using the estimate on $\|H-h\|_{L^\infty(\Sigma)}$). Exactly as there, by the mean value theorem, the existence of $\hat r$ follows, and $C^{-1}\sigma\leq \sinh \hat r\leq C\sigma$ because of (i).
\end{proof}
\begin{prp}[Part II] \label{prp311} Let $\Sigma$ be a surface immersed in an asymptotically hyperboloidal initial data set with $\sigma=\sqrt{|\Sigma|/4\pi}$. Suppose that conditions \textup{(i)-(iv)} in Proposition \textup{\ref{prp311partI}} hold with a constant $B>0$ and, for $\hat r$ as in Proposition \ref{prp311partI}, set $w_{\hat r}:=r-\hat r$. Then, there exist $c=c(B)>0$, a universal small constant $c_0>0$ and a large radius $\sigma_0=\sigma_0(B,c_0)>1$ such that if
\begin{enumerate}[label=\textup{(\roman*)}]
\item $|r-\hat r|\leq B\sigma^{-1}$;
\item $|\sinh \hat r-\sigma|\leq B$ and $\textup{diam}\left(\Sigma\right)\leq B\sigma$;
\item $|\nu^\top|+\|w_{\hat r}\|_{W^{1,\infty}(\mathbb S^2)}<c_0$;
\end{enumerate}
and $\sigma>\sigma_0$, then
\begin{equation}\label{graphcontrol}
\|w_{\hat r}\|_{C^{2,\frac12}(\Sigma)}\leq c\sigma^{-1}.
\end{equation}
\end{prp}
\begin{rem}
\textup{(i)} Hypothesis \textup{(iii)} implies, by \eqref{holderscaling} and \eqref{CalphaW1infty}, that $\|w_{\hat r}\|_{W^{1,\infty}(\Sigma)}<c_0$ and $\|w_{\hat r}\|_{C^{0,\alpha}(\mathbb S^2)}<c_0$ for $\alpha\in (0,1)$.\\
\noindent \textup{(ii)} Rescaling $\Sigma$ by its area, we find that hypothesis \eqref{eq314b}, in particular $\|w_{\hat r}\|_{W^{1,\infty}(\Sigma)}\equiv\|w_{\hat r}\|_{C^1(\Sigma)}<B_3\sigma^{-1}$, implies that, by \eqref{holderscaling} and \eqref{CalphaW1infty}, $\|w_{\hat r}\|_{C^{0,\alpha}(\mathbb S^2)}<B_3\sigma^{-1}$ for $\alpha\in (0,1)$, and $\|w_{\hat r}\|_{W^{1,\infty}(\mathbb S^2)}<B_3\sigma^{-1}$. Thus, if $\sigma$ is suitably large with respect to $B_3$ and $c_0$, we find that the hypothesis of Proposition \textup{\ref{prp311}} are satisfied. 
\\
\noindent \textup{(iii)} Note that, inequality \eqref{CalphaW1infty} on the sphere, i.e. $\sigma=1$, implies that $\|w\|_{C^{0,\alpha}(\mathbb S^2)}\leq C(c_0)$, with $C(c_0)=o(1)$ with respect to $c_0$, since $\alpha\in (0,1)$.
\end{rem}
\begin{rem}
This is a regularity result on the graph function of $\Sigma$: provided that $\Sigma$ is a graph of a function on an approximating sphere $\mathbb S_{\hat r}$ and provided this graph function is suitably small in $W^{1,\infty}$, we obtain information on the $C^{2,\frac12}$ norm of the graph.
\end{rem}
\begin{proof}
The proof will be based on elliptic regularity (Schauder's estimates) on the unit sphere. In order to use this tool, we have to show that $(\Sigma,g)$ is conformal to the round metric on the unit sphere, and thus $w_{\hat r}$ is a map also on the unit sphere. For the rest of the proof, for sake of brevity, we will write $w$ instead of $w_{\hat r}$, without ambiguities.\\
\indent First of all, we know that $w_{\hat r}$ is well-defined by Lemma \ref{prp311partI}. To see that the metric $g$ is conformal, we start normalizing the metric $g$ and computing the Gauss curvature of this normalized metric, say $\hat g$. We will obtain that 
\begin{equation}\label{hatversion310}
\|w\|_{C^{2,\frac12}(\hat \Sigma)}\leq c\sigma^{-1}.
\end{equation}
The conclusion, i.e. \eqref{graphcontrol}, will follow by the fact that $|\nabla w|=\frac1{\sigma}|\hat \nabla w|$, and a similar identity holds for a further derivative, and by Definition \ref{dfn35holder}.
\\
\indent We start now the proof of \eqref{hatversion310}. As briefly sketched above, we start showing that $g$ can be rescaled to a metric on $\mathbb S^2$. For this reason, we introduce 
\begin{equation}\label{rescalingmetric}
\hat g:=4\pi |\Sigma|^{-1}g=\sigma^{-2}g.
\end{equation}
Appendix \ref{appendixschauder} shows that this can be thought as a metric on $\mathbb S^2$, conformal to the round metric $g_0$ on $\mathbb S^2$.  In particular, by Lemma \ref{kwdecomp}, we can decompose $w=u+v$, with $u\in \textup{Ker}(\Delta_{g_0}+2)$ and $v$ orthogonal to $u$, and thus write 
\begin{equation*}
u=C\sum_{i=1}^3 \langle w,x_i\rangle_{L^2(\mathbb S^2)}x_i, \qquad \|u\|_{C^{2,\frac12}(\mathbb S^2)}\leq c\sigma^{-1}.
\end{equation*}
We now have to estimate $\|v\|_{C^{2,\frac12}(\mathbb S^2)}$. In order to do this, we have to find the equation satisfied by $v$. We start with the elliptic equation satisfied by $w$. Continuing from the equation above for $\Delta r$, we find, employing hypothesis (iv) and (iii), 
\begin{equation*}
\Delta r=2\sqrt{1+\frac1{\sinh^2 r}}-2-\frac1{\sinh^2 \hat r}-\frac{m}{\sinh^3 r}+\frac{m}{\sinh^3 \hat r}-q_0+(h-H)+F,
\end{equation*}
with $q_0=O(\sigma^{-3-\delta})$. We remark that $q_0$ is a constant (for $\sigma$ and $\hat r$ fixed). We then set $\mathcal R=F-(H-h)-q_0$, noticing that $|\mathcal R|=O(\sigma^{-3-\delta})$. Combining \eqref{rescalingmetric} with $\hat g=e^{2\beta}g_0$, we find $g=\sigma^2 e^{2\beta}g_0$, and thus 
\begin{equation}\label{eq319Delta0}
\begin{aligned}
\Delta_0 w &=\sigma^2 e^{2\beta}\Delta (r-\hat r)=\sigma^2 e^{2\beta}\Delta r\\
&=\sigma^2 e^{2\beta}f_1(r,\hat r)+\sigma^2 e^{2\beta} f_2(r,\hat r)+\sigma^2e^{2\beta}\mathcal R\\
&=\sigma^2 f_1(r,\hat r)+\sigma^2 f_2(r,\hat r)+\sigma^2(e^{2\beta}-1)f_1(r,\hat r)+\sigma^2(e^{2\beta}-1)f_2(r,\hat r)+\sigma^2e^{2\beta}\mathcal R,
\end{aligned}
\end{equation}
where
\begin{equation*}
f_1(r,\hat r):=2\sqrt{1+\frac1{\sinh^2 r}}-2-\frac1{\sinh^2 \hat r}, \qquad f_2(r,\hat r):=-\frac{m}{\sinh^3 r}+\frac{m}{\sinh^3 \hat r}.
\end{equation*}
Note that, since $\beta$ is uniformly bounded,
\begin{equation*}
\|(e^{2\beta}-1)f_i\|_{C^{0,\alpha}}\leq C\|\beta\|_{C^{0,\alpha}}\|f_i\|_{C^{0,\alpha}},
\end{equation*}
also with $f_i$ replaced by $\mathcal R$. Moreover, by \cite[Proof of Thm. 5.1.]{nevestian}, $\beta$ satisfies 
\begin{equation*}
\Delta_0\beta=1-\hat \kappa e^{2\beta},
\end{equation*}
and since $\hat\kappa$ and $\beta$ are uniformly bounded Calderón–Zygmund estimates imply that $\|\beta\|_{W^{2,p}(\mathbb S^2)}$ is bounded for every $p<\infty$. For $p>2$, Morrey's embedding implies that $\|\beta\|_{C^{0,\alpha}(\mathbb S^2)}$ is bounded. Thus it is sufficient to bound the $C^{0,\alpha}(\mathbb S^2)$-norm of $f_1$, $f_2$ and $\mathcal R$.\\
\indent Firstly, note that $\nabla \mathcal R=\nabla F-\nabla H$, and the $C^{0,\frac12}$-norm of $F$ has already been studied above. On the other hand, $\|H-h\|_{C^{0,\frac12}(\mathbb S^2)}$ is controlled by $\|\nabla H\|_{L^4(\mathbb S^2)}\leq C\sigma^{\frac12}\|\nabla H\|_{L^4(\Sigma)}\leq c\sigma^{-\frac72}$ by Morrey's inequality, using also \eqref{nerzrescal}. Thus $\|\mathcal R\|_{C^{0,\frac12}(\mathbb S^2)}\leq c\sigma^{-3-\delta}$ and so 
\begin{equation*}
\sigma^2\|\mathcal R\|_{C^{0,\frac12}(\mathbb S^2)}\leq c\sigma^{-1-\delta}.
\end{equation*}
We now bound $\|f_i(r,\hat r)\|_{C^{0,\alpha}(\mathbb S^2)}$. First of all, notice that $f_i(\hat r,\hat r)=0$. Moreover 
\begin{equation*}
\frac{d}{dr}f_1\bigg|_{r=\hat r}=-\frac{2}{\sinh^2 \hat r}, \qquad \frac{d^2}{dr^2}f_1\bigg|_{r}=\frac{4\cosh r}{\sinh^3 r}.
\end{equation*}
Since $|w|=|r-\hat r|$ is uniformly bounded and small for $\sigma$ suitably large, we have
\begin{equation*}
f_1(r,\hat r)=-\frac{2}{\sinh^2 \hat r}w+\frac{2\cosh r_\xi}{\sinh^3 r_\xi}w^2,
\end{equation*}
for some $r_\xi$ close to $\hat r$. Thus 
\begin{equation*}
\|f_1(r,\hat r)\|_{C^{0,\alpha}(\mathbb S^2)}\leq C\sigma^{-2}\|w\|_{C^{0,\alpha}(\mathbb S^2)}+C\sigma^{-2}\|w\|_{C^{0,\alpha}(\mathbb S^2)}^2.
\end{equation*}
A similar argument holds for $f_2$. Thus, the ${C^{0,\alpha}(\mathbb S^2)}$-norm of the RHS of \eqref{eq319Delta0} is bounded by 
\begin{equation*}
C\|w\|_{C^{0,\alpha}(\mathbb S^2)}+C\|w\|_{C^{0,\alpha}(\mathbb S^2)}^2+c\sigma^{-1-\delta}.
\end{equation*}
Moreover, note that we can write 
\begin{equation*}
\Delta_0w+\frac{2\sigma^2}{\sinh^2 \hat r}w=O\left(\sigma^{-1}\|w\|_{C^{0,\alpha}(\mathbb S^2)}\right)+O\left(\|w\|_{C^{0,\alpha}(\mathbb S^2)}^2+c\sigma^{-1-\delta}\right),
\end{equation*}
since $f_2'(\hat r,\hat r)\sim \sigma^{-3}$. Note also that 
\begin{equation*}
\left|\frac{\sigma^2}{\sinh^2 \hat r}-1\right|\leq c\sigma^{-1},
\end{equation*}
and thus we can rewrite 
\begin{equation}\label{furtherineq322}
\Delta_0w+2w=O\left(\sigma^{-1}\|w\|_{C^{0,\alpha}(\mathbb S^2)}\right)+O\left(\|w\|_{C^{0,\alpha}(\mathbb S^2)}^2+c\sigma^{-1-\delta}\right),
\end{equation}
Since $\Delta_0 u=-2u$, we find 
\begin{equation*}
\begin{aligned}
\Delta_0 v=\Delta_0(w-u)&=-2w+O\left(\sigma^{-1}\|w\|_{C^{0,\alpha}(\mathbb S^2)}\right)+O\left(\|w\|_{C^{0,\alpha}(\mathbb S^2)}^2+c\sigma^{-1-\delta}\right)+2u\\
&=-2v+O\left(\sigma^{-1}\|w\|_{C^{0,\alpha}(\mathbb S^2)}\right)+O\left(\|w\|_{C^{0,\alpha}(\mathbb S^2)}^2+c\sigma^{-1-\delta}\right).
\end{aligned}
\end{equation*}
Thus, by Schauder's estimate we have 
\begin{equation*}
\|v\|_{C^{2,\alpha}(\mathbb S^2)}\leq C\|v\|_{L^\infty(\mathbb S^2)}+C\sigma^{-1}\|w\|_{C^{0,\alpha}(\mathbb S^2)}+C\|w\|_{C^{0,\alpha}(\mathbb S^2)}^2+c\sigma^{-1-\delta},
\end{equation*}
and since $\|w\|_{C^{0,\alpha}(\mathbb S^2)}\leq \|u\|_{C^{0,\alpha}(\mathbb S^2)}+\|v\|_{C^{0,\alpha}(\mathbb S^2)}$, for $\sigma$ large we can rewrite 
\begin{equation*}
\begin{aligned}
\|v\|_{C^{2,\alpha}(\mathbb S^2)}&\leq C\|v\|_{L^\infty(\mathbb S^2)}+C\sigma^{-1}\|u\|_{C^{0,\alpha}(\mathbb S^2)}+C\|w\|_{C^{0,\alpha}(\mathbb S^2)}^2+c\sigma^{-1-\delta}\\
&\leq C\|v\|_{L^\infty(\mathbb S^2)}+C\|w\|_{C^{0,\alpha}(\mathbb S^2)}^2+c\sigma^{-1-\delta},
\end{aligned}
\end{equation*}
where we used \eqref{inequC012}. Moreover 
\begin{equation*}
\begin{aligned}
C\|w\|_{C^{0,\alpha}(\mathbb S^2)}^2&\leq C\|w\|_{C^{0,\alpha}(\mathbb S^2)}\left(\|u\|_{C^{0,\alpha}(\mathbb S^2)}+\|v\|_{C^{0,\alpha}(\mathbb S^2)}\right)\\
&\leq C\|w\|_{C^{0,\alpha}(\mathbb S^2)}\left(c\sigma^{-1}+\|v\|_{C^{0,\alpha}(\mathbb S^2)}\right)\\
&\leq Cc\sigma^{-1}+C\|w\|_{C^{0,\alpha}(\mathbb S^2)}\|v\|_{C^{0,\alpha}(\mathbb S^2)}\\
&\leq Cc\sigma^{-1}+Cc_0\|v\|_{C^{0,\alpha}(\mathbb S^2)}
\end{aligned}
\end{equation*}
using also \eqref{inequC012}. Provided $c_0$ is suitably small with respect to the universal constant $C$, we find 
\begin{equation*}
\begin{aligned}
\|v\|_{C^{2,\alpha}(\mathbb S^2)} &\leq C\|v\|_{L^\infty(\mathbb S^2)}+C\|w\|_{C^{0,\alpha}(\mathbb S^2)}^2+Cc\sigma^{-1}.
\end{aligned}
\end{equation*}
Note that  
\begin{equation*}
\|v\|_{L^\infty(\mathbb S^2)}\leq \|u\|_{L^\infty(\mathbb S^2)}+\|w\|_{L^\infty(\mathbb S^2)}\leq c\sigma^{-1},
\end{equation*}
using hypothesis (i) and \eqref{inequC012}. Thus $\|v\|_{C^{2,\alpha}(\mathbb S^2)}\leq c\sigma^{-1}$, and combining this with \eqref{C2ofu323}, $\|w\|_{C^{2,\alpha}(\mathbb S^2)}\leq c\sigma^{-1}$.
\end{proof}
\subsection{Spectral analysis of the stability operator} 
The stability operator is at the core of the construction of the foliation by Neves-Tian in \cite{nevestian}. Typically, in the field of CMC-foliation this operator is of crucial interest and importance, see \cite{huiskenyau}, \cite{huang}, \cite{nerz}, \cite{cederbaumsakovich}. 
\begin{dfn}\label{defstabop} Let $\Sigma\hookrightarrow M$ be a surface embedded in the Riemannian 3-manifold $(M,\overline g)$. We define the \textit{stability operator} associated to $\Sigma$ as 
\begin{equation*}
L^\Sigma f:=\Delta f+\left(|A|^2+\overline{\textup{Ric}}(\nu,\nu)\right)f,
\end{equation*}
for every $f\in C^2(\Sigma)$.
\end{dfn}
Our method of construction of a STCMC foliation will rely on the VPSTMC, and the stability operator is fundamental in order to estimate the evolution in time of the $L^2$-norm of the speed of the flow, and in order to prove its exponential decay to zero. The following proposition is well-known, but we sketch here a proof under our assumption for the convenience of the reader.
\begin{prp}\label{spectralprop} Suppose that $\Sigma$ is a surface in $(M,\overline g)$ such that there exist an approximate radius $\hat r>0$, a constant $B>0$ satisfying
\begin{enumerate}[label=\textup{(\roman*)}]
\item $\frac{\sigma}{2}\leq \sinh \hat r\leq \frac32 \sigma$;
\item $|\sinh r-\sinh \hat r|\leq B$ on $\Sigma$;
\item $|\overset{\circ}{A}|\leq B\sigma^{-2}$, $|H-h|\leq B\sigma^{-3-\delta}$ and $|\nu^\top|\leq \frac{10^{-3}}{\sqrt m}$;
\item \begin{equation*}
h=2\sqrt{1+\frac{1}{\sinh^2 \hat r}}-\frac{m}{\sinh^3 \hat r}+O(\sigma^{-3-\delta}).
\end{equation*}
\end{enumerate}
Then, there exists $\sigma_0=\sigma_0(C,B,\delta,m)>1$ such that if $\sigma>\sigma_0$ then 
\begin{equation}\label{spectralineq}
\int_\Sigma (-\Delta f)f-\left(|A|^2+\overline{\textup{Ric}}(\nu,\nu)\right) f^2 \ d\mu\geq \frac{20m}{27\sigma^3}\int_\Sigma f^2 \ d\mu,
\end{equation}
for every $f\in C^\infty(\Sigma)$ such that $\int_\Sigma f \ d\mu=0$.
\end{prp}
\begin{rem}
Notice that $\Sigma$ is almost CMC by assumption.
\end{rem}
\begin{proof}
By the Gauss equation 
\begin{equation*}
2\kappa=\frac{H^2}{2}+\overline{R}-2\overline{\textup{Ric}}(\nu,\nu)-|\overset{\circ}{A}|^2=\frac{h^2}{2}+\overline{R}-2\overline{\textup{Ric}}(\nu,\nu)-\frac{|\overset{\circ}{A}|^2}{2}+\frac{H^2-h^2}2.
\end{equation*}
Moreover, Lemma \ref{lemma14ii} implies that
\begin{equation*}
\overline{\textup{Ric}}(\nu,\nu)=-2-\frac{m}{\sinh^3r}+\frac{3m|\nu^\top|^2}{2\sinh^3r}+O(e^{-5r}).
\end{equation*}
Proceeding as at the beginning of the proof of Proposition \ref{prp311}, i.e. as in \cite[Thm. 7.1]{nevestian}, we find that there exists $\hat r=\hat r(\sigma)>0$ such that 
\begin{equation*}
h=2\sqrt{1+\frac{1}{\sinh^2 \hat r}}-\frac{m}{\sinh^3 \hat r}+O(\sigma^{-3-\delta})
\end{equation*}
and $C^{-1}\sigma\leq \sinh\hat r\leq C\sigma$ for some $C>0$. Since $|\overset{\circ}{A}|^2$ and $H^2-h^2$ are lower order terms by assumption (iii), we find
\begin{equation}\label{gausscurvature358}
\begin{aligned}
2\kappa&=2+\frac2{\sinh^2 \hat r}-\frac{2m}{\sinh^3 \hat r}-6+4+\frac{2m}{\sinh^3 r}-\frac{3m|\nu^\top|^2}{\sinh^3 r}+O(\sigma^{-3-\delta})\\
&=\frac2{\sinh^2 \hat r}-\frac{3m|\nu^\top|^2}{\sinh^3 \hat r}+O(\sigma^{-3-\delta}),
\end{aligned}
\end{equation}
where we used that 
\begin{equation*}
\left|\frac1{\sinh^3 r}-\frac1{\sinh^3 \hat r}\right|\leq C\sigma^{-4}\left|\sinh r-\sinh \hat r\right|\leq CB\sigma^{-4}.
\end{equation*}
Note that the latter term in \eqref{gausscurvature358} is a uniform bound from below for $2\kappa$ on $\Sigma$. On the other hand, we remark that 
\begin{equation*}
\begin{aligned}
2\kappa-|A|^2-\overline{\textup{Ric}}(\nu,\nu)&=\overline R-3\overline{\textup{Ric}}(\nu,\nu)-2|\overset{\circ}{A}|^2\\
&=-6+6+\frac{3m}{\sinh^3 r}-\frac{9m|\nu^\top|^2}{2\sinh^3 r}+O(\sigma^{-3-\delta})\\
&=\frac{3m}{\sinh^3 \hat r}-\frac{9m|\nu^\top|^2}{2\sinh^3 \hat r}+O(\sigma^{-3-\delta})\\
&\geq \frac{5m}{2\sinh^3 \hat r}
\end{aligned}
\end{equation*}
for $\hat r$ large, using again hypothesis (iii). By Rayleigh Theorem we have 
\begin{equation*}
\int_\Sigma (-\Delta f)f-\left(|A|^2+\overline{\textup{Ric}}(\nu,\nu)\right) f^2 \ d\mu\geq \frac{5m}{2\sinh^3 \hat r}\int_\Sigma f^2 \ d\mu.
\end{equation*}
\end{proof}
\section{The flow}\label{theflow}
The flow we study for the rest of the paper is a modification of the volume preserving mean curvature flow, and it has been introduced by the author in \cite{vpstmcf}.
\begin{dfn} Let $(M,\overline g,\overline K)$ be an initial data set and let $\iota:\Sigma\hookrightarrow  M$ be a closed surface. A time dependent family of immersions $F_t:\Sigma\hookrightarrow\M$, with $t\in [0,T)$ for some $0<T\leq \infty$, which satisfies 
\begin{equation}\label{generalflow513}
\begin{cases}
\frac{\partial}{\partial t} F_t(\cdot)=-\left(\mathcal H(\cdot,t)-\hbar(t)\right)\nu(\cdot,t)\\
F_0=\iota
\end{cases}
\end{equation}
is called a solution to the \textup{volume preserving spacetime mean curvature flow}, with initial value $\iota$.
\end{dfn}
Local existence and regularity of this flow are assured and have been investigated in \cite{vpstmcf}. In the following, we will assume that the ambient initial data set is asymptotically hyperboloidal. We write $\Sigma_t:=F_t(\Sigma)$ to denote the solution of the flow at time $t$ and we call $g(t)$ the induced metric and denote by $d\mu_t$ the corresponding measure and by $A(t)$ the second fundamental form of $\Sigma$ at time $t$. Since in what follows we will frequently use complicated integral expressions involving these quantities associated with $\Sigma_t$, we will in such cases abbreviate $g(t)$, $d\mu_t$, $A(t)$, etc. simply by $g$, $d\mu$, $A$, etc., leaving the dependence on $t$ implicit in the domain of integration $\Sigma_t$.\\
\indent We recall the evolution equations satisfied by the main geometric quantities on $\Sigma_t$. At each fixed $t$, we choose a frame $\{\vec e_\alpha(t)\}_{\alpha=1}^3$ on $(M,\overline g)$ such that  $\{\vec e_1(t),\vec e_2(t)\}$ are tangent vectors on $\Sigma_t$ and $\vec e_3(t):=\nu_t$. The following Lemma collects the equations satisfied by the main geometric quantities on $\Sigma_t$, see \cite{huiskenpolden}.
\begin{lem}\label{evolution41} Let $\{F_t\}_{t\in [0,T)}$ be a solution to the flow \eqref{generalflow513}. Then we have 
\begin{enumerate}[label=\textup{(\roman*)}]
\item $\frac{\partial g_{ij}}{\partial t}=-2\left(\mathcal H-\hbar\right)h_{ij}$;
\item $\frac{\partial}{\partial t}(d\mu_t)=- \left(\mathcal H-\hbar\right)Hd\mu_t$;
\item $\frac{\partial}{\partial t}\nu=\nabla \mathcal H$;
\item $\frac{\partial}{\partial t}h_{ij}=\nabla_i\nabla_j\mathcal H+\left(\mathcal H-\hbar\right)\left(-h_{ik}h_j^k+\overline{\textup{Rm}}_{ikjl}\nu^k\nu^l\right)$;
\item $\frac{\partial H}{\partial t}=\Delta \mathcal H+\left(\mathcal H-\hbar\right)(|A|^2+\overline{\textup{Ric}}(\nu,\nu))$.
\end{enumerate} 
\end{lem}
Since we will be mainly interested in the evolution of the speed of the flow, we remark that (v) implies 
\begin{equation}\label{evHcal}
\begin{aligned}
\frac{\partial \mathcal H}{\partial t}&=\frac{H}{\mathcal H}\frac{\partial H}{\partial t}-\frac{P}{\mathcal H} \frac{\partial P}{\partial t}\\
&=\frac{H}{\mathcal H}\left(\Delta \mathcal H+\left(|A|^2+\overline{\textup{Ric}}(\nu,\nu)\right)(\mathcal H-\hbar)\right)-\frac{P}{\mathcal H} \frac{\partial P}{\partial t}.
\end{aligned}
\end{equation}
\begin{lem}\label{lemmaevPt} Let $\Sigma_t$ be evolving according to \eqref{generalflow513}. Then 
\begin{equation}\label{evPt}
|\partial_t P|\leq C|\mathcal H-\hbar||A||\overline P|+C|\overline\nabla \ \overline P||\mathcal H-\hbar|+C|\overline P||\nabla\mathcal H|.
\end{equation}
\end{lem}
\begin{proof}
Since by definition $\partial_t P=\partial_t\left(g^{ij}\overline P_{ij}\right)=2(\mathcal H-\hbar)h^{ij}\overline P_{ij}+\overline\nabla_\gamma\overline P_{ij}(\hbar-\mathcal H)\nu^\gamma+2g^{ij}\overline P\left(\partial_t\partial_i F,\partial_j F\right)$ and proceeding as in \cite[Lemma 4.3]{vpstmcf} we have the thesis.
\end{proof}
Furthermore, note that (i) and (ii) in Lemma \ref{evolution41} imply 
\begin{equation*}
\frac{d}{dt}\textup{Vol}(\Omega_t)=\int_{\Sigma_t} \left(\mathcal H-\hbar\right) \ d\mu=0, \qquad \frac{d}{dt}|\Sigma_t|=\int_{\Sigma_t} H(\hbar-\mathcal H) \ d\mu,
\end{equation*}
where $\Omega_t$ is the region of $M$ enclosed by $\Sigma_t$, i.e. $\Sigma_t=\partial\Omega_t$. Note that, differently from the standard volume preserving mean curvature flow, this is not an area decreasing flow.
\subsection{Assumptions for the rest of the Section}\label{section41} For the rest of the Section, we will suppose that $\Sigma_t$ satisfies some assumption for $t\in[0,T]$, $T>0$. As one can see, these hypotheses are related to Definition \ref{dfnroundenss36}. We assume that there exist a universal constant $C>0$ and a constant $B_\infty>0$ such that the following inequalities are satisfied in $[0,T]$. We assume
\begin{equation}\label{hypRaggirR}
\frac12\sigma < \sinh r(t)< \frac32\sigma, \qquad \frac72\pi\sigma^2<|\Sigma_t|<5\pi\sigma^2.
\end{equation}
Moreover, that for every $t\in [0,T]$ there exists $\hat r_t$ such as in Proposition \ref{prp311partI}, such that 
\begin{equation}\label{hyp45}
|\sinh \hat r_t-\sigma|\leq C
\end{equation}
for some universal constant $C>0$ suitably large. Furthermore, we assume
\begin{equation}\label{hyp46}
H(t)< 3, \quad |A(t)|< 3,\qquad \frac3{2\sigma}< \mathcal H(t)< \frac{5}{2\sigma},
\end{equation}
\begin{equation}\label{hypP1219}
|P(t)|<3, \qquad |\nabla P(t)|<C\sigma^{-5}.
\end{equation}
\begin{equation}\label{hyp48}
\|\overset{\circ}{A}\|_{L^\infty(\Sigma_t)}< B_\infty \sigma^{-2},\qquad \|\mathcal H-\hbar\|_{L^\infty(\Sigma_t)}< C\sigma^{-2}, \qquad |\nu_t^\top|<2.
\end{equation}
\begin{rem}
Note that \eqref{hypP1219} is an immediate consequence of \eqref{hypRaggirR}, \textup{Lemma \ref{asymptP}} and \textup{Lemma \ref{asymptnablaP}}. Moreover, \textup{Lemma \ref{lemma33}} implies 
\begin{equation}\label{lemma22eq}
\left||A(t)|^2+\overline{\textup{Ric}}(\nu_t,\nu_t)\right|\leq \frac{\mathcal H^2(t)}{2}+O(\sigma^{-3})
\end{equation}
provided $\overline{\textup{Ric}}(\nu_t,\nu_t)=-2+O(\sigma^{-3})$. This is a consequence of \textup{Lemma \ref{lemma14ii}} combined with \eqref{hypRaggirR}.
\end{rem}
\indent We start showing that along the flow assumptions on the $L^2$-norm of the speed are inherited by the $H^1$-norm.
\begin{lem}\label{lemma4a4} Let $\Sigma_t$ be a solution to \eqref{generalflow513}, and suppose that \eqref{hypRaggirR}, \eqref{hyp45}, \eqref{hyp46}, \eqref{hypP1219} and \eqref{hyp48} hold. Then
\begin{equation*}
\frac{d}{dt}\int_{\Sigma_t} (\mathcal H-\hbar)^2 \ d\mu\leq -C\sigma\int_{\Sigma_t} |\nabla \mathcal H|^2 \ d\mu+C\sigma^{-1}\int_{\Sigma_t} (\mathcal H-\hbar)^2 \ d\mu
\end{equation*}
\end{lem}
\begin{proof} Using Lemma \ref{evolution41} and equation \eqref{evHcal} we get
\begin{equation}\label{eq411}
\begin{aligned}
\frac{d}{dt}\int_{\Sigma_t} (\mathcal H-\hbar)^2 \ d\mu &=2\int_{\Sigma_t} \left(\frac{\partial \mathcal H}{\partial t}-\dot\hbar\right)(\mathcal H-\hbar) \ d\mu-\int_{\Sigma_t} H\left(\mathcal H-\hbar\right)^3 \ d\mu\\
&=2\int_{\Sigma_t} \left(\frac{H}{\mathcal H}\frac{\partial H}{\partial t}-\frac{P}{\mathcal H}\frac{\partial P}{\partial t}\right)(\mathcal H-\hbar) \ d\mu-\int_{\Sigma_t} H\left(\mathcal H-\hbar\right)^3 \ d\mu\\
&=2\int_{\Sigma_t} \left(\frac{H}{\mathcal H}\left(\Delta\mathcal H+(\mathcal H-\hbar)(|A|^2+\overline{\textup{Ric}}(\nu,\nu))\right)-\frac{P}{\mathcal H}\frac{\partial P}{\partial t}\right)(\mathcal H-\hbar) \ d\mu\\
& \quad -\int_{\Sigma_t} H\left(\mathcal H-\hbar\right)^3 \ d\mu.
\end{aligned}
\end{equation}
Note that
\begin{equation*}
\begin{aligned}
\int_{\Sigma_t} \frac{H}{\mathcal H}\Delta\mathcal H(\mathcal H-\hbar) \ d\mu&=-\int_{\Sigma_t} \nabla\left(\frac{H}{\mathcal H}\right)\cdot\nabla\mathcal H(\mathcal H-\hbar) \ d\mu-\int_{\Sigma_t} \frac{H}{\mathcal H}|\nabla\mathcal H|^2 \ d\mu\\
&=-\int_{\Sigma_t} \left(\frac{\nabla H}{\mathcal H}-\nabla\mathcal H \frac{H}{\mathcal H^2}\right)\cdot\nabla \mathcal H(\mathcal H-\hbar) \ d\mu-\int_{\Sigma_t} \frac{H}{\mathcal H}|\nabla\mathcal H|^2 \ d\mu\\
&=-\int_{\Sigma_t} \left(\frac{\nabla \mathcal H}{H}+\frac{P\nabla P}{H\mathcal H}-\nabla\mathcal H \frac{H}{\mathcal H^2}\right)\cdot\nabla \mathcal H(\mathcal H-\hbar) \ d\mu\\
& \quad -\int_{\Sigma_t} \frac{H}{\mathcal H}|\nabla\mathcal H|^2 \ d\mu\\
&=-\int_{\Sigma_t} |\nabla \mathcal H|^2\frac{\mathcal H-\hbar}{H} \ d\mu-\int_{\Sigma_t} \nabla P\cdot\nabla\mathcal H\frac{P(\mathcal H-\hbar)}{H\mathcal H} \ d\mu \\
& \quad +\int_{\Sigma_t} |\nabla\mathcal H|^2\frac{H(\mathcal H-\hbar)}{\mathcal H^2} \ d\mu-\int_{\Sigma_t} \frac{H}{\mathcal H}|\nabla\mathcal H|^2 \ d\mu.
\end{aligned}
\end{equation*}
By \eqref{hyp46}, \eqref{hypP1219} and \eqref{hyp48} we find that 
\begin{equation*}
-\int_{\Sigma_t} |\nabla \mathcal H|^2\frac{\mathcal H-\hbar}{H} \ d\mu+\int_{\Sigma_t} |\nabla\mathcal H|^2\frac{H(\mathcal H-\hbar)}{\mathcal H^2} \ d\mu-\int_{\Sigma_t} \frac{H}{\mathcal H}|\nabla\mathcal H|^2 \ d\mu\leq -C\sigma\int_{\Sigma_t} |\nabla\mathcal H|^2 \ d\mu.
\end{equation*}
On the other hand, using Young's inequality
\begin{equation*}
\begin{aligned}
\left|\int_{\Sigma_t} \nabla P\cdot\nabla\mathcal H\frac{P(\mathcal H-\hbar)}{H\mathcal H} \ d\mu\right|&\leq C\sigma^{-4}\int_{\Sigma_t} |\nabla\mathcal H||\mathcal H-\hbar| \ d\mu\\
&\leq \varepsilon \sigma \int_{\Sigma_t} |\nabla \mathcal H|^2 \ d\mu+C\sigma^{-9}\int_{\Sigma_t} (\mathcal H-\hbar)^2 \ d\mu.
\end{aligned}
\end{equation*}
Moreover, by \eqref{lemma22eq}, 
\begin{equation*}
\int_{\Sigma_t} \frac{H}{\mathcal H}\left(|A|^2+\overline{\textup{Ric}}(\nu,\nu)\right)(\mathcal H-\hbar)^2 \ d\mu\leq C\sigma^{-1}\int_{\Sigma_t} (\mathcal H-\hbar)^2 \ d\mu.
\end{equation*}
Combining equation \eqref{evPt} with the hypotheses we find 
\begin{equation}\label{evP2bis}
|\partial_tP|\leq C\sigma^{-5}|\mathcal H-\hbar|+C\sigma^{-5}|\nabla\mathcal H|,
\end{equation}
and thus
\begin{equation*}
\begin{aligned}
\left|\int_{\Sigma_t} \frac{P}{\mathcal H}\frac{\partial P}{\partial t}(\mathcal H-\hbar) \ d\mu\right|&\leq C\sigma^{-4}\int_{\Sigma_t} (\mathcal H-\hbar)^2 \ d\mu+C\sigma^{-4}\int_{\Sigma_t} |\mathcal H-\hbar||\nabla\mathcal H| \ d\mu\\
&\leq C\sigma^{-4}\int_{\Sigma_t} (\mathcal H-\hbar)^2 \ d\mu +\varepsilon \sigma\int_{\Sigma_t} |\nabla \mathcal H|^2 \ d\mu.
\end{aligned}
\end{equation*}
Combining these inequalities, we conclude the proof.
\end{proof}
\begin{lem}\label{lemma4a5} Let $\Sigma_t$ be a solution to \eqref{generalflow513}, and suppose that \eqref{hypRaggirR}, \eqref{hyp45}, \eqref{hyp46}, \eqref{hypP1219} and \eqref{hyp48} hold. Then
\begin{equation}\label{eq314}
\frac{d}{dt}\int_{\Sigma_t} |\nabla\mathcal H|^2 \ d\mu\leq C\sigma^{-2}\int_{\Sigma_t} |\nabla\mathcal H|^2 \ d\mu+C\sigma^{-3}\int_{\Sigma_t} (\mathcal H-\hbar)^2 \ d\mu
\end{equation}
\end{lem}
\begin{proof} By the evolution equation of the area form we have
\begin{equation*}
\frac{d}{dt}\int_{\Sigma_t} |\nabla\mathcal H|^2 \ d\mu=\int_{\Sigma_t} \partial_t|\nabla\mathcal H|^2 \ d\mu+\int_{\Sigma_t} H(\hbar-\mathcal H) |\nabla\mathcal H|^2 \ d\mu.
\end{equation*}
On the other hand, Lemma \ref{evolution41} leads to
\begin{equation*}
\begin{aligned}
\partial_t |\nabla\mathcal H|^2&=\partial_t\left(g^{ij}\nabla_i\mathcal H\nabla_j\mathcal H\right)=2(\mathcal H-\hbar)h^{ij}\nabla_i\mathcal H\nabla_j\mathcal H+2g^{ij}\nabla_i\left(\frac{\partial \mathcal H}{\partial t}\right)\nabla_j\mathcal H\\
&=2(\mathcal H-\hbar)h^{ij}\nabla_i\mathcal H\nabla_j\mathcal H+2g^{ij}\nabla_i\left(\frac{H}{\mathcal H}\frac{\partial H}{\partial t}-\frac{P}{\mathcal H}\frac{\partial P}{\partial t}\right)\nabla_j\mathcal H.
\end{aligned}
\end{equation*}
We first deal with the following term, which using integration by parts becomes
\begin{equation}\label{mainineq344}
\begin{aligned}
&2\int_{\Sigma_t} g^{ij}\nabla_i\left(\frac{H}{\mathcal H}\left(\Delta \mathcal H+(|A|^2+\overline{\textup{Ric}}(\nu,\nu))(\mathcal H-\hbar)\right)\right)\nabla_j \mathcal H \ d\mu\\
= \ & 2\int_{\Sigma_t} g^{ij}\nabla_i\left(\frac{H}{\mathcal H}\Delta\mathcal H\right) \nabla_j\mathcal H-2\int_\Sigma \frac{H}{\mathcal H}(|A|^2+\overline{\textup{Ric}}(\nu,\nu))(\mathcal H-\hbar)\Delta\mathcal H \ d\mu\\
= \ & -2\int_{\Sigma_t} \frac{H}{\mathcal H}(\Delta\mathcal H)^2-2\int_{\Sigma_t} \frac{H}{\mathcal H}(|A|^2+\overline{\textup{Ric}}(\nu,\nu))(\mathcal H-\hbar)\Delta\mathcal H \ d\mu\\
\leq \ & -C\sigma\int_{\Sigma_t} \left(\Delta\mathcal H\right)^2 \ d\mu-2\int_{\Sigma_t} \frac{H}{\mathcal H}(|A|^2+\overline{\textup{Ric}}(\nu,\nu))(\mathcal H-\hbar)\Delta\mathcal H \ d\mu\\
\leq \ & -C\sigma\int_{\Sigma_t} (\Delta \mathcal H)^2 \ d\mu+C\sigma^{-1}\int_{\Sigma_t} |\mathcal H-\hbar||\Delta \mathcal H| \ d\mu\\
\leq \ & -C\sigma \int_{\Sigma_t} \left(\Delta\mathcal H\right)^2 \ d\mu+C\sigma^{-3}\int_{\Sigma_t} (\mathcal H-\hbar)^2 \ d\mu,
\end{aligned}
\end{equation}
using Young's inequality. Since moreover the terms of the form $A*(\mathcal H-\hbar)*\nabla\mathcal H*\nabla\mathcal H$ are easily bounded by hypothesis \eqref{hyp48}, we conclude the proof remarking that, using integration by parts,
\begin{equation*}
\begin{aligned}
\left|2\int_{\Sigma_t} g^{ij}\nabla_i\left(\frac{P}{\mathcal H}\frac{\partial P}{\partial t}\right)\nabla_j\mathcal H \ d\mu\right|&=2\left|\int_{\Sigma_t} \frac{P}{\mathcal H}\left(\partial_t P\right)\Delta\mathcal H \ d\mu\right|\leq C\sigma\int_{\Sigma_t} |\partial_t P||\Delta\mathcal H| \ d\mu\\
&\leq C\sigma^{-4}\int_{\Sigma_t} (|\mathcal H-\hbar|+|\nabla\mathcal H|)|\Delta\mathcal H| \ d\mu,
\end{aligned}
\end{equation*}
which is easily absorbed by \eqref{mainineq344} using again Young's inequality.
\end{proof}
\begin{cor}\label{corH1ev}
Let $\Sigma_t$ be a solution to \eqref{generalflow513} and suppose that \eqref{hypRaggirR}, \eqref{hyp45}, \eqref{hyp46}, \eqref{hypP1219} and \eqref{hyp48} hold. Suppose moreover that
\begin{equation*}
\|\mathcal H-\hbar\|_{L^2(\Sigma_t)}\leq \cin \sigma^{-3},
\end{equation*}
for some $\cin>0$ and for every $t\in [0,T]$ and that
\begin{equation*}
\|\nabla\mathcal H\|_{L^2(\Sigma_0)}<\cin \sigma^{-3},
\end{equation*}
at the initial time. Then, there exists $\cin'>0$, only depending on $\cin$ and $C$, and $\sigma_0=\sigma_0(\cin,C)>1$, such that, if $\sigma>\sigma_0$, 
\begin{equation*}
\|\nabla\mathcal H\|_{L^2(\Sigma_t)}<\cin'\sigma^{-3},
\end{equation*}
for every $t\in[0,T]$.
\end{cor}
\begin{proof} Define the function $f(t):= \|\mathcal H-\hbar\|_{L^2(\Sigma_t)}^2+\sigma^2\|\nabla \mathcal H\|_{L^2(\Sigma_t)}^2$. Then, combining Lemma \ref{lemma4a4} and Lemma \ref{lemma4a5},
\begin{equation*}
\begin{aligned}
f'(t)\leq &-C\sigma\int_{\Sigma_t} |\nabla \mathcal H|^2 \ d\mu+C\sigma^{-1}\int_{\Sigma_t} (\mathcal H-\hbar)^2 \ d\mu\\
&+C\int_{\Sigma_t} |\nabla\mathcal H|^2 \ d\mu+C\sigma^{-1}\int_{\Sigma_t} (\mathcal H-\hbar)^2 \ d\mu\\
&\leq -C'\sigma^{-1}f(t)+C\sigma^{-1}\int_{\Sigma_t} (\mathcal H-\hbar)^2 \ d\mu
\end{aligned}
\end{equation*}
for $\sigma$ suitably large with respect to $C$. Since by the hypotheses $\|\mathcal H-\hbar\|_{L^2(\Sigma_t)}\leq \cin \sigma^{-3}$ for every $t\in [0,T]$ and $f(0)<\cin'\sigma^{-4}$, we can find a constant $\cin''>0$, only depending on $\cin$ and $C$ such that 
\begin{equation*}
f(t)<\cin''\sigma^{-4}
\end{equation*}
for every $t\in [0,T]$. Thus $\|\nabla \mathcal H\|_{L^2(\Sigma_t)}<\cin''\sigma^{-3}$ for every $t\in[0,T]$.
\end{proof}
\subsection{$W^{1,4}$-estimates} In this subsection, together with the assumptions of Section \ref{section41}, we will assume that 
\begin{equation}\label{addassumption419}
\|\mathcal H-\hbar\|_{L^2(\Sigma_t)}\leq \cin \sigma^{-3}, \qquad \|\nabla\mathcal H\|_{L^2(\Sigma_t)}\leq \cin \sigma^{-3},
\end{equation}
and 
\begin{equation}\label{addassumption420}
\left|\frac{H}{\mathcal H}-\sinh \hat r_t-\frac{m}2\right|\leq \cin \sigma^{-\delta},
\end{equation}
\begin{equation*}
\textup{Ric}^{\Sigma_t}\geq 0,
\end{equation*}
for every $t\in [0,T]$. We will dedicate the next Section to showing that \eqref{addassumption419} and \eqref{addassumption420} are preserved. 
\begin{lem}\label{lemma37} Let $\Sigma_t$ be a solution to \eqref{generalflow513} satisfying \eqref{hypRaggirR}, \eqref{hyp45}, \eqref{hyp46}, \eqref{hypP1219}, \eqref{hyp48}, \eqref{addassumption419} and \eqref{addassumption420}. There exist a universal constant $C>0$ and a radius $\sigma_0=\sigma_0(\cin, C)>1$ such that, if $\sigma>\sigma_0$, then
\begin{equation*}
\begin{aligned}
\frac{d}{dt}\int_{\Sigma_t} (\mathcal H-\hbar)^4 \ d\mu\leq &-C\sigma\int_{\Sigma_t} |\nabla\mathcal H|^2 (\mathcal H-\hbar)^2 \ d\mu+C\sigma^{-1}\int_{\Sigma_t} (\mathcal H-\hbar)^4 \ d\mu\\
&+(|\dot\hbar|+C\sigma^{-4})\int_{\Sigma_t} |\mathcal H-\hbar|^3 \ d\mu,
\end{aligned}
\end{equation*}
for every $t\in [0,T]$.
\end{lem}
\begin{proof} By Lemma \ref{evolution41}, we have 
\begin{equation*}
\frac{d}{dt}\int_{\Sigma_t} (\mathcal H-\hbar)^4 \ d\mu=4\int_{\Sigma_t} \left(\frac{\partial \mathcal H}{\partial t}-\dot\hbar\right)(\mathcal H-\hbar)^3 \ d\mu-\int_{\Sigma_t} H(\mathcal H-\hbar)^5 \ d\mu.
\end{equation*}
By the evolution equation \eqref{evHcal} we then have
\begin{equation*}
\begin{aligned}
4\int_{\Sigma_t} \frac{\partial \mathcal H}{\partial t}(\mathcal H-\hbar)^3 \ d\mu= \ & 4\int_{\Sigma_t} \left(\frac{H}{\mathcal H}\frac{\partial H}{\partial t}-\frac{P}{\mathcal H}\frac{\partial P}{\partial t}\right)(\mathcal H-\hbar)^3 \ d\mu
\\
= \ & 4\int_{\Sigma_t} \left(\frac{H}{\mathcal H}\left(\Delta\mathcal H+(|A|^2+\overline{\textup{Ric}}(\nu,\nu))(\mathcal H-\hbar)\right)-\frac{P}{\mathcal H}\frac{\partial P}{\partial t}\right)(\mathcal H-\hbar)^3 \ d\mu.
\end{aligned}
\end{equation*}
We first remark that, using integration by parts and \eqref{hyp46},
\begin{equation*}
\begin{aligned}
4\int_{\Sigma_t} \frac{H}{\mathcal H} \Delta \mathcal H (\mathcal H-\hbar)^3 \ d\mu &=-4\int_{\Sigma_t} \nabla\left(\frac{H}{\mathcal H}\right) \cdot\nabla\mathcal H (\mathcal H-\hbar)^3 \ d\mu-12\int_{\Sigma_t} \frac{H}{\mathcal H}|\nabla\mathcal H|^2(\mathcal H-\hbar)^2 \ d\mu\\
&\leq -C\sigma\int_{\Sigma_t} |\nabla\mathcal H|^2(\mathcal H-\hbar)^2\ d\mu+\left|4\int_{\Sigma_t} \nabla\left(\frac{H}{\mathcal H}\right) \cdot\nabla\mathcal H (\mathcal H-\hbar)^3 \ d\mu\right|.
\end{aligned}
\end{equation*}
By the identity 
\begin{equation*}
\nabla\left(\frac{H}{\mathcal H}\right)=\frac{\mathcal H\nabla H-H\nabla\mathcal H}{\mathcal H^2}=\frac{\nabla H}{\mathcal H}-\frac{H}{\mathcal H^2}\nabla \mathcal H=\frac{\nabla\mathcal H}{H}+\frac{P}{\mathcal H H}\nabla P-\frac{H}{\mathcal H^2}\nabla\mathcal H,
\end{equation*}
we moreover get, also using the bounds for $P$ and $\nabla P$,
\begin{equation*}
\begin{aligned}
\left|4\int_{\Sigma_t} \nabla\left(\frac{H}{\mathcal H}\right) \cdot\nabla\mathcal H (\mathcal H-\hbar)^3 \ d\mu\right|\leq \ & C\int_{\Sigma_t} |\nabla\mathcal H|^2|\mathcal H-\hbar|^3 \ d\mu
\\
& +C\sigma\int_{\Sigma_t} |\nabla P||\mathcal H-\hbar|^3 \ d\mu\\
& +C\sigma^2\int_{\Sigma_t} |\nabla\mathcal H|^2|\mathcal H-\hbar|^3 \ d\mu.
\end{aligned}
\end{equation*}
Since by the hypotheses $|\mathcal H-\hbar|\leq \varepsilon \sigma^{-1}$ and $|\nabla P|\leq C\sigma^{-5}$, we get 
\begin{equation*}
4\int_{\Sigma_t} \frac{H}{\mathcal H} \Delta \mathcal H (\mathcal H-\hbar)^3 \ d\mu\leq -(C-\varepsilon)\sigma\int_{\Sigma_t} |\nabla\mathcal H|^2(\mathcal H-\hbar)^2 \ d\mu+C\sigma^{-4}\int_{\Sigma_t} |\mathcal H-\hbar|^3 \ d\mu.
\end{equation*}
We remark that furthermore, using \eqref{lemma22eq} and \eqref{hyp48},
\begin{equation*}
\left|\int_{\Sigma_t} \frac{H}{\mathcal H}\left(|A|^2+\overline{\textup{Ric}}(\nu,\nu)\right)(\mathcal H-\hbar)^4 \ d\mu\right|+\left|\int_{\Sigma_t} H(\mathcal H-\hbar)^5 \ d\mu\right|\leq C\sigma^{-1}\int_{\Sigma_t} (\mathcal H-\hbar)^4 \ d\mu.
\end{equation*}
We conclude the proof estimating 
\begin{equation*}
\begin{aligned}
\left|\int_{\Sigma_t} \frac{P}{\mathcal H}\partial_t P(\mathcal H-\hbar)^3 \ d\mu\right|&\leq C\sigma\int_{\Sigma_t} |\partial_tP||\mathcal H-\hbar|^3 \ d\mu\\
&\leq C\sigma^{-4}\int_{\Sigma_t} (|\mathcal H-\hbar|+|\nabla\mathcal H|)|\mathcal H-\hbar|^3\ d\mu\\
&\leq C\sigma^{-1}\int_{\Sigma_t} (\mathcal H-\hbar)^4 \ d\mu+C\sigma^{-4}\int_{\Sigma_t} \left(|\nabla\mathcal H||\mathcal H-\hbar|\right)(\mathcal H-\hbar)^2 \ d\mu\\
&\leq C\sigma^{-1}\int_{\Sigma_t} (\mathcal H-\hbar)^4 \ d\mu + C\sigma^{-7}\int_{\Sigma_t} |\nabla\mathcal H|^2(\mathcal H-\hbar)^2 \ d\mu.
\end{aligned}
\end{equation*}
In conclusion, we find
\begin{equation*}
\begin{aligned}
\frac{d}{dt}\int_{\Sigma_t} (\mathcal H-\hbar)^4 \ d\mu\leq &-C\sigma\int_{\Sigma_t} |\nabla\mathcal H|^2 (\mathcal H-\hbar)^2 \ d\mu+C\sigma^{-1}\int_{\Sigma_t} (\mathcal H-\hbar)^4 \ d\mu\\
&+(|\dot\hbar|+C\sigma^{-4})\int_{\Sigma_t} |\mathcal H-\hbar|^3 \ d\mu. 
\end{aligned}
\end{equation*}
\end{proof}
\begin{lem}\label{lemma38} Let $\Sigma_t$ be a solution to \eqref{generalflow513} satisfying \eqref{hypRaggirR}, \eqref{hyp45}, \eqref{hyp46}, \eqref{hypP1219}, \eqref{hyp48}, \eqref{addassumption419} and \eqref{addassumption420}. There exist a universal constant $C>0$ and a radius $\sigma_0=\sigma_0(\cin, C,m)>1$ such that, if $\sigma>\sigma_0$, then
\begin{equation*}
\frac{d}{dt}\int_{\Sigma_t} |\nabla\mathcal H|^4 \ d\mu\leq -C\sigma\int_{\Sigma_t} |\nabla^2\mathcal H|^2|\nabla\mathcal H|^2 \ d\mu+C\sigma^{-1}\int_{\Sigma_t} |\nabla\mathcal H|^4 \ d\mu+C\sigma^{-5}\int_{\Sigma_t} (\mathcal H-\hbar)^4 \ d\mu.
\end{equation*}
\end{lem}
\begin{proof} By the evolution equation of the area form we have
\begin{equation}\label{evolutionvolformL4}
\frac{d}{dt}\int_{\Sigma_t} |\nabla \mathcal H|^4 \ d\mu=2\int_{\Sigma_t}\left(\frac{\partial}{\partial t}|\nabla \mathcal H|^2\right)|\nabla\mathcal H|^2 \ d\mu+\int_{\Sigma_t} H(\hbar-\mathcal H)|\nabla\mathcal H|^4 \ d\mu.
\end{equation}
On the other hand, we find
\begin{equation*}
\begin{aligned}
\partial_t |\nabla\mathcal H|^2&=\partial_t\left(g^{ij}\nabla_i\mathcal H\nabla_j\mathcal H\right)=2(\mathcal H-\hbar)h^{ij}\nabla_i\mathcal H\nabla_j\mathcal H+2g^{ij}\nabla_i\left(\frac{\partial \mathcal H}{\partial t}\right)\nabla_j\mathcal H\\
&=2(\mathcal H-\hbar)h^{ij}\nabla_i\mathcal H\nabla_j\mathcal H+2g^{ij}\nabla_i\left(\frac{H}{\mathcal H}\frac{\partial H}{\partial t}-\frac{P}{\mathcal H}\frac{\partial P}{\partial t}\right)\nabla_j\mathcal H\\
&=2(\mathcal H-\hbar)h^{ij}\nabla_i\mathcal H\nabla_j\mathcal H\\
&\quad +2g^{ij}\nabla_i\left(\frac{H}{\mathcal H}\left(\Delta\mathcal H+(|A|^2+\overline{\textup{Ric}}(\nu,\nu))(\mathcal H-\hbar)\right)-\frac{P}{\mathcal H}\frac{\partial P}{\partial t}\right)\nabla_j\mathcal H
\end{aligned}
\end{equation*}
and we insert this in \eqref{evolutionvolformL4}. Since we easily have
\begin{equation}\label{easyterm1}
\left|\int_{\Sigma_t} (\mathcal H-\hbar)|\nabla\mathcal H|^4 \ d\mu\right|\leq C\sigma^{-2}\int_{\Sigma_t} |\nabla\mathcal H|^4 \ d\mu
\end{equation}
by the assumption \eqref{hyp48}, proceeding with the other terms we find 
\begin{equation}\label{decompositionintegral424}
\begin{aligned}
4\int_{\Sigma_t} g^{ij}\nabla_i\left(\frac{H}{\mathcal H}\Delta\mathcal H\right) \nabla_j\mathcal H |\nabla\mathcal H|^2 \ d\mu=
\ & 4\sinh \hat r_t\int_{\Sigma_t} g^{ij}\nabla_i\left(\Delta\mathcal H\right) \nabla_j\mathcal H |\nabla\mathcal H|^2 \ d\mu\\
&+4\int_{\Sigma_t} g^{ij}\nabla_i\left(\left(\frac{H}{\mathcal H}-\sinh \hat r_t\right)\Delta\mathcal H\right) \nabla_j\mathcal H |\nabla\mathcal H|^2 \ d\mu,
\end{aligned}
\end{equation}
and now we estimate these new terms with the help of assumption \eqref{addassumption420}. We employ Bochner's identity in the first addend of the right hand side of \eqref{decompositionintegral424}, i.e. 
\begin{equation*}
2g^{ij}\nabla_i\Delta\mathcal H\nabla_j\mathcal H=\Delta|\nabla\mathcal H|^2-2|\nabla^2\mathcal H|^2-2\textup{Ric}^\Sigma(\nabla\mathcal H,\nabla\mathcal H)\leq \Delta|\nabla\mathcal H|^2-2|\nabla^2\mathcal H|^2,
\end{equation*}
where we also used the assumption $\textup{Ric}^{\Sigma_t}\geq 0$. Using this and integration by parts in \eqref{decompositionintegral424}, we find 
\begin{equation*}
\begin{aligned}
4\int_{\Sigma_t} g^{ij}\nabla_i\left(\frac{H}{\mathcal H}\Delta\mathcal H\right) \nabla_j\mathcal H |\nabla\mathcal H|^2 \ d\mu\leq  &-2\sinh \hat r_t\int_{\Sigma_t} |\nabla|\nabla \mathcal H|^2|^2 \ d\mu
\\
&-4\sinh \hat r_t\int_{\Sigma_t} |\nabla^2\mathcal H|^2|\nabla\mathcal H|^2 \ d\mu\\
&+Cm\int_{\Sigma_t} |\nabla^2\mathcal H|^2|\nabla\mathcal H|^2 \ d\mu,
\end{aligned}
\end{equation*}
using \eqref{addassumption420} for $\sigma$ suitably large.\\
\indent Now we deal with the following term using integration by parts, obtaining
\begin{equation}\label{ineq328}
\begin{aligned}
&\int_{\Sigma_t} g^{ij}\nabla_i\left(\frac{H}{\mathcal H}\left((|A|^2+\overline{\textup{Ric}}(\nu,\nu)(\mathcal H-\hbar)\right)\right)\nabla_j\mathcal H|\nabla \mathcal H|^2 \ d\mu \\
\leq \ & C\sigma^{-1}\int_{\Sigma_t} |\mathcal H-\hbar||\nabla^2\mathcal H||\nabla\mathcal H|^2 \ d\mu\\
\leq \ & \varepsilon\sigma \int_{\Sigma_t} |\nabla^2\mathcal H|^2|\nabla\mathcal H|^2 \ d\mu+C\sigma^{-3}\int_{\Sigma_t} (\mathcal H-\hbar)^2|\nabla\mathcal H|^2 \ d\mu
\\
\leq \ &\varepsilon\sigma \int_{\Sigma_t} |\nabla^2\mathcal H|^2|\nabla\mathcal H|^2 \ d\mu+C\sigma^{-5}\int_{\Sigma_t} (\mathcal H-\hbar)^4 \ d\mu+C\sigma^{-1}\int_{\Sigma_t} |\nabla\mathcal H|^4 \ d\mu,
\end{aligned}
\end{equation}
using also \eqref{lemma22eq}. 
We note that the term involving $(\mathcal H-\hbar)h^{ij}$ are easily bounded as in \eqref{easyterm1}. We conclude applying integration by parts to
\begin{equation}\label{ineq329}
\begin{aligned}
&\int_{\Sigma_t} g^{ij}\nabla_i\left(\frac{P}{\mathcal H}\partial_tP\right)\nabla_j \mathcal H|\nabla\mathcal H|^2 \ d\mu
\\
=&-\int_{\Sigma_t}\frac{P}{\mathcal H}\partial_t P\nabla\cdot\left(\nabla\mathcal H|\nabla\mathcal H|^2\right) \ d\mu
\\
\leq & \ C\sigma\int_{\Sigma_t} |\partial_t P||\nabla^2\mathcal H||\nabla\mathcal H|^2 \ d\mu
\\ 
\leq & \ C\sigma^{-4}\int_{\Sigma_t} (|\mathcal H-\hbar|+|\nabla\mathcal H|)|\nabla^2\mathcal H||\nabla\mathcal H|^2 \ d\mu
\\
\leq & \ C\sigma^{-4}\int_{\Sigma_t} |\mathcal H-\hbar||\nabla^2\mathcal H||\nabla\mathcal H|^2 \ d\mu+C\sigma^{-4}\int_{\Sigma_t} |\nabla\mathcal H||\nabla^2\mathcal H||\nabla\mathcal H|^2 \ d\mu,
\end{aligned}
\end{equation}
where we used inequality \eqref{evP2bis}. Notice that the first addend in the last inequality of \eqref{ineq329} can be dealt as in \eqref{ineq328}, while  
\begin{equation*}
\sigma^{-4}\int_{\Sigma_t} |\nabla\mathcal H||\nabla^2\mathcal H||\nabla\mathcal H|^2 \ d\mu\leq \varepsilon\sigma\int_{\Sigma_t} |\nabla^2\mathcal H|^2|\nabla\mathcal H|^2 \ d\mu+C\sigma^{-2}\int_{\Sigma_t} |\nabla\mathcal H|^4 \ d\mu.
\end{equation*}
In conclusion, since $\frac12\sigma\leq \sinh \hat r_t\leq \frac32\sigma$, we find that, for $\sigma$ suitably large, 
\begin{equation*}
\frac{d}{dt}\int_{\Sigma_t} |\nabla\mathcal H|^4 \ d\mu\leq -C\sigma\int_{\Sigma_t} |\nabla^2\mathcal H|^2|\nabla\mathcal H|^2+C\sigma^{-1}\int_{\Sigma_t} |\nabla\mathcal H|^4 \ d\mu+C\sigma^{-5}\int_{\Sigma_t} (\mathcal H-\hbar)^4 \ d\mu.
\end{equation*}
\end{proof}
In order to combine Lemma \ref{lemma37} and Lemma \ref{lemma38}, we estimate the evolution of the integral mean $\hbar(t)$.
\begin{lem}\label{lemma39} Let $\Sigma_t$ be a solution to \eqref{generalflow513} satisfying \eqref{hypRaggirR}, \eqref{hyp45}, \eqref{hyp46}, \eqref{hypP1219}, \eqref{hyp48}, \eqref{addassumption419} and \eqref{addassumption420}. There exist a universal constant $C>0$ and a radius $\sigma_0=\sigma_0(\cin, C)>1$ such that, if $\sigma>\sigma_0$, then
\begin{equation*}
|\dot\hbar(t)|\leq C\sigma^{-5},
\end{equation*}
for every $t\in [0,T]$.
\end{lem}
\begin{proof}
By the evolution equations we find the identity
\begin{equation*}
\begin{aligned}
|\Sigma_t|\dot\hbar+\hbar \int_{\Sigma_t} H(\hbar-\mathcal H) \ d\mu &=\frac{d}{dt}\left(|\Sigma_t|\hbar\right)\\
&=\frac{d}{dt}\int_{\Sigma_t} \mathcal H \ d\mu\\
&=\int_{\Sigma_t} \frac{\partial \mathcal H}{\partial t} \ d\mu+\int_{\Sigma_t} \mathcal H H(\hbar-\mathcal H) \ d\mu.
\end{aligned}
\end{equation*}
Using \eqref{evHcal}, we get
\begin{equation*}
\begin{aligned}
|\Sigma_t|\dot\hbar=&\int_{\Sigma_t} \frac{H}{\mathcal H}\left(\Delta \mathcal H+(|A|^2+\overline{\textup{Ric}}(\nu,\nu))(\mathcal H-\hbar)\right)-\frac{P}{\mathcal H}\frac{\partial P}{\partial t} \ d\mu\\
\quad &+\int_{\Sigma_t} \mathcal H H(\hbar-\mathcal H) \ d\mu -\hbar \int_{\Sigma_t} H(\hbar-\mathcal H) \ d\mu\\
=&-\int_{\Sigma_t} \nabla \left(\frac{H}{\mathcal H}\right)\cdot\nabla\mathcal H \ d\mu-\int_{\Sigma_t}\frac{P}{\mathcal H}\frac{\partial P}{\partial t} \ d\mu\\
&+\int_{\Sigma_t} \frac{H}{\mathcal H}\left(|A|^2+\overline{\textup{Ric}}(\nu,\nu)\right)(\mathcal H-\hbar) \ d\mu+\int_{\Sigma_t} \mathcal H H(\hbar-\mathcal H) \ d\mu -\hbar \int_{\Sigma_t} H(\hbar-\mathcal H) \ d\mu.
\end{aligned}
\end{equation*}
On the other hand, computations show that 
\begin{equation*}
\begin{aligned}
-\int_{\Sigma_t} \nabla\left(\frac{H}{\mathcal H}\right) \cdot \nabla\mathcal H \ d\mu&=-\int_{\Sigma_t} \frac{\mathcal H\nabla H-H\nabla\mathcal H}{\mathcal H^2}\cdot\nabla\mathcal H \ d\mu\\
&=-\int_{\Sigma_t} \frac{\nabla H\cdot \nabla \mathcal H}{\mathcal H} \ d\mu +\int_{\Sigma_t} \frac{H}{\mathcal H^2}|\nabla\mathcal H|^2 \ d\mu\\
&=-\int_{\Sigma_t} \frac{|\nabla\mathcal H|^2}{H}+\frac{P\nabla P\cdot\nabla \mathcal H}{H\mathcal H} \ d\mu +\int_{\Sigma_t} \frac{H}{\mathcal H^2}|\nabla\mathcal H|^2 \ d\mu\\
&\leq C\sigma^2\|\nabla \mathcal H\|_{L^2(\Sigma_t)}^2+C\sigma\left(\int_{\Sigma_t} |\nabla P|^2+|\nabla\mathcal H|^2 \ d\mu\right)\\
&\leq C\sigma^2\|\nabla \mathcal H\|_{L^2(\Sigma_t)}^2+C\sigma^{-7}\leq C\sigma^{-4}.
\end{aligned}
\end{equation*}
On the other hand, by Lemma \ref{lemma33} combined with \eqref{hyp48}, we find
\begin{equation}\label{decayconsRic}
\begin{aligned}
\left|\int_{\Sigma_t} \frac{H}{\mathcal H}\left(|A|^2+\overline{\textup{Ric}}(\nu,\nu)\right)(\mathcal H-\hbar) \ d\mu-\int_{\Sigma_t} \frac{H\mathcal H}2(\mathcal H-\hbar) \ d\mu\right|& \leq C\sigma^{-3}\int_{\Sigma_t} \left|\frac{H}{\mathcal H}\right||\mathcal H-\hbar| \ d\mu\\
& \leq C\sigma^{-1}\|\mathcal H-\hbar\|_{L^2({\Sigma_t})}\\
& \leq C\cin \sigma^{-4}.
\end{aligned}
\end{equation}
We thus estimate the terms 
\begin{equation}\label{eqord-3}
\left|\int_{\Sigma_t} \frac{H\mathcal H}2\left(\mathcal H-\hbar\right) \ d\mu+\int_{\Sigma_t} \mathcal H H(\hbar-\mathcal H) \ d\mu -\hbar \int_{\Sigma_t} H(\hbar-\mathcal H) \ d\mu\right|\leq C\|\mathcal H-\hbar\|_{L^2({\Sigma_t})}\leq C\cin\sigma^{-3}.
\end{equation}
We finally estimate, thanks to \eqref{evP2bis}, 
\begin{equation*}
\begin{aligned}
\left|\int_{\Sigma_t} \frac{P}{\mathcal H}\frac{\partial P}{\partial t} \ d\mu\right|
&\leq C\sigma^{-4}\int_{\Sigma_t} |\mathcal H-\hbar| +|\nabla \mathcal H| \ d\mu\\
&\leq C\sigma^{-3}\left(\|\mathcal H-\hbar\|_{L^2(\Sigma_t)}+\|\nabla \mathcal H\|_{L^2(\Sigma_t)}\right)\\
&\leq C\cin\sigma^{-6}.
\end{aligned}
\end{equation*}
Combining the pieces together we conclude that  
\begin{equation*}
|\Sigma_t||\dot\hbar|\leq C\cin\sigma^{-3}.
\end{equation*}
\end{proof}
\begin{cor}\label{cor310}
Let $\Sigma_t$ be a solution to \eqref{generalflow513} satisfying \eqref{hypRaggirR}, \eqref{hyp45}, \eqref{hyp46}, \eqref{hypP1219}, \eqref{hyp48}, \eqref{addassumption419} and \eqref{addassumption420}. There exist a universal constant $C>0$ and a radius $\sigma_0=\sigma_0(\cin, C)>1$ such that, if $\sigma>\sigma_0$, then
\begin{equation*}
(|\dot\hbar|+C\sigma^{-4})\int_{\Sigma_t} |\mathcal H-\hbar|^3 \ d\mu\leq C\cin \sigma^{-10}\|\mathcal H-\hbar\|_{L^\infty(\Sigma_t)},
\end{equation*}
for every $t\in[0,T]$.
\end{cor}
\begin{proof}
The term satisfies 
\begin{equation*}
\begin{aligned}
(|\dot\hbar|+C\sigma^{-4})\int_{\Sigma_t} |\mathcal H-\hbar|^3 \ d\mu & \leq C\sigma^{-4}\|\mathcal H-\hbar\|_{L^\infty(\Sigma_t)}\int_{\Sigma_t} (\mathcal H-\hbar)^2 \ d\mu\\
& \leq C\cin \sigma^{-10}\|\mathcal H-\hbar\|_{L^\infty(\Sigma_t)}.
\end{aligned}
\end{equation*}
\end{proof}
\begin{lem}\label{evW14} Let $\Sigma_t$ be a solution to \eqref{generalflow513} satisfying \eqref{hypRaggirR}, \eqref{hyp45}, \eqref{hyp46}, \eqref{hypP1219}, \eqref{hyp48}, \eqref{addassumption419} and \eqref{addassumption420}. Suppose moreover that there exists $B_\infty>0$ such that 
\begin{equation}\label{Linftyhyp434}
\|\mathcal H-\hbar\|_{L^\infty(\Sigma_t)}\leq B_\infty \sigma^{-2},
\end{equation}
for all $t\in [0,T]$. Suppose moreover that 
\begin{equation*}
\sigma^{-4}\|\mathcal H-\hbar\|_{L^4(\Sigma_0)}^4+\|\nabla\mathcal H\|_{L^4(\Sigma_0)}^4\leq \cin \sigma^{-14}.
\end{equation*}
There exist a universal constant $C>0$ and a radius $\sigma_0=\sigma_0(\cin, C,B_\infty)>1$ such that, if $\sigma>\sigma_0$ and $B_2>C$, then
\begin{equation*}
\|\mathcal H-\hbar\|_{W^{1,4}(\Sigma_t)}\leq B_2\sigma^{-\frac52},
\end{equation*}
for every $t\in [0,T]$.
\end{lem}
\begin{proof} 
First of all, we remark that using integration by parts 
\begin{equation*}
\begin{aligned}
\sigma^{-1}\int_{\Sigma_t} |\nabla\mathcal H|^4 \ d\mu &=\sigma^{-1}\int_{\Sigma_t} g^{ij}\nabla_i(\mathcal H-\hbar)\nabla_j\mathcal H|\nabla\mathcal H|^2 \ d\mu
\\
&=-\sigma^{-1}\int_{\Sigma_t} (\mathcal H-\hbar)\Delta \mathcal H|\nabla\mathcal H|^2 \ d\mu\\
&\quad -2\sigma^{-1}\int_{\Sigma_t} (\mathcal H-\hbar)g^{ij}\nabla_j\mathcal H g^{kl}\nabla_i\nabla_k\mathcal H\nabla_l\mathcal H \ d\mu
\\
&\leq (\sqrt 2+2)\sigma^{-1}\int_{\Sigma_t} |\mathcal H-\hbar||\nabla^2\mathcal H||\nabla\mathcal H|^2 \ d\mu\\
&\leq 2\int_{\Sigma_t} \left(\frac{\sigma^{-3}(\mathcal H-\hbar)^2}{\varepsilon}+\varepsilon\sigma|\nabla^2\mathcal H|^2\right)|\nabla\mathcal H|^2 \ d\mu.
\end{aligned}
\end{equation*}
Multiplying by $\varepsilon N/2$, we get 
\begin{equation}\label{equationuseful}
-N\sigma^{-3}\int_{\Sigma_t} (\mathcal H-\hbar)^2|\nabla\mathcal H|^2 \ d\mu\leq -\frac{\varepsilon N\sigma^{-1}}{2}\int_{\Sigma_t} |\nabla\mathcal H|^4 \ d\mu+N\sigma\varepsilon^2\int_{\Sigma_t} |\nabla^2\mathcal H|^2|\nabla\mathcal H|^2 \ d\mu.
\end{equation}
Notice that the hypotheses imply $\|\mathcal H-\hbar\|_{H^1(\Sigma_t)}\leq \cin \sigma^{-2}$ for every $t\in[0,T]$, and thus combining Lemmas \ref{lemma37}, \ref{lemma38} and \ref{lemma39} we get
\begin{equation*}
\begin{aligned}
&\frac{d}{dt}\left(N\sigma^{-4}\|\mathcal H-\hbar\|_{L^4(\Sigma_t)}^4+\|\nabla\mathcal H\|_{L^4(\Sigma_t)}^4\right)
\\
\leq &-CN\sigma^{-3}\int_{\Sigma_t} |\nabla \mathcal H|^2(\mathcal H-\hbar)^2 \ d\mu +CN\sigma^{-5}\int_{\Sigma_t} (\mathcal H-\hbar)^4 d\mu\\
&+C\cin N\sigma^{-14}\|\mathcal H-\hbar\|_{L^\infty(\Sigma_t)}-C\sigma\int_{\Sigma_t} |\nabla^2\mathcal H|^2 |\nabla\mathcal H|^2 \ d\mu+C\sigma^{-1}\int_{\Sigma_t} |\nabla \mathcal H|^4 \ d\mu\\
&+C\sigma^{-5}\int_{\Sigma_t} (\mathcal H-\hbar)^4 \ d\mu\\
\leq & -\frac{C\varepsilon N\sigma^{-1}}{2}\int_{\Sigma_t} |\nabla \mathcal H|^4 \ d\mu +CN\sigma \varepsilon^2 \int_{\Sigma_t} |\nabla^2\mathcal H|^2|\nabla \mathcal H|^2 \ d\mu +CN\sigma^{-7}\|\mathcal H-\hbar\|_{H^1(\Sigma_t)}^4\\
&+C\cin N\sigma^{-14}\|\mathcal H-\hbar\|_{L^\infty(\Sigma_t)}-C\sigma\int_{\Sigma_t} |\nabla^2\mathcal H|^2 |\nabla\mathcal H|^2 \ d\mu+C\sigma^{-1}\int_{\Sigma_t} |\nabla \mathcal H|^4 \ d\mu,
\end{aligned}
\end{equation*}
where we also used Sobolev's inequality and \eqref{equationuseful}. Using hypothesis \eqref{Linftyhyp434} and choosing $\varepsilon$ and $N$ so that $\varepsilon N$ is large but $N\varepsilon^2$ is small, we find that, setting $f_N(t):=N\sigma^{-4}\|\mathcal H-\hbar\|_{L^4(\Sigma_t)}^4+\|\nabla\mathcal H\|_{L^4(\Sigma_t)}^4$,
\begin{equation*}
\begin{aligned}
\frac{d}{dt}f_N(t)\leq &-C\sigma^{-1} \left(N\sigma^{-4}\|\mathcal H-\hbar\|_{L^4(\Sigma_t)}^4+\|\nabla\mathcal H\|_{L^4(\Sigma_t)}^4\right)\\
&+C\cin N\sigma^{-15}+C\cin NB_\infty\sigma^{-16}\\
\leq & -C\sigma^{-1} f_N(t)+C\cin N\sigma^{-15}.
\end{aligned}
\end{equation*}
for some universal $C>0$, if $\sigma$ is suitably large with respect to $B_\infty$. Since, by the hypotheses,
\begin{equation*}
f_N(0)\leq \cin\sigma^{-14},
\end{equation*}
we find that, for $B_2$ suitably large, $f_N(t)<B_2\sigma^{-14}$ for every $t\in [0,T]$. Thus, by definition of $f_N$, 
\begin{equation*}
\|\mathcal H-\hbar\|_{W^{1,4}(\Sigma_t)}\leq c(B_2)\sigma^{-\frac52},
\end{equation*}
for every $t\in [0,T]$.
\end{proof}
\begin{rem}
As a consequence, by \eqref{SobLinfty28}, we find 
\begin{equation}\label{Linfty439}
\|\mathcal H-\hbar\|_{L^\infty(\Sigma_t)}\leq c(B_2)\sigma^{-3}
\end{equation}
for every $t\in [0,T]$. Thus, the estimate on the $L^\infty$-norm of the speed has improved.
\end{rem}
\subsection{Roundness implies exponential decay}
\begin{lem}\label{expdecay413}
Let $\Sigma_t$ be a solution to \eqref{generalflow513} satisfying \eqref{hypRaggirR}, \eqref{hyp45}, \eqref{hyp46}, \eqref{hypP1219}, \eqref{hyp48}, \eqref{lemma22eq} and \eqref{addassumption420}. Suppose moreover that there exists $B_\infty>0$ such that 
\begin{equation*}
\|\mathcal H-\hbar\|_{L^\infty(\Sigma_t)}\leq B_\infty \sigma^{-3}, \qquad \left|h(t)-\left(2\sqrt{1+\frac1{\sinh^2 \hat r}}-\frac{m}{\sinh^3 \hat r}\right)\right|\leq B_\infty\sigma^{-3-\delta},
\end{equation*}
for every $t\in [0,T]$. There exists a radius $\sigma_0=\sigma_0(\cin, m, B_\infty)>1$ such that, if $\sigma>\sigma_0$, then
\begin{equation*}
\frac{d}{dt}\int_{\Sigma_t} (\mathcal H-\hbar)^2 \ d\mu\leq -\frac{10m}{9\sigma^2}\int_{\Sigma_t} (\mathcal H-\hbar)^2 \ d\mu,
\end{equation*}
for every $t\in [0,T]$.
\end{lem}
\begin{proof}
Proceeding as in \eqref{eq411}, we have 
\begin{equation*}
\begin{aligned}
\frac{d}{dt}\int_{\Sigma_t} (\mathcal H-\hbar)^2 \ d\mu &=2\int_{\Sigma_t} \frac{\partial H}{\partial t} \left(\mathcal H-\hbar\right) \ d\mu-\int_{\Sigma_t} H(\mathcal H-\hbar)^3 \ d\mu\\
&=2\int_{\Sigma_t} \frac{H}{\mathcal H}\left(\Delta \mathcal H+\left(|A|^2+\overline{\textup{Ric}}(\nu,\nu)\right)(\mathcal H-\hbar)\right)(\mathcal H-\hbar) \ d\mu\\
& \quad -2\int_{\Sigma_t} \frac{P}{\mathcal H}\frac{\partial P}{\partial t}\left(\mathcal H-\hbar\right) \ d\mu-\int_{\Sigma_t} H(\mathcal H-\hbar)^3 \ d\mu\\
&=2\sinh \hat r_t\int_{\Sigma_t} L_{\Sigma_t} (\mathcal H-\hbar)(\mathcal H-\hbar) \ d\mu\\
& \quad -2\int_{\Sigma_t} \frac{P}{\mathcal H}\frac{\partial P}{\partial t}\left(\mathcal H-\hbar\right) \ d\mu-\int_{\Sigma_t} H(\mathcal H-\hbar)^3 \ d\mu\\
& \quad +2\int_{\Sigma_t} \left(\frac{H}{\mathcal H}-\sinh \hat r_t\right)\left(\Delta \mathcal H+\left(|A|^2+\overline{\textup{Ric}}(\nu,\nu)\right)(\mathcal H-\hbar)\right)(\mathcal H-\hbar) \ d\mu,
\end{aligned}
\end{equation*}
using also Definition \ref{defstabop}. Note that, employing hypothesis \eqref{addassumption420},
\begin{equation*}
\begin{aligned}
&\int_{\Sigma_t} \left(\frac{H}{\mathcal H}-\sinh \hat r_t\right)\left(\Delta \mathcal H+\left(|A|^2+\overline{\textup{Ric}}(\nu,\nu)\right)(\mathcal H-\hbar)\right)(\mathcal H-\hbar) \ d\mu\\
= &-\int_{\Sigma_t} \left(\frac{H}{\mathcal H}-\sinh \hat r_t\right)|\nabla \mathcal H|^2 \ d\mu-\int_{\Sigma_t} \nabla\left(\frac{H}{\mathcal H}\right)\cdot \nabla\mathcal H \ (\mathcal H-\hbar) \ d\mu\\
&+\int_{\Sigma_t} \left(\frac{H}{\mathcal H}-\sinh \hat r_t\right)\left(|A|^2+\overline{\textup{Ric}}(\nu,\nu)\right)(\mathcal H-\hbar)^2 \ d\mu\\
\leq &-\frac{m}4\int_{\Sigma_t} |\nabla\mathcal H|^2 \ d\mu+\int_{\Sigma_t} \frac{|\mathcal H-\hbar|}{H}|\nabla\mathcal H|^2 \ d\mu +\int_{\Sigma_t} \frac{|\mathcal H-\hbar|}{\mathcal H^2}H|\nabla \mathcal H|^2 \ d\mu\\
&+\int_{\Sigma_t} \frac{|P|}{H\mathcal H}|\mathcal H-\hbar||\nabla\mathcal H||\nabla P| \ d\mu+\frac{m+\epsilon}2\left(\frac{(2+\epsilon)^2}{2\sinh^2 \hat r_t}+\epsilon\right)\int_{\Sigma_t} (\mathcal H-\hbar)^2 \ d\mu,
\end{aligned}
\end{equation*}
for every $\epsilon>0$, where in the latter line we combined \eqref{addassumption420}, \eqref{lemma22eq} and Lemma \ref{lem111}, provided $\sigma$ is suitably large with respect to $\epsilon^{-1}$. By the decay of $|\mathcal H-\hbar|$ and $|\nabla P|$, we conclude that 
\begin{equation*}
\begin{aligned}
&\int_{\Sigma_t} \left(\frac{H}{\mathcal H}-\sinh \hat r_t\right)\left(\Delta \mathcal H+\left(|A|^2+\overline{\textup{Ric}}(\nu,\nu)\right)(\mathcal H-\hbar)\right)(\mathcal H-\hbar) \ d\mu\\
\leq & \ -\frac{m}8\int_{\Sigma_t} |\nabla\mathcal H|^2 \ d\mu+\frac{m+\epsilon}2\left(\frac{(2+\epsilon)^2}{2\sinh^2 \hat r_t}+\epsilon\right)\int_{\Sigma_t} (\mathcal H-\hbar)^2 \ d\mu,
\end{aligned}
\end{equation*}
for $\sigma$ suitably large. Since moreover, by \eqref{evP2bis},
\begin{equation*}
\begin{aligned}
\int_{\Sigma_t} \left|\frac{P}{\mathcal H}\right|\left|\frac{\partial P}{\partial t}\right||\mathcal H-\hbar| \ d\mu &\leq C\sigma^{-4}\int_{\Sigma_t} (\mathcal H-\hbar)^2 \ d\mu+C\sigma^{-4}\int_{\Sigma_t} |\mathcal H-\hbar||\nabla\mathcal H| \ d\mu\\
&\leq C\sigma^{-4}\int_{\Sigma_t} (\mathcal H-\hbar)^2 \ d\mu+C\sigma^{-4}\int_{\Sigma_t}|\nabla \mathcal H|^2 \ d\mu,
\end{aligned}
\end{equation*}
combining the pieces together we get, for $\sigma$ suitably large 
\begin{equation*}
\begin{aligned}
\frac{d}{dt}\int_{\Sigma_t} (\mathcal H-\hbar)^2 \ d\mu\leq \ & 2\sinh \hat r_t\int_{\Sigma_t} L_\Sigma (\mathcal H-\hbar)(\mathcal H-\hbar) \ d\mu\\
&+\left\{(m+\epsilon)\left(\frac{(2+\epsilon)^2}{2\sinh^2 \hat r_t}+\epsilon\right)+\epsilon\right\}\int_{\Sigma_t} (\mathcal H-\hbar)^2 \ d\mu\\
\leq \ & -\frac{5m}{\sinh^2 \hat r}\int_{\Sigma_t} (\mathcal H-\hbar)^2 \ d\mu\\
&+\left\{(m+\epsilon)\left(\frac{(2+\epsilon)^2}{2\sinh^2 \hat r_t}+\epsilon\right)+\epsilon\right\}\int_{\Sigma_t} (\mathcal H-\hbar)^2 \ d\mu,
\end{aligned}
\end{equation*}
where we used that 
\begin{equation*}
\left|\int_{\Sigma_t} H\left(\mathcal H-\hbar\right)^3 \ d\mu\right|\leq CB_\infty \sigma^{-3}\int_{\Sigma_t} (\mathcal H-\hbar)^2 \ d\mu
\end{equation*}
and Proposition \ref{spectralprop} in the last inequality. We conclude noticing that 
\begin{equation*}
\lim_{\epsilon\to0} (m+\epsilon)\left(\frac{(2+\epsilon)^2}{2\sinh^2 \hat r}+\epsilon\right)=\frac{2m}{\sinh^2 \hat r_t}<\frac{5m}{\sinh^2 \hat r_t},
\end{equation*}
and so there exists $\epsilon$ suitably small, and thus $\sigma_0$ suitably large, that for $\sigma>\sigma_0$, 
\begin{equation}\label{hypexpdecay}
\begin{aligned}
\frac{d}{dt}\int_{\Sigma_t} (\mathcal H-\hbar)^2 \ d\mu &\leq -\frac{5m}{2\sinh^2 \hat r_t}\int_{\Sigma_t} (\mathcal H-\hbar)^2 \ d\mu\\
&\leq -\frac{10m}{9\sigma^2}\int_{\Sigma_t} (\mathcal H-\hbar)^2 \ d\mu,
\end{aligned}
\end{equation}
which is the thesis.
\end{proof}
The following corollary shows the validity of \eqref{addassumption420} under the hypothesis of almost CMC-ness.
\begin{cor}\label{cor417}
Let $\Sigma_t$ be a solution to \eqref{generalflow513} satisfying \eqref{hypRaggirR}, \eqref{hyp45}, \eqref{hyp46}, \eqref{hypP1219}, \eqref{hyp48}, \eqref{lemma22eq}. Suppose moreover that $\|H-h\|_{L^\infty(\Sigma_t)}\leq B_\infty\sigma^{-3-\delta}$ for every $t\in [0,T]$ and for some $\delta>0$. Suppose moreover that there exists $C>0$ such that 
\begin{equation*}
\left|h(t)-2\sqrt{1+\frac1{\sinh^2 \hat r_t}}+\frac{m}{\sinh^3 \hat r_t}\right|\leq C \sigma^{-3-\delta}.
\end{equation*}
Then, there exist a constant $C_m'$, only depending on $C$ and $m$, and a radius $\sigma_0=\sigma_0(C,m,B_\infty)$ such that, if $\sigma>\sigma_0$, then 
\begin{equation*}
\left|\frac{H}{\mathcal H}-\sinh \hat r_t-\frac{m}2\right|\leq C_m'\sigma^{-\delta}.
\end{equation*}
\end{cor}
\begin{proof}
By the definition of $\hat r_t$ as in Proposition \ref{prp311partI} and the assumption on $\|H-h\|_{L^\infty(\Sigma_t)}$, we have 
\begin{equation*}
H^2(t)=\left(2+\frac1{\sinh^2 \hat r_t}-\frac{m}{\sinh^3 \hat r_t}+O(\sigma^{-3-\delta})\right)^2=4+\frac4{\sinh^2 \hat r_t}-\frac{4m}{\sinh^3 \hat r_t}+O(\sigma^{-3-\delta}).
\end{equation*}
Since $P^2=4+O(\sigma^{-5})$, we find 
\begin{equation*}
\begin{aligned}
\frac{H}{\mathcal H}&=\frac{2+\frac{1}{\sinh^2\hat r_t}-\frac{m}{\sinh^3\hat r_t}+O(\sigma^{-3-\delta})}{\sqrt{\frac4{\sinh^2 \hat r_t}-\frac{4m}{\sinh^3 \hat r_t}+O(\sigma^{-3-\delta})}}\\
&=\frac{\sinh \hat r_t}{2}\frac{2+\frac{1}{\sinh^2\hat r_t}-\frac{m}{\sinh^3\hat r_t}+O(\sigma^{-3-\delta})}{\sqrt{1-\frac{m}{\sinh \hat r_t}+O(\sigma^{-1-\delta})}}\\
&=\frac{\sinh \hat r_t}2\left(2+\frac{1}{\sinh^2\hat r_t}-\frac{m}{\sinh^3\hat r_t}+O(\sigma^{-3-\delta})\right)\left(1+\frac{m}{2\sinh \hat r_t}+O(\sigma^{-1-\delta})\right)\\
&=\frac{\sinh \hat r_t}{2}\left(2+\frac{m}{\sinh \hat r_t}+O(\sigma^{-1-\delta})\right)\\
&=\sinh \hat r_t+\frac{m}2+O(\sigma^{-\delta}).
\end{aligned}
\end{equation*}
\end{proof}
\subsection{Proof of Theorem \ref{mainthmintro}}
We begin the Section with an auxiliary Lemma which shows that the Euclidean radius of $\Sigma_t$ can be controlled along the flow.\\
\indent In the following, we set $r_t(x)=|F_t(x)|$, i.e. $r_t$ is the restriction of the coordinate $r$ to $\Sigma_t$.
\begin{lem}\label{lemma418}
Let $\Sigma_t$ be a solution to \eqref{generalflow513} satisfying \eqref{hypRaggirR}, \eqref{hyp45}, \eqref{hyp46}, \eqref{hypP1219}, \eqref{hyp48} and \eqref{hypexpdecay}. Assume moreover that 
\begin{equation*}
\|\mathcal H-\hbar\|_{L^2(\Sigma_0)}\leq \cin \sigma^{-3},
\end{equation*}
and that, for every $t\in [0,T]$,
\begin{equation*}
\|\mathcal H-\hbar\|_{W^{1,4}(\Sigma_t)}\leq B_2\sigma^{-\frac52}.
\end{equation*}
Suppose furthermore that $w(t):=w_{\hat r_t}$ defined as in Proposition \ref{prp311} satisfies $\|w_{\hat r}(t)\|_{C^{2,\frac12}(\Sigma_t)}\leq c\sigma^{-1}$, for some $c>0$. Then 
\begin{enumerate}[label=\textup{(\roman*)}]
\item there exists $B=B(C,\cin,B_2)>0$ such that, if $\sigma$ is suitably large,
\begin{equation}
\|\mathcal H-\hbar\|_{L^\infty(\Sigma_t)}\leq B\sigma^{-\frac{10}{3}}e^{-\frac{5m}{27\sigma^2}t}
\end{equation}
for every $t\in [0,T]$;
\item there exists $C>0$ and $\sigma_0=\sigma_0(\cin, C,m,B_2,c)>1$ such that, if $\sigma>\sigma_0$,
\begin{equation*}
|r_t-r_0|\leq C\sigma^{-1},
\end{equation*}
for every $t\in [0,T]$.
\end{enumerate}
\end{lem}
\begin{proof}
By the fundamental theorem of calculus we have 
\begin{equation}\label{ftc448}
\left|r_t-r_0\right|\leq |F_t-F_0|\leq \int_0^t |\partial_\tau F_\tau| \ d\tau\leq C\int_0^t|\mathcal H-\hbar| \ d\tau,
\end{equation}
where $C$ is a bound for the unit normal $\nu_\tau$ of $\Sigma_\tau$ with respect to the Euclidean metric. Remember also that, by Morrey's inequality, which is scale invariant,
\begin{equation*}
\left|\mathcal H(x)-\mathcal H(y)\right|\leq C\|\nabla \mathcal H\|_{L^4(\Sigma_t)}d_t(x,y)^\frac12,
\end{equation*}
for every $x,y\in \Sigma_t$. Thus, if $x_t^0\in \Sigma_t$ is such that $\mathcal H(x_t^0)-\hbar=\|\mathcal H-\hbar\|_{L^\infty(\Sigma_t)}$ and $d_t(y,x_t^0)<\left(\frac{\|\mathcal H-\hbar\|_{L^{\infty(\Sigma_t)}}}{2C\|\nabla \mathcal H\|_{L^4(\Sigma_t)}}\right)^2=:\epsilon_t$, then 
\begin{equation*}
|\mathcal H(y)-\hbar|\geq \frac{\|\mathcal H-\hbar\|_{L^\infty(\Sigma_t)}}{2}.
\end{equation*}
Notice that, by Sobolev-Poincare's inequality, $\frac{\|\mathcal H-\hbar\|_{L^{\infty(\Sigma_t)}}}{2C\|\nabla \mathcal H\|_{L^4(\Sigma_t)}}\leq C_{\textup{Poin-Sob}}\sigma^{\frac12}$. Since the diameter of $\Sigma_t$ is close to $\pi\sinh\hat r_t\geq C^{-1}\sigma$, we can choose a universal constant $\hat C>0$, only depending on $C^{-1}$ and $C_{\textup{Poin-Sob}}$ such that $\frac{\epsilon_t}{\hat C}<\textup{diam}(\Sigma_t)$. Thus, we replace in the following $\epsilon_t$ with this rescaling, continuing to indicate the radius with $\epsilon_t$. By hypothesis \ref{hypexpdecay} we have 
\begin{equation}\label{lastlemma1}
\begin{aligned}
\cin^2\sigma^{-6}e^{-\frac{10m}{9\sigma^2}t}\geq\|\mathcal H-\hbar\|_{L^2(\Sigma_t)}^2&=\int_{\Sigma_t}(\mathcal H-\hbar)^2 \ d\mu\\
&\geq \int_{B_{d_t}(x_t^0,\epsilon_t)} (\mathcal H-\hbar)^2 \ d\mu\\
&\geq \frac{\|\mathcal H-\hbar\|_{L^\infty(\Sigma_t)}^2}{4}\mu_t\left(B_{d_t}(x_t^0,\epsilon_t)\right).
\end{aligned}
\end{equation}
By Lemma \ref{lemmagraphsob}, we find that 
\begin{equation}\label{lastlemma2}
\mu_t\left(B_{d_t}(x_t^0,\epsilon_t)\right)\geq \vartheta \epsilon_t^2=\vartheta \left(\frac{\|\mathcal H-\hbar\|_{L^{\infty(\Sigma_t)}}}{2C\|\nabla \mathcal H\|_{L^4(\Sigma_t)}}\right)^4.
\end{equation}
Thus, combining \eqref{lastlemma1} and \eqref{lastlemma2},
\begin{equation*}
\begin{aligned}
\vartheta \|\mathcal H-\hbar\|_{L^\infty(\Sigma_t)}^6&\leq 64C^4\cin^2\|\nabla\mathcal H\|_{L^4(\Sigma_t)}^4\sigma^{-6}e^{-\frac{10m}{9\sigma^2}t}\\
&\leq 64C^4\cin^2B_2^4\sigma^{-20}e^{-\frac{10m}{9\sigma^2}t},
\end{aligned}
\end{equation*}
and thus 
\begin{equation}\label{Linftycontrol454}
\|\mathcal H-\hbar\|_{L^\infty(\Sigma_t)}\leq \frac{C\cin^\frac12}{\vartheta^\frac16}B_2^\frac23\sigma^{-\frac{10}{3}}e^{-\frac{5m}{27\sigma^2}t}.
\end{equation}
Combining \eqref{Linftycontrol454} with \eqref{ftc448}, we get 
\begin{equation*}
\begin{aligned}
|r_t-r_0|\leq C(\cin,\vartheta,B_2)\sigma^{-\frac{10}{3}}\int_0^te^{-\frac{5m}{27\sigma^2}\tau}= & \ C(\cin,\vartheta,B_2)\sigma^{-\frac{10}3}\frac{27\sigma^2}{5m}\left(1-e^{-\frac{5mt}{27\sigma^2}}\right)\\
\leq & \ C(\cin,\vartheta,B_2,m)\sigma^{-\frac43}.
\end{aligned}
\end{equation*}
\end{proof}
Before proving the main theorem, we show that the compatibility condition  for the area in \eqref{hypRaggirR} is preserved along the flow. We will show that the other condition in \eqref{hypRaggirR} is preserved along the flow in the proof of Theorem \ref{proofmainthm}.
\begin{lem}\label{arearadiuslem}
Let $\Sigma_t$ be a solution to \eqref{generalflow513} satisfying \eqref{hypRaggirR}, \eqref{hyp45}, \eqref{hyp46}, \eqref{hypP1219}, \eqref{hyp48}, \eqref{lemma22eq} and \eqref{hypexpdecay}. There exist a universal constant $C>0$ and a radius $\sigma_0=\sigma_0(C,\cin,m)>1$ such that 
\begin{equation*}
|\sigma_{\Sigma_t}-\sigma|\leq C\sigma^{-1},
\end{equation*}
for every $t\in[0,T]$.
\end{lem}
\begin{proof}
\begin{equation*}
\begin{aligned}
\left|\frac{d}{dt}\sigma_{\Sigma_t}\right|=\left|\frac{d}{dt}\frac{\sqrt{|\Sigma_t|}}{\sqrt{4\pi}}\right|=\frac1{4\sqrt \pi}\left|\frac{\frac{d}{dt}|\Sigma_t|}{\sqrt{|\Sigma_t|}}\right|&=\frac1{4\sqrt{\pi}}\left|\frac{\int_{\Sigma_t} H(\hbar-\mathcal H) \ d\mu}{\sqrt{|\Sigma_t|}}\right|\\
&\leq C\|\mathcal H-\hbar\|_{L^2(\Sigma_t)}\\
&\leq C\cin \sigma^{-3}e^{-\frac{5m}{9\sigma^2}t}
\end{aligned}
\end{equation*}
and thus 
\begin{equation*}
\begin{aligned}
\left|\sigma_{\Sigma_t}-\sigma\right|=\left|\sigma_{\Sigma_t}-\sigma_{\Sigma_0}\right|\leq & \ C\cin\sigma^{-3}\int_0^te^{-\frac{5m}{9\sigma^2}\tau} \ d\tau\\
= & \ \frac{9C\cin}{5m} \sigma^{-1}\left(1-e^{-\frac{5m}{9\sigma^2}t}\right)\\
\leq & \ \frac{9C\cin}{5m} \sigma^{-1}.
\end{aligned}
\end{equation*}
\end{proof}
\begin{thm}\label{proofmainthm}
Let $\left(M,\overline{\textup g},\overline{\textup{K}}\right)$ be a an asymptotically hyperboloidal initial data set, with $m>0$. Let $\Sigma$ be a \textup{CMC} surface in the sense of \textup{Section \ref{thmnevestian}}, i.e. in particular there exists $\hat r_0$, $C>0$ and $\delta>0$ such that  
\begin{equation}\label{equationCMC}
H^\Sigma=2\frac{\cosh \hat r_0}{\sinh \hat r_0}-\frac{m}{\sinh^3 \hat r_0}+O(e^{-(3+\delta)\hat r_0}),
\end{equation}
and, setting $\sigma:=\sqrt{|\Sigma|/4\pi}$, $|\sinh \hat r_0-\sigma|\leq C$ and $|r-\hat r_0|\leq C\sigma^{-1}$. Suppose moreover that there exists $\cin>0$ such that 
\begin{equation*}
\|\mathcal H-\hbar\|_{L^2(\Sigma)}\leq \cin \sigma^{-3}, \qquad \|\nabla\mathcal H\|_{L^2(\Sigma)}\leq \cin \sigma^{-3}.
\end{equation*}
Then the volume preserving spacetime mean curvature flow \eqref{generalflow513} $\Sigma_t$, with $\Sigma_0=\Sigma$, exists for every positive time. Moreover, there exists $B_1$, $B_2$, $B_3>0$ suitably large such that, for every $t>0$, $\Sigma_t$ satisfies the following properties:
\begin{enumerate}[label=\textup{(\roman*)}]
\item $\Sigma_t\in \mathcal B_{\sigma,\hat r_0}(B_1,B_2,B_3)$;
\item There exists $\cin'$, only depending on $\cin$, such that 
\begin{equation*}
\|\mathcal H-\hbar\|_{L^2(\Sigma_t)}\leq \cin'\sigma^{-3}, \qquad \|\nabla\mathcal H\|_{L^2(\Sigma_t)}\leq \cin'\sigma^{-3},
\end{equation*}
for every $t>0$;
\item Finally, $\Sigma_t$ converges to a \textup{STCMC} surface as $t\to\infty$.
\end{enumerate}
\end{thm}
\begin{proof}
By the properties of the CMC surfaces constructed by Neves-Tian and reviewed in Section \ref{thmnevestian}, we have that there exists a universal constant $C>0$ such that 
\begin{equation*}
\|w_{\hat r_0}^\Sigma\|_{C^{2,\frac12}(\Sigma)}\leq Ce^{-\hat r_0}\leq C\sigma^{-1}, \qquad \left|\sinh \hat r_0-\sigma\right|\leq C.
\end{equation*}
In particular we have that $|r-\hat r_0|\leq C\sigma^{-1}$. Moreover $h\equiv H$ satisfies \eqref{equationCMC} and also $\|\overset{\circ}{A}\|_{L^\infty(\Sigma)}\leq C\sigma^{-2}$, by the construction in \cite{nevestian}. Moreover, by Lemma \ref{lemma213}, see also Remark \ref{usefulremark}, the spacetime mean curvature of $\Sigma$ satisfies
\begin{equation*}
\|\mathcal H-\hbar\|_{L^2(\Sigma)}<\cin\sigma^{-3}, \qquad \|\nabla \mathcal H\|_{L^2(\Sigma)}< \cin\sigma^{-3},
\end{equation*}
\begin{equation*}
\|\mathcal H-\hbar\|_{W^{1,4}(\Sigma)}=\|\mathcal H-\hbar\|_{L^4(\Sigma)}+\sigma\|\nabla\mathcal H\|_{L^4(\Sigma)}\leq C\sigma^{-\frac52}.
\end{equation*}
From now on, we focus on the short-time solution $\Sigma_t=F_t(\Sigma)$ of \eqref{generalflow513} with $\Sigma_0=\Sigma$. Notice that, at $t=0$, equation \eqref{equationCMC}, for $\sigma$ suitably large, implies that 
\begin{equation}
\left|h-2\frac{\cosh \hat r_0}{\sinh \hat r_0}\right|<C\sigma^{-3}
\end{equation}
for some universal $C>0$. Moreover $|\partial_r^\top|< C\sigma^{-2}$ by the estimate on $\nabla w_{\hat r_0}$, for $C>0$ suitably large. By continuity, we find that there exist a small time $\overline t>0$ and $C>0$ such that the hypothesis of Proposition \ref{prp311partI} and $|r_t-\hat r_0|< C\sigma^{-1}$ are satisfied for every $t\in [0,\overline t)$. Thus, for every $t\in [0,\overline t)$ there exists $\hat r_t$ such that 
\begin{equation}
H(t)=2\sqrt{1+\frac1{\sinh^2 \hat r_t}}-\frac{m}{\sinh^3 \hat r_t}+O(\sigma^{-3-\delta}).
\end{equation}
The fact that $|r_t-\hat r_0|\leq C\sigma^{-1}$ for every $t\in [0,\overline t)$ and $|\sinh \hat r_0-\sigma|\leq C$ imply by definition of $\hat r_t$ that $\hat r_0-C\sigma^{-1}\leq r_{\Sigma_t}\leq \hat r_t\leq R_{\Sigma_t}\leq \hat r_0+C\sigma^{-1}$. Thus $|r_t-\hat r_t|\leq 2C\sigma^{-1}$.\\
\indent Since at $t=0$ the following conditions hold strictly, define $T_{\textup{max}}>0$ to be the maximal time $T_*>0$ such that, choosing $B_1$, $B_2$, $B_3$ suitably large, and $c_0>0$ small,
\begin{enumerate}
\item $F_t$ exists for every $t\in [0,T_*]$;
\item $|\partial_r^\top|\leq\frac{c_0}2\sigma^{-1}$ on $\Sigma_t$, $|r_t-\hat r_t|\leq 10C\sigma^{-1}$ and $|\sinh \hat r_t-\sigma|\leq 10C$ for every $t\in [0,T_*]$;
\item $\Sigma_t\in \overline{\mathcal B}_{\sigma,\hat r_t}(B_1,B_2,B_3)$  for every $t\in [0,T_*]$;
\item $\|\mathcal H-\hbar\|_{L^2(\Sigma_t)}\leq(\cin+1)\sigma^{-3}$ for every $t\in [0,T_*]$. More precisely,
\begin{equation*}
\|\mathcal H-\hbar\|_{L^2(\Sigma_t)}\leq (\cin+1)\sigma^{-3} e^{-\frac{5m}{9\sigma^2}t};
\end{equation*}
\item For every $t\in [0,T_*]$,
\begin{equation}\label{ineqcondmaxh}
\left|h(t)-2\sqrt{1+\frac1{\sinh^2 \hat r_t}}+\frac{m}{\sinh^3 \hat r_t}\right|\leq 2C \sigma^{-3-\delta}.
\end{equation}
\end{enumerate}
Notice that the collection of the interval of time we are maximizing is non-empty since $T_\textup{max}\geq \overline t>0$. We want to show that $T_\textup{max}=\infty$. Suppose, by contrary, that $0<T_\textup{max}<\infty$. We show that conditions (1)-(5) holds with $\leq$ replaced by $<$, and thus the flow $F_t$ can by continued past $T_\textup{max}$, contradicting the maximality of $T_{\textup{max}}$.\\
\indent Set $w(t):=w_{\hat r_t}=r_t-\hat r_t$. Notice that, by construction, for $\sigma$ suitably large with respect to $B$, $\|w(t)\|_{L^\infty(\Sigma_t)}=o(1)$ uniformly for every $t\in [0,T_\textup{max})$ and also $\sigma|\nabla w(t)|<\frac{c_0}2$ on $\Sigma_t$ uniformly for every $t\in [0,T_\textup{max})$. Thus $\|w(t)\|_{W^{1,\infty}(\mathbb S^2)}<c_0$ for $\sigma$ suitably large, uniformly in $[0,T_\textup{max})$ and thus Proposition \ref{prp311} implies that there exists $c=c(B)>0$ such that 
\begin{equation}\label{graphmainproof}
\|w(t)\|_{C^{2,\frac12}(\Sigma_t)}\leq c\sigma^{-1},
\end{equation}
for every $t\in [0,T_\textup{max})$. We furthermore remark that, by construction, the following properties hold.
\begin{enumerate}[label=(\roman*)]
\item Assumption \eqref{hypRaggirR} continues to hold, for $\sigma$ suitably large, thanks to Lemma \ref{arearadiuslem} and Definition \ref{dfnroundenss36}, together with the fact that $\Sigma_t$ is an approximating sphere.
\item We will show that \eqref{hyp45} holds along the flow at the end of the proof. We assume for the moment that this is true, and in the end we will obtain that \eqref{hyp45} holds with a constant only depending on the initial setting. Notice that at the initial time $t=0$, condition (2) above is satisfied with  strictly smaller constants.
\item Assumption \eqref{hyp46} and \eqref{hypP1219} are true because of Definition \ref{dfnroundenss36} and the estimate for $h(t)$. See also Lemma \ref{lemma213} and Remark \ref{usefulremark};
\item Also Assumption \eqref{hyp48} holds, combining Definition \ref{dfnroundenss36} with the Sobolev inequality. Moreover, at the initial time, 
\begin{equation*}
|\partial_r^\top\bigg|_{t=0}|=\left|\nabla w(0)\right|\leq C\sigma^{-1},
\end{equation*}
and thus \eqref{stimapartialrnu} implies that $|\nu-\partial_r|\leq C\sigma^{-1}$, and so $|\nu^\top|\leq C\sigma^{-1}<\min\{2,\frac{10^{-3}}{\sqrt m}\}$, for $\sigma$ large. Thus, by Proposition \ref{prp311}, $|\nu_t^\top|\leq C\sigma^{-1}<\min\{2,\frac{10^{-3}}{\sqrt m}\}$, for $\sigma$ possibly larger. We furthermore note that, by Lemma \ref{lemma418}, $\|\mathcal H-\hbar\|_{L^\infty(\Sigma_t)}\leq C\sigma^{-\frac{10}3}$.
\item As a consequence, as explained in Section \ref{section41}, also \eqref{lemma22eq} holds.
\item Moreover, the exact computation carried out in Corollary \ref{cor417} implies that the inequality 
\begin{equation*}
\left|\frac{H}{\mathcal H}-\sinh \hat r_t-\frac{m}2\right|\leq C_m'\sigma^{-\delta}
\end{equation*}
is true, because by Lemma \ref{improvementH220} we improve the result of (iv) obtaining $\|H-h\|_{L^\infty(\Sigma_t)}\leq C\sigma^{-\frac{13}3}$, and we combing it with the estimate for $h(t)$.
\item Note that $\textup{Ric}^{\Sigma_t}\geq 0$ because, in 2 dimension, $\textup{Ric}^{\Sigma_t}=\frac12R^{\Sigma_t}g_t$ and by Gauss equation
\begin{equation*}
\begin{aligned}
R^{\Sigma_t}&=\overline R-|A|^2+H^2-2\overline{\textup{Ric}}(\nu_t,\nu_t)\\
&=\overline R-|\overset{\circ}{A}|^2+\frac{H^2}2-2\overline{\textup{Ric}}(\nu_t,\nu_t)\\
&=\overline R-2\overline{\textup{Ric}}(\nu_t,\nu_t)+\frac{h^2(t)}2+O(\sigma^{-3})\\
&=\frac2{\sinh^2 \hat r_t}+O(\sigma^{-3}),
\end{aligned}
\end{equation*}
using Lemma \ref{lemma14ii} and the estimate for $h(t)$.
\end{enumerate}
Thus, by the definition of $T_\textup{max}$ and (i)-(vii), Corollary \ref{corH1ev} implies that 
\begin{equation}\label{L2nablacontrol466}
\|\nabla \mathcal H\|_{L^2(\Sigma_t)}<\cin'\sigma^{-3}
\end{equation}
for every $t\in [0,T_\textup{max}]$. Note that, using (iv), and choosing $B_\infty>C$, Lemma \ref{evW14} implies that 
\begin{equation}\label{4control}
\|\mathcal H-\hbar\|_{W^{1,4}(\Sigma_t)}<B_2\sigma^{-\frac52},
\end{equation}
for every $t\in [0,T_\textup{max}]$, if $B_2$ is chosen suitably large in Lemma \ref{evW14}. Combining \eqref{L2nablacontrol466} with the definition of $T_\textup{max}$, this implies that \eqref{addassumption419} holds, and thus, combining (vi) and (vii) with \eqref{Linfty439}, which holds for every $t\in [0,T_\textup{max}]$ choosing $B_\infty>c(B_2)$ in \eqref{Linfty439}, Lemma \ref{expdecay413} implies that \eqref{hypexpdecay} holds for every $t\in[0,T_\textup{max}]$. Thus in particular
\begin{equation*}
\|\mathcal H-\hbar\|_{L^2(\Sigma_t)}<\cin \sigma^{-3}
\end{equation*}
holds for every $t\in [0,T_\textup{max}]$. By \eqref{4control}, Proposition \ref{prp310} implies that, provided $B_3$ is suitably large with respect to $B_2$, $\|\overset{\circ}{A}\|_{L^\infty(\Sigma_t)}<B_3\sigma^{-2}$ for every $t\in [0,T_\textup{max}]$. We remark moreover that the constant $2C$ in \eqref{ineqcondmaxh} is chosen to be twice the one obtained in Proposition \ref{prp311} with the uniform bounds in (i)-(vii), and thus condition \eqref{ineqcondmaxh} holds strictly for $t\in [0,T_\textup{max})$. Thus, conditions (1), (3), (4) and (5) in the definition of $T_\textup{max}$ hold strictly for every $[0,T_\textup{max}]$.
\\
\indent It remains to show that condition (2) in the definition of $T_\textup{max}$ holds with strict inequalities. Note that, by \eqref{graphmainproof}, for $\sigma$ suitably large we have $|\partial_r^\top|<\frac{c_0}4\sigma^{-1}$ uniformly in time. Moreover, since \eqref{hypexpdecay} holds, by Lemma \ref{lemma418}, we have that, provided $\sigma$ is suitably large,
\begin{equation*}
r_t-4C\sigma^{-1}\leq r_{\Sigma_t}\leq \hat r_t\leq R_{\Sigma_t}\leq r_t+4C\sigma^{-1},
\end{equation*}
since, again by Lemma \ref{lemma418}, $R_{\Sigma_t}\leq r_0(\overline x_t)+C\sigma^{-1}\leq \hat r_0+2C\sigma^{-1}\leq r_t(x)+4C^{-1}$ for every $x\in \Sigma_t$, where $r_t(\overline x_t)=R_{\Sigma_t}$. We used that $|\hat r_0-r_0(\cdot)|\leq C\sigma^{-1}$. The same holds for $r_{\Sigma_t}$. Moreover
\begin{equation*}
\sinh \hat r_t=\sinh\left(\hat r_0+O(\sigma^{-1})\right)=\sinh \hat r_0+O(1)=\sigma+O(1).
\end{equation*}
Thus condition (2) holds strictly provided the universal constant $C>0$ was chosen suitably large.\\
\indent In conclusion, we proved that $\Sigma_t\in \mathcal B_{\hat r_t,\sigma}(B_1,B_2,B_3)$ for every $t\in [0,T_\textup{max}]$, and, as long as we are in the roundness class, the second fundamental form of $\Sigma_t$ is uniformly bounded, see (iii) above. Thus, by classical arguments, see \cite{vpstmcf} for the case of the volume preserving spacetime mean curvature flow, the interval of existence of $F_t$ is strictly bigger than $[0,T_\textup{max}]$, contradicting its maximality and showing that $T_\textup{max}=\infty$.
\end{proof}
\begin{rem} Exponential convergence is an immediate consequence of Lemma \ref{lemma418}. In order to get the smooth convergence of $\Sigma_t$ to a limit (STCMC) surface $\Sigma_\infty$, we invite the reader to follow the standard arguments in \cite{huiskenyau} and \cite{vpstmcf}. 
\end{rem}
\subsection{Unique stable leaves of the foliation}\label{thefoliationproof} Once the family of STCMC surfaces above is constructed, showing that it forms a \textit{foliation}, i.e. it is a locally unique family of stable leaves which does not intersect each other, is a routine verification. However, we sketch in this Subsection the verification for our hyperbolic case.\\
\indent In the following, we will indicate the STCMC sufaces $\Sigma_\infty^\sigma=\lim_{t\to\infty}F_t(\Sigma^\sigma)$ simply as $\{\Sigma_{\textup{st}}^\sigma\}_{\sigma\geq \sigma_0}$.
\\
\\
\noindent \textbf{Stability and invertibility of the spacetime stability operator.} Notice that the computations in Lemma \ref{expdecay413}, combined with the fact that $m>0$, imply that the surfaces $\{\Sigma_{\textup{st}}^\sigma\}_{\sigma\geq \sigma_0}$ are \textit{stable}. We first recall the definition of spacetime stability operator in the case under consideration, then prove that stability implies invertibility of the operator, following \cite{huiskenyau}. Finally we show that the operator is stable.\\
\indent Let $\Sigma\hookrightarrow M$ be a closed 2 faced surface. Let $\mathcal V:\Sigma\times (-\xi,\xi)\to M$ be a normal variation, i.e.
\begin{equation*}
\mathcal V(\cdot,0)=\textup{Id}_\Sigma, \qquad \frac{\partial}{\partial t}\bigg|_{t=0}\mathcal V=f\nu,
\end{equation*}
for some $f\in C^\infty(\Sigma)$. We then set 
\begin{equation*}
\mathcal Lf:=\frac{\partial}{\partial t}\bigg|_{t=0}\mathcal H\left(\mathcal V(\cdot,t)\right)=-\frac{H}{\mathcal H}\left(\Delta f+\left(|A|^2+\overline{\textup{Ric}}(\nu,\nu)\right) f\right)-\frac{P}{\mathcal H}\frac{\partial P}{\partial t}\bigg|_{t=0},
\end{equation*}
where $\mathcal H(\mathcal V(\cdot,t))$ is the spacetime mean curvature of the surface $\Sigma_t:=\mathcal V(\Sigma,t)$ while $\mathcal H=\mathcal H(\mathcal V(\cdot,0))$. Note that, in our case, $\mathcal L$ is not self-adjoint, due to the presence of $\frac{H}{\mathcal H}$. Proceeding similar to \cite{huiskenyau}, one can prove invertibility once having shown that $\mathcal L$ is stable on functions of zero mean value. We say that $\Sigma$ is \textit{stable} if for every $\int_\Sigma f\ d\mu=0$, $f\not\equiv 0$, 
\begin{equation}\label{spacetimestability}
\int_\Sigma \left(\mathcal Lf\right) f \ d\mu>0.
\end{equation}
In fact, we prove more. Remember that 
\begin{equation*}
\frac{H}{\mathcal H}\geq \frac{4\sigma}{5}, \qquad \frac{P}{\mathcal H}\sim O(\sigma), \qquad 
\left|\frac{\partial P}{\partial t}\right|\leq C\sigma^{-5}|f|+C\sigma^{-5}|\nabla f|.
\end{equation*}
By direct computation
\begin{equation*}
\begin{aligned}
\int_\Sigma \left(\mathcal Lf\right) f \ d\mu &\geq \int_\Sigma \sigma\left(|\nabla f|^2-(|A|^2+\textup{Ric}(\nu,\nu))f^2\right) \ d\mu\\
&\quad +\int_\Sigma \left(\frac{H}{\mathcal H}-\sigma\right)\left(|\nabla f|^2-(|A|^2+\textup{Ric}(\nu,\nu))f^2\right) \ d\mu\\
&\quad -C\sigma^{-4}\int_\Sigma f^2 \ d\mu-C\sigma^{-4}\int_\Sigma |f||\nabla f| \ d\mu.
\end{aligned}
\end{equation*}
Using Proposition \ref{spectralprop} for the operator $L=-\Delta-\left(|A|^2+\textup{Ric}(\nu,\nu)\right)$, and the fact that $\nabla\left(\frac{H}{\mathcal H}\right)=\frac{P\nabla P}{H\mathcal H}\sim\sigma^{-4}$ implies $\left|\frac{H}{\mathcal H}-\sigma\right|\leq C\sigma^{-3}$, we have
\begin{equation*}\label{eq443}
\int_\Sigma \left(\mathcal Lf\right)f \ d\mu\geq C(m)\sigma^{-2}\int_\Sigma f^2 \ d\mu-C\sigma^{-4}\int_\Sigma |\nabla f|^2 \ d\mu.
\end{equation*}
On the other hand 
\begin{equation}\label{eq444}
\int_\Sigma |\nabla f|^2 \ d\mu\leq \int_\Sigma (Lf)f \ d\mu+\sigma^{-2}\|f\|_2^2.
\end{equation}
Combining these inequalities we have 
\begin{equation*}
\begin{aligned}
\frac{4\sigma}5\int_\Sigma (Lf)f \ d\mu&\leq \int_\Sigma (\mathcal Lf)f \ d\mu+C\sigma^{-4}\int_\Sigma f^2 \ d\mu+C\sigma^{-4}\int_\Sigma |\nabla f|^2 \ d\mu\\
&\leq  \int_\Sigma (\mathcal Lf)f \ d\mu+C\sigma^{-4}\int_\Sigma f^2 \ d\mu+C\sigma^{-4}\int_\Sigma (Lf)f \ d\mu.
\end{aligned}
\end{equation*}
Thus
\begin{equation*}
\int_\Sigma (\mathcal Lf)f \ d\mu\geq \frac{\sigma}{10}\int_\Sigma (Lf)f \ d\mu-C\sigma^{-4}\int_\Sigma f^2 \ d\mu.
\end{equation*}
Combining this with \eqref{eq444} we end with 
\begin{equation}\label{conclusionstability}
\int_\Sigma (\mathcal L f)f \ d\mu\geq C(m)\sigma^{-2}\int_\Sigma f^2 \ d\mu\geq 0, \qquad  \int_\Sigma |\nabla f|^2 \ d\mu\leq C(m)\int_\Sigma (\mathcal Lf)f \ d\mu.
\end{equation}
To prove stability, we set $\overline f:=\fint_\Sigma f\ d\mu$ and $f^t:=f-\overline f$. Then
\begin{equation}\label{decomp}
\begin{aligned}
\mathcal Lf^t=\mathcal L(f-\overline f)=:\mathcal L f+\theta\overline f,
\end{aligned}
\end{equation}
where, by Lemma \ref{lemma33},
\begin{equation*}
\begin{aligned}
\theta=\frac{H}{\mathcal H}\left(|A|^2+\overline{\textup{Ric}}(\nu,\nu)\right)+\frac{P}{\mathcal H}\left(\frac{\partial P}{\partial t}\bigg|_{t=0}\right)&=\frac{H}{\mathcal H}\left(\frac{\mathcal H^2}{2}+\frac{P^2}2+\overline{\textup{Ric}}(\nu,\nu)+|\overset{\circ}{A}|^2\right)+O(\sigma^{-3})\\
&=\frac{H}{\mathcal H}\left(\frac{\mathcal{H}^2}{2}+\frac{m}{\sinh^3 \hat r}\right)+O(\sigma^{-3}).
\end{aligned}
\end{equation*}
In particular, $\theta$ is close to a constant $\overline\theta$ which scales as $\sim \sigma^{-1}$, up to an error of order $O(\sigma^{-3})$. Integrating \eqref{decomp} against $f^t$, we find
\begin{equation*}
\begin{aligned}
\int_\Sigma \left(\mathcal Lf^t\right)f^t&=\int_\Sigma \mathcal L f(f-\overline f) \ d\mu+\overline f\int_\Sigma (\theta-\overline\theta) f^t \ d\mu\\
&\leq \|\mathcal Lf\|_2\|f^t\|_2+C\sigma^{-2}|\overline f|\|f^t\|_2.
\end{aligned}
\end{equation*}
Since $\mathcal L$ is stable, and $f^t$ has zero mean, 
\begin{equation}\label{eq456}
\|f^t\|_2^2\leq C\sigma^2\int_\Sigma (\mathcal Lf^t) f^t \ d\mu\leq C\sigma^2\|\mathcal Lf\|_2\|f^t\|_2+C|\overline f|\|f^t\|_2.
\end{equation}
On the other hand, applying \eqref{conclusionstability} to $f^t$, we also have $\|\nabla f^t\|_2\leq C\sigma\|\mathcal Lf\|_2+C\sigma^{-1}|\overline f|$. Again by definition of stability operator, since $f^t$ has zero mean, we have
\begin{equation*}
\int_\Sigma \mathcal L f^t \ d\mu=\int_\Sigma \nabla\left(\frac{H}{\mathcal H}\right)\cdot\nabla f^t-\int_\Sigma (\theta-\overline\theta) f^t \ d\mu-2\int_\Sigma \frac{P}{\mathcal H}\overline P(\nu,\nabla f^t) \ d\mu,
\end{equation*}
and since $\left\|\nabla\left(\frac{H}{\mathcal H}\right)\right\|_2\leq C\sigma^{-3}$, 
\begin{equation*}
\left|\int_\Sigma \mathcal Lf^t \ d\mu\right|\leq C\sigma^{-3}\|\nabla f^t\|_2+C\sigma^{-2}\|f^t\|_2.
\end{equation*}
Thus, by definition of $\overline\theta$, 
\begin{equation*}
\begin{aligned}
\frac{\sigma}{10}|\overline f|\leq C|\overline f||\Sigma||\overline\theta|=C\left|\overline f\int_\Sigma \theta \ d\mu\right|&\leq C\left|\int_\Sigma \mathcal Lf^t-\int_\Sigma \mathcal Lf\right|\\
&\leq C\sigma^{-3}\|\nabla f^t\|_2+C\sigma^{-2}\|f^t\|_2+C\sigma\|\mathcal Lf\|_2.
\end{aligned}
\end{equation*}
Combining this with \eqref{eq456} we find that $|\overline f|\leq C\|\mathcal Lf\|_2$. We conclude the proof as follows. On one hand, the inequality 
\begin{equation*}
\|\mathcal Lf^t\|_2\geq C\sigma^{-2}\|f^t\|_2-C\sigma^{-4}\left(\|f^t\|_2+\|\nabla f^t\|_2\right),
\end{equation*}
obtained computing the difference between $\mathcal L$ and $L$, combined with \eqref{conclusionstability}, gives $\|\mathcal Lf^t\|_2\geq C\sigma^{-2}\|f^t\|_2$. On the other, 
\begin{equation*}
\left|\mathcal L\overline f\right|=\left|-\frac{H}{\mathcal H}\left(|A|^2+\overline{\textup{Ric}}(\nu,\nu)\right)\overline f-\frac{P}{\mathcal H}\overline f\textup{tr}_\Sigma(\overline\nabla_\nu\overline P)\right|\leq C\sigma^{-1}|\overline f|,
\end{equation*}
and thus 
\begin{equation*}
\|f\|_2\leq \|f^t\|_2+\|\overline f\|_2\leq C\sigma^2\|\mathcal Lf^t\|+C\sigma\|\mathcal Lf\|_2\leq \sigma^2\|\mathcal Lf\|_2,
\end{equation*}
since $\mathcal Lf^t=\mathcal Lf-\mathcal L\overline f$.
\\
\\
\noindent \textbf{Local uniqueness.} Once we have the invertibility of $\mathcal L$, we follow \cite{huiskenyau} in order to prove that in each class of roundness, say $\mathcal B_{\sigma}(B_1,B_2,B_3)$, there is a unique STCMC surface of spacetime mean curvature $\mathcal H(\Sigma_{\textup{st}}^\sigma)$. Suppose, by contrary, that we have $\Sigma_1^\sigma$ and $\Sigma_2^\sigma$ with $\mathcal H(\Sigma_1^\sigma)=\mathcal H(\Sigma_2^\sigma)\equiv \textup{const}$. We show that necessarily $\Sigma_1^\sigma=\Sigma_2^\sigma$.\\
\indent We remark that on a STCMC surface the estimate on the graph function can be improved. In fact, we proved that for $\Sigma^\sigma$, $\|w_{\hat r}\|_{C^{2,\frac12}(\Sigma)}\leq C\sigma^{-1}$. Using \eqref{furtherineq322} on $v$, with $w=u+v$ as in Section \ref{subsectgraphs}, we have that the estimate on $\|v\|_{C^{2,\frac12}(\Sigma)}\leq C\sigma^{-1-\min\{\delta,1\}}$. Moreover, by \eqref{hatkappaidentity}, 
\begin{equation*}
\hat \kappa=\sigma^2\left(-1+\frac{h^2}{4}+\frac{m}{\sinh^3 \hat r}\right)+O(\sigma^{-3})\equiv \hat \kappa_0+O(\sigma^{-3}).
\end{equation*}
In particular, by Lemma \ref{lemma14ii} and the mean value theorem we have 
\begin{equation*}
\hat \kappa-\hat \kappa_0=-3m\sigma^{-1}w+O(\sigma^{-1}|w|)+O(\sigma^{-3}).
\end{equation*}
Combining this with \eqref{KWineq} we have that also $\|u\|_{C^{0,\frac12}}\leq C\Sigma_i\left|\langle w,x_i\rangle_{L^2}\right|\leq C\sigma^{-2}$. By elliptic regularity we have $\|w\|_{C^{2,\frac12}(\Sigma)}\leq C\sigma^{-1-\min\{\delta,1\}}$.\\
\indent Consider now again $\Sigma_1^\sigma$ and $\Sigma_2^\sigma$. We write $\Sigma_2^\sigma$ as a graph on $\Sigma_1^\sigma$, that is $\Sigma_s:=\Sigma_1^\sigma+s(\rho_2-\rho_1)\nu_{\Sigma_1^\sigma}$, where $\rho_i$ are the radial graph functions of $\Sigma_i^\sigma$. With an abuse of notation, the scalar function $\mathcal H(s):=\mathcal H(\Sigma_s)$, i.e. the spacetime mean curvature of $\Sigma_s$, is such that $\mathcal H(0)=\mathcal H(1)$. Moreover, 
\begin{equation*}
\frac{d\mathcal H}{ds}(0)=\mathcal L\left\langle (\rho_2-\rho_1)\partial_r,\nu_{\Sigma_1^\sigma}\right\rangle,
\end{equation*}
where the argument of $\mathcal L$ is the lapse function of the variation $\Sigma_s$ at time $s=0$. The tangential part is zero since the base surface has constant spacetime mean curvature.\\
\indent To estimate $d\mathcal H/ds(0)$, we proceed as follows. Consider the function $\mathcal H^m(s)=\sqrt{\left(H^m(s)\right)^2-4}$, where $H^m(s)$ is the mean curvature of $\Sigma_s$ with respect to the anti-de Sitter-Schwarzschild metric $\overline g^m_{\textup{adS}}$. By the mean value theorem, there exists an intermediate value of $\xi\in [0,1]$ such that 
\begin{equation}\label{inthis1}
\mathcal H'(0)=\left(\mathcal H-\mathcal H^m\right)'(0)+\left(\mathcal H^m(1)-\mathcal H^m(0)\right)-\frac12\left(\mathcal H^m\right)''(\xi).
\end{equation}
We thus have to bound $\left|\left(\mathcal H-\mathcal H^m\right)'\right|$, $\left|\left(\mathcal H^m\right)'\right|$ and $\left|\left(\mathcal H^m\right)''\right|$.\\
\indent In polar coordinates, we can write $\Sigma_s=\{(r,x): \ r=\rho_s(x)\}$, where $x\in\mathbb S^2$. By a direct computation, using the fact that $|\overline g-\overline g^m_\textup{adS}|$ and its derivatives are of order $O(\sigma^{-5})$ and the definition of mean curvature, one finds that 
\begin{equation}\label{usedthis1}
\left|H-H^m\right|\leq C\sigma^{-5}, \qquad \left|\partial_s(H-H^m)\right|\leq C\sigma^{-5}\sum_{k=0}^2\left|D^k(\rho_1-\rho_2)\right|,
\end{equation}
where $D$ is the derivative on $\mathbb S^2$. On the other hand, $H^m(s)$ can be explicitly written as a function $F(\rho_s,D\rho_s,D^2\rho_s)$, where 
\begin{equation*}
F(\rho,D\rho,D^2\rho)=\frac{a'}{a}\frac{1}{\left(1+\frac{|D\rho|^2}{a^2}\right)^\frac32}-\frac1{a^2\sqrt{1+\frac{|D\rho|^2}{a^2}}}\textup{tr}(D^2\rho)+\frac{1}{a^4\left(1+\frac{|D\rho|^2}{a^2}\right)^\frac32}D^2\rho(D\rho,D\rho),
\end{equation*}
where $a^2(r)=\sinh^2 r+\frac{m}{3\sinh r}$. Using the chain rule and the fact that the leading term in the derivatives of $F$ is $\left|F_\rho\right|\leq C\sigma^{-2}$, one ends with 
\begin{equation*}
\left|\partial_sH^m\right|^2\leq C\sigma^{-2}\sum_{k=0}^2\left|D^k(\rho_1-\rho_2)\right|.
\end{equation*}
A similar approach to the second derivatives, but with quadratic terms, gives 
\begin{equation*}
\left|\partial_s^2H^m\right|\leq C\sigma^{-2}\sum_{k=0}^2\left|D^k(\rho_2-\rho_1)\right|^2,
\end{equation*}
since again $|F_{\rho\rho}|\leq C\sigma^{-2}$ is the leading term. Note moreover that $\partial_s(H+H^m)=2\partial_s H^m+\partial_s\left(H-H^m\right)$, and thus 
\begin{equation}\label{usedthis2}
\left|\partial_s\left(H+H^m\right)\right|\leq C\sigma^{-2}\sum_{k=0}^2\left|D^k(\rho_1-\rho_2)\right|.
\end{equation}
Since $\mathcal H^m=\sqrt{\left(H^m\right)^2-4}$, and $\mathcal H^m\sim \sigma^{-1}$, the chain rule gives
\begin{equation}\label{usedthis3}
\left|\partial_s\mathcal H^m\right|\leq C\sigma^{-1}\sum_{k=0}^2|D^k(\rho_1-\rho_2)|, \qquad \left|\partial_s^2\mathcal H^m\right|\leq C\sigma^{-1}\sum_{k=0}^2|D^k(\rho_1-\rho_2)|^2.
\end{equation}
On the other hand, proceeding as above, $|\partial_s H|\leq |\partial_s H^m|+|\partial_s(H-H^m)|\leq C\sigma^{-2}\sum_{k=0}^n|D^k(\rho_1-\rho_2)|$.\\
\indent By the computation in Lemma \ref{lemmaevPt}, $|\partial_sP|\leq C\sigma^{-5}\sum_{k=0}^2\left|D^k(\rho_1-\rho_2)\right|$, and thus
\begin{equation*}
\left|\partial_s\mathcal H\right|=\left|\frac{(\partial_s H)H-P\partial_s P}{\mathcal H}\right|\leq C\sigma^{-1}\sum_{k=0}^2|D^k(\rho_1-\rho_2)|.
\end{equation*}
This is a raw estimate for $\mathcal H'$, but, as we will see below, it needs to be refined to conclude.\\
Finally we estimate 
\begin{equation*}
\begin{aligned}
\left|\partial_s\left(\mathcal H^2-\left(\mathcal H^m\right)^2\right)\right|&=\left|\partial_s\left(H^2-P^2-\left(H^m\right)^2-4\right)\right|\\
&=\left|\partial_s\left(H-H^m\right)(H+H^m)+\partial_s(H+H^m)(H-H^m)-2P\partial_sP\right|\\
&\leq C\sigma^{-5}\sum_{k=0}^2|D^k(\rho_1-\rho_2)|,
\end{aligned}
\end{equation*}
where we used \eqref{usedthis1}, \eqref{usedthis2}, the fact that $H$ and $H^m$ are bounded and again Lemma \ref{evPt}. This allows to conclude that 
\begin{equation*}
\begin{aligned}
\left|\partial_s\left(\mathcal H-\mathcal H^m\right)\right|&\leq \frac{\left|\partial_s\left(\mathcal H^2-(\mathcal H^m)^2\right)\right|}{\mathcal H+\mathcal H^m}+\frac{\left|\mathcal H^2-(\mathcal H^m)^2\right|\left|\partial_s(\mathcal H+\mathcal H^m)\right|}{(\mathcal H+\mathcal H^m)^2}\\
&\leq C\sigma^{-4}\sum_{k=0}^2|D^k(\rho_1-\rho_2)|.
\end{aligned}
\end{equation*}
Employing this and \eqref{usedthis3} in \eqref{inthis1}, we finally find
\begin{equation*}
\begin{aligned}
\left|\mathcal L\left\langle (\rho_2-\rho_1)\partial_r,\nu_{\Sigma_1^\sigma}\right\rangle\right|=\left|\mathcal H'(0)\right|&\leq C\left(\sigma^{-1}\sum_{k=0}^2|D^k(\rho_1-\rho_2)|+\sigma^{-4}\right)\sum_{k=0}^2|D^k(\rho_1-\rho_2)|\\
&\leq C\left(\sigma^{-1}\|\rho_1-\rho_2\|_{C^2}+\sigma^{-4}\right)\sum_{k=0}^2|D^k(\rho_2-\rho_1)|.
\end{aligned}
\end{equation*}
Note that by definition
\begin{equation*}
\rho_2-\rho_1=\hat r_2-\hat r_1+(w_{\hat r_2}-w_{\hat r_1}).
\end{equation*}
We proved above that, since $\Sigma_1^\sigma$ and $\Sigma_2^\sigma$ are STCMC, $\|w_{\hat r_2}-w_{\hat r_1}\|_{C^{2,\frac12}}\leq C\sigma^{-1-\min\{\delta,1\}}$. Moreover, by the fact that 
\begin{equation*}
|h_1-h_2|=\left|\sqrt{\mathcal H_1^2+4}-\sqrt{\mathcal H_2^2+4}+O(\sigma^{-5}\right|\leq O(\sigma^{-5})
\end{equation*}
and Lemma \ref{lemma33}, we find $|\sinh \hat r_1-\sinh \hat r_2|\leq C\sigma^{-\delta}$ and thus $|\hat r_2-\hat r_1|\leq C\sigma^{-1-\delta}$. Thus $\|\rho_2-\rho_1\|_{C^2}\leq C\sigma^{-1-\min\{1,\delta\}}$. Fix now the following notation for the lapse function 
\begin{equation*}
\varphi:=\left\langle (\rho_2-\rho_1)\partial_r,\nu_{\Sigma_1^\sigma}\right\rangle.
\end{equation*}
By the stability above, combined with the elliptic regularity theory, 
\begin{equation*}
\|\mathcal L\varphi\|_2\geq C\sigma^{-2}\|\varphi\|_{W^{2,2}}.
\end{equation*}
Note moreover that $\|\rho_1-\rho_2\|_{W^{2,2}}=\|\langle\partial_r,\nu_1\rangle^{-1}\varphi\|_{W^{2,2}}$ and, by \eqref{stimapartialrnu}, 
\begin{equation*}
\langle\partial_r,\nu_1\rangle=1-\frac12|\nabla w|^2+O(\sigma^{-5}).
\end{equation*}
By explicit computation we find that $\|\rho_1-\rho_2\|_{W^{2,2}}\leq C\|\varphi\|_{W^{2,2}}$. In fact, using the computations in Appendix \ref{directcomputation}, we have 
\begin{equation}
|\nabla\langle \partial_r,\nu\rangle^{-1}|\leq C\sigma^{-2}, \qquad |\nabla^2\langle \partial_r,\nu\rangle^{-1}|\leq C\sigma^{-4}+C\sigma^{-2}|\nabla A|.
\end{equation}
Thus $\|\nabla^2\langle\partial_r,\nu\rangle^{-1}\|_2\leq C\sigma^{-3}+C\sigma^{-2}\|\nabla A\|_2\leq C\sigma^{-3}$, using Lemma \ref{H1boot} in the roundness class. As a consequence, setting $\alpha=\langle\partial_r,\nu\rangle^{-1}$,
\begin{equation*}
\begin{aligned}
\sigma^2\|\nabla^2\left(\langle \partial_r,\nu\rangle^{-1}\varphi\right)\|_2&\leq \sigma^2\|\alpha\|_\infty\|\nabla^2\varphi\|_2+2\sigma^2\|\nabla\alpha\|_\infty\|\nabla\varphi\|_2+\sigma^2\|\nabla^2\alpha\|_2\|\varphi\|_\infty\\
&\leq C\|\varphi\|_{W^{2,2}}+C\sigma\|\nabla \varphi\|_2+C\sigma^{-2}\|\varphi\|_{W^{2,2}},
\end{aligned}
\end{equation*}
also using Sobolev's immersion of $W^{2,2}$ in $L^\infty$. Thus
\begin{equation}
\begin{aligned}
\|\mathcal L\varphi\|_2&\leq C\left(\sigma^{-1}\|\rho_1-\rho_2\|_{C^2}+\sigma^{-4}\right)\|\rho_1-\rho_2\|_{W^{2,2}}\\
&\leq C\left(\sigma^{-1}\|\rho_1-\rho_2\|_{C^2}+\sigma^{-4}\right)\|\varphi\|_{W^{2,2}}\\
&\leq C\left(\sigma\|\rho_1-\rho_2\|_{C^2}+\sigma^{-2}\right)\|\mathcal L\varphi\|_{2},
\end{aligned}
\end{equation}
which, thanks to the fact that $\sigma\|\rho_2-\rho_1\|_{C^2}$ is small for $\sigma$ large, implies that $\mathcal L\varphi=0$. By the invertibility of $\mathcal L$ we thus have $\varphi=0$, and thus $\rho_1=\rho_2$, that is $\Sigma_1^\sigma=\Sigma_2^\sigma$.
\\
\\
\noindent \textbf{$\Sigma_{\textup{st}}^\sigma$ is a foliation.} Finally, the invertibility of the stability operator $\mathcal L$ also shows that $\Sigma_{\textup{st}}^\sigma$ is in fact a foliation. By the usual standard argument (see for example \cite{huang}), using the inverse function theorem we can construct a map $F:(0,\hbar_0)\times\mathbb S^2\to M$ such that $\Sigma_\hbar:=F(\hbar,\mathbb S^2)$ is the unique surface of constant spacetime mean curvature equal to $\hbar$ in a suitably small neighborhood of $\Sigma^\sigma$. The use of the inverse function theorem requires the invertibility of $\mathcal L$ proved above. By continuity and the fact that $\hbar$ goes to zero as $\sigma\to\infty$, we can extend $F_\hbar(\cdot):=F(\hbar,\cdot)$ on the whole interval $(0,\hbar_0)$. We show that the leaves do not intersect. Let $\Sigma^\hbar:=F_\hbar(\mathbb S^2)$. By continuity and uniqueness, each surface of such a kind is a $\Sigma_\textup{st}^\sigma$ for some $\sigma$.\\
\indent We show that the family of surfaces is a foliation, that is the lapse function of the variation with respect to $\hbar$ has a sign. Deriving in $\hbar$ the equation $\mathcal H(\Sigma^\hbar)\equiv \hbar$, we have
\begin{equation*}
\mathcal L\left(\langle\partial_\hbar F,\nu_\hbar\rangle\right)=1,
\end{equation*}
and we set $\varphi:=\langle\partial_\hbar F,\nu_\hbar\rangle$. We claim that there exists $C>0$ such that 
\begin{equation*}
\left|\varphi_\hbar-\frac1{\fint_\Sigma \mathcal L1 \ d\mu}\right|\leq C.
\end{equation*}
Since we also prove that $\left|\fint_\Sigma \mathcal L1 \ d\mu\right|\leq C\sigma^{-1}$ and $\fint_\Sigma \mathcal L1 \ d\mu<0$, we have the conclusion.\\
\indent By the fact that $\mathcal L1=-\frac{H}{\mathcal H}L1-\frac{P}{\mathcal H}\partial_t P$ and 
\begin{equation*}
L1=|A|^2+\textup{Ric}(\nu,\nu)=\frac{\hbar^2}{2}+|\overset{\circ}{A}|^2+O(\sigma^{-3})=\frac{\hbar^2}{2}+O(\sigma^{-3}),
\end{equation*}
we have 
\begin{equation}\label{eq449}
\mathcal L1=-\frac{H}{\hbar}\frac{\hbar^2}{2}+O(\sigma^{-3})=-\frac{h\hbar}{2}+O(\sigma^{-3}),
\end{equation}
where we used that $H=h+O(\sigma^{-5})$ since $H=\sqrt{\hbar^2+P^2}$. This implies that 
\begin{equation*}
\left|\mathcal L1-\fint_\Sigma \mathcal L1 \ d\mu\right|\leq C\sigma^{-3}.
\end{equation*}
Note that \eqref{eq449} also says that $\fint_\Sigma \mathcal L1<0$ and gives an estimate on its smallness. Thus
\begin{equation*}
\left|\mathcal L\left(\varphi_\hbar-\frac1{\fint_\Sigma \mathcal L 1 \ d\mu}\right)\right|=\left|\frac{\mathcal L1-\fint_\Sigma \mathcal L1 \ d\mu}{\fint_\Sigma \mathcal L1 \ d\mu}\right|\leq C\sigma^{-2}.
\end{equation*}
Integrating in $L^2$, using the stability and the Sobolev immersion, we have that $u:=\varphi_\hbar-\frac1{\fint_\Sigma \mathcal L 1 \ d\mu}$ satisfies
\begin{equation*}
\|u\|_{L^\infty}\leq C\sigma^{-1}\|u\|_{W^{2,2}}\leq C\sigma\|\mathcal Lu\|_2\leq C,
\end{equation*}
which is the thesis.
\section{Barycenter and center of mass}\label{conclusion5}
In this concluding Section, we study the evolution of the barycenter of the surface $\Sigma_t$ along the flow. We start defining the notion of barycenter in the hyperbolic setting, aprés the one introduced by Cederbaum-Cortier-Sakovich in \cite{cederbaumcortiersakovich}. With $I$, we will intend the map canonical immersion $I:\mathbb H^3\to \R^{1,3}$ defined in \eqref{isometricimmersion}.
\subsection{Barycenter of a surface} Consider now a surface $\Sigma$ embedded in $\mathbb H^3$ by $F:\Sigma\hookrightarrow \mathbb H^3$. We could define the hyperbolic barycenter of $\Sigma$ as 
\begin{equation*}
C_\Sigma:=\frac1{|\Sigma|}\int_\Sigma I\circ F \ d\mu_\Sigma.
\end{equation*}
Notice that $F(\Sigma)$ is a surface in $\mathbb H^3$, and thus $I(F(\Sigma))$ a surface in the hyperboloid model. Moreover, the integral is in the cone generated by $I(F(\Sigma))$ in $\mathbb R^{3,1}$, since integration is an affine operation. However, in general it does not belong to the hypeboloid. But the normalization 
\begin{equation*}
\hat C_\Sigma:=\frac{C_\Sigma}{\sqrt{-|C_\Sigma|_\eta^2}}
\end{equation*}
does. Note that the square root in the denominator is well defined because $C_\Sigma$ is in the cone. In particular 
\begin{equation*}
C_\Sigma=|\Sigma|^{-1}\left(\int_\Sigma \cosh(|F|) \ d\mu, \int_\Sigma \sinh(|F|)\left(\frac{F}{|F|}\right) \ d\mu\right).
\end{equation*}
\begin{lem}\label{lemma41} If $\Sigma\stackrel{F}{\hookrightarrow}\mathbb H^3$, $\hat r>0$ and $\sigma>0$ are such that $||F|-\hat r|\leq B\sigma^{-1}$, for some $B>0$, and $\frac12\leq \sinh \hat r\leq \frac32\sigma$, there exists a universal constant $C>0$ and $\hat R_0(B)>1$ such that $-|C_\Sigma|_\eta^2\geq C\sigma^2$ if $\hat r>\hat R_0$.
\end{lem}
\begin{proof}
By the hypothesis, there exists 
\begin{equation}\label{shin|F|}
\left|\sinh|F|-\sinh \hat r\right|\leq C\cosh \hat r\left||F|-\hat r\right|\leq CB.
\end{equation}
Thus, decomposing the position vector on $\Sigma$ as 
\begin{equation}\label{decomposingFF}
\frac{F}{|F|}=\frac{F-\hat r\nu+\hat r\nu}{\hat r+O(\sigma^{-1})}=\nu+\frac{F-\hat r\nu}{\hat r+O(\sigma^{-1})}+\frac{O(\sigma^{-1})}{\hat r},
\end{equation}
where $\nu$ is the position on $\mathbb S^2$, combining \eqref{shin|F|} and \eqref{decomposingFF} we have
\begin{equation*}
\left|\int_\Sigma \sinh(|F|)\left(\frac{F}{|F|}\right) \ d\mu\right|=\left|\int_{\Sigma}\left(\sinh \hat r+O(1)\right)\left(\nu+\frac{O(\sigma^{-1})}{\hat r}\right) \ d\mu\right|\leq CB|\Sigma|,
\end{equation*}
since $\int_{\mathbb S^2} \nu \ d\mu=0$. Thus
\begin{equation*}
\begin{aligned}
-|C_\Sigma|_\eta^2&=\frac1{|\Sigma|^2}\left(\left(\int_\Sigma \cosh(|F|) \ d\mu\right)^2 -\left|\int_\Sigma \sinh(|F|)\left(\frac{F}{|F|}\right) \ d\mu\right|^2\right)\\
&\geq \frac1{|\Sigma|^2}\left(C\sigma^2|\Sigma|^2-CB^2|\Sigma|^2\right)\geq C\sigma^2
\end{aligned}
\end{equation*}
for $\sigma$ (and thus $\hat r$) suitably large with respect to $B$, where the estimate $\cosh |F|\geq C\sigma$ is a consequence of \eqref{shin|F|}. Thus the proof is completed.
\end{proof}
Since $\hat C_\Sigma\in H$, we can give the following definition.
\begin{dfn}\label{definitionbarycenter} Let $\Sigma\hookrightarrow \mathbb H^3$ be a surface. We define the hyperbolic barycenter of $\Sigma$ as the point $\boldsymbol z_\Sigma\in\mathbb H^3$ defined by 
\begin{equation*}
\boldsymbol z_\Sigma:=I^{-1}\left(\hat C_\Sigma\right).
\end{equation*}
\end{dfn}
\subsection{Evolution of the barycenter of $\Sigma_t$}
In this final Section, we consider a solution of the VPSTMCF equation \eqref{generalflow513} starting from a CMC-surface $\Sigma^\sigma$ of area radius $\sigma$. For this reason, the considered solution will be indicated as $\Sigma^\sigma_t$. As reviewed in Proposition \ref{prp55}, we also set $\Sigma^\sigma_\textup{st}:=\displaystyle \lim_{t\to\infty} \Sigma_t^\sigma$. In the following, we will moreover indicate with $d_{\mathscr{H}}$ the distance on the Minkowski hyperboloid $\mathscr{H}=\{w\in\R^{1,3}: \ |w|_{\R^{1,3}}=-1\}$.
\begin{lem}\label{distancegraphorev45}
Let $\Sigma_t^\sigma$ be a solution to \eqref{generalflow513} satisfying the properties constructed in Theorem \textup{\ref{proofmainthm}}. Then, there exists a universal constant $C>0$ and a radius $\sigma_0=\sigma_0(B_1,B_2,B_3,\cin)>1$ such that, if $\sigma>\sigma_0$, for every $t>0$, 
\begin{equation}\label{estimate0z}
\left|\frac{d}{dt} d_{\mathscr{H}}\left(\hat C_{\Sigma_t},\hat N\right)\right|\leq C\sigma^{-1}\|\mathcal H-\hbar\|_{L^2(\Sigma_t)},
\end{equation}
where $\hat N=(1,0,0,0)\in\mathbb R^{1,3}$.
\end{lem}
\begin{proof}
For sake of brevity, we set $\hat C(t):=\hat C_{\Sigma_t}$ and, by definition of distance on the Minkowski hyperboloid,
\begin{equation*}
d(t)=d_{\mathscr{H}}(\hat C(t),(1,0,0,0))=\textup{arcosh}\left(Q(t)\right),
\end{equation*}
where $Q(t)$ is the first component of $\hat C(t)$, i.e. 
\begin{equation}\label{defA422}
Q(t):=\frac{1}{\sqrt{-|C(t)|_\eta^2}}\left(\frac1{|\Sigma_t|}\int_{\Sigma_t} \cosh(|F_t|) \ d\mu\right),
\end{equation}
where $|F_t|$ is the Euclidean distance of $F_t$ from the origin of the hyperbolic coordinates. Since $\hat C(t)$ belongs to $\mathscr H$, we have that 
\begin{equation}
|Q(t)|^2-\left(\frac{1}{\sqrt{-|C(t)|_\eta^2}}\left(\frac1{|\Sigma_t|}\int_{\Sigma_t} \sinh(|F_t|)\frac{F_t}{|F_t|} \ d\mu\right)\right)^2=1.
\end{equation}
Setting 
\begin{equation}\label{dfnboldq}
\bold Q(t):=\frac{1}{\sqrt{-|C(t)|_\eta^2}}\left(\frac1{|\Sigma_t|}\int_{\Sigma_t} \sinh(|F_t|)\frac{F_t}{|F_t|} \ d\mu\right),
\end{equation}
it follows that 
\begin{equation}
d(t)=\textup{arcosh}(Q(t))=\textup{arcsinh}(\bold Q(t)).
\end{equation}
Thus, 
\begin{equation}
\left|\frac{d}{dt}d(t)\right|\leq \left|\frac{d\bold Q}{dt}\right|.
\end{equation}
We compute the derivative of $\bold Q(t)$. First of all, by definition of Minkowski norm on the cone, we have
\begin{equation*}
-|C(t)|_\eta^2=\left(\frac{1}{|\Sigma_t|}\int_{\Sigma_t} \cosh(|F_t|) \ d\mu\right)^2-\left(\frac1{|\Sigma_t|}\int_{\Sigma_t} \sinh(|F_t|) \frac{F_t}{|F_t|} \ d\mu\right)^2.
\end{equation*}
Using Lemma \ref{lemma41}, we have that $-|C(t)|_\eta^2\geq C\sigma^2$. Thus, we firstly compute 
\begin{equation*}
\begin{aligned}
\left|\frac{d}{dt}\left(-|C(t)|_{\eta}^2\right)^{-\frac12}\right|&=\left|\frac12\frac{\frac{d}{dt}\left(-|C(t)|_\eta^2\right)}{\left(-|C(t)|_\eta^2\right)^\frac32}\right|\\
&\leq C\sigma^{-3}\left|\frac{d}{dt}\left(\frac{1}{|\Sigma_t|}\int_{\Sigma_t} \cosh(|F_t|) \ d\mu\right)\right|\left(\frac{1}{|\Sigma_t|}\int_{\Sigma_t} \cosh(|F_t|) \ d\mu\right)\\
& \quad + C\sigma^{-3}\left|\frac{d}{dt}\left(\frac{1}{|\Sigma_t|}\int_{\Sigma_t} \sinh(|F_t|)\frac{F_t}{|F_t|} \ d\mu\right)\right|\left(\frac{1}{|\Sigma_t|}\int_{\Sigma_t} \sinh(|F_t|) \ d\mu\right).
\end{aligned}
\end{equation*}
Moreover, we find 
\begin{equation*}
\begin{aligned}
\frac{d}{dt}\left(\frac{1}{|\Sigma_t|}\int_{\Sigma_t} \cosh(|F_t|) \ d\mu\right)=&-\frac{1}{|\Sigma_t|^2}\frac{d|\Sigma_t|}{dt}\int_{\Sigma_t} \cosh(|F_t|) \ d\mu\\
& + \frac1{|\Sigma_t|}\int_{\Sigma_t} \sinh(|F_t|) \partial_t |F_t| \ d\mu\\
& - \frac1{|\Sigma_t|}\int_{\Sigma_t} H (\mathcal H-\hbar) \cosh(|F_t|) \ d\mu.
\end{aligned}
\end{equation*}
Since $\partial_t |\Sigma_t|=\int_{\Sigma_t} H(\hbar-\mathcal H) \ d\mu$ and $\left|\partial_t|F_t|\right|\leq C|\mathcal H-\hbar|$, we find that 
\begin{equation}\label{evcoshcomp424}
\left|\frac{d}{dt}\left(\frac{1}{|\Sigma_t|}\int_{\Sigma_t}\cosh(|F_t|) \ d\mu\right)\right|\leq C\|\mathcal H-\hbar\|_{L^2({\Sigma_t})},
\end{equation}
also using that $|\cosh(|F_t|)|,|\sinh(|F_t|)|\leq C\sigma$. On the other hand,
\begin{equation}
\begin{aligned}
&\frac{d}{dt}\left(\frac1{|\Sigma_t|}\int_{\Sigma_t} \sinh(|F_t|)\frac{F_t}{|F_t|} \ d\mu\right)\\
=&-\frac1{|\Sigma_t|^2}\int_{\Sigma_t} \sinh(|F_t|)\frac{F_t}{|F_t|} \ d\mu+\frac1{|\Sigma_t|}\int_{\Sigma_t}\cosh(|F_t|)\partial_t |F_t|\left(\frac{F_t}{|F_t|}\right) \ d\mu\\
&+\frac1{|\Sigma_t|}\int_{\Sigma_t}\sinh(|F_t|)\partial_t \frac{F_t}{|F_t|} \ d\mu-\frac1{|\Sigma_t|}\int_{\Sigma_t}H(\mathcal H-\hbar)\sinh(|F_t|)\frac{F_t}{|F_t|} \ d\mu.
\end{aligned}
\end{equation}
Since 
\begin{equation}
\partial_t\frac{F_t}{|F_t|}=\frac{\partial_t F_t|F_t|-\partial_t|F_t|F_t}{|F_t|^2}\leq \frac{|\partial_t F_t|}{|F_t|}\leq C|\mathcal H-\hbar|,
\end{equation}
for $\hat r>1$, we end with 
\begin{equation}
\left|\frac{d}{dt}\left(\frac1{|\Sigma_t|}\int_{\Sigma_t} \sinh(|F_t|)\frac{F_t}{|F_t|} \ d\mu\right)\right|\leq C\|\mathcal H-\hbar\|_{L^2(\Sigma_t)}.
\end{equation} 
Combining these, we have
\begin{equation}\label{evnormterm425}
\left|\frac{d}{dt}\left(-|C(t)|_{\eta}^2\right)^{-\frac12}\right|\leq C\sigma^{-2}\|\mathcal H-\hbar\|_{L^2({\Sigma_t})}.
\end{equation}
In conclusion, by the definition \eqref{dfnboldq}, 
\begin{equation}
\left|\frac{d\bold Q}{dt}\right|\leq C\sigma^{-1}\|\mathcal H-\hbar\|_{L^2(\Sigma_t)},
\end{equation}
which ends the proof.
\end{proof}
\begin{cor}\label{cor56}
Let $\Sigma_t$ be a solution to \eqref{generalflow513} satisfying the properties constructed in Theorem \textup{\ref{proofmainthm}}. Then, there exists a universal constant $C=C(m,\cin)>0$ and a radius $\sigma_0=\sigma_0(B_1,B_2,B_3,\cin)>1$ such that, if $\sigma>\sigma_0$, for every $t>0$,
\begin{equation*}
\left|\frac{d}{dt} d_{\mathscr{H}}\left(\hat C_{\Sigma_t},\hat N\right)\right|\leq C\sigma^{-4}e^{-\frac{5m}{9\sigma^2}t}.
\end{equation*}
\end{cor}
\begin{proof}
Combine Lemma \ref{distancegraphorev45} with 
\begin{equation*}
\|\mathcal H-\hbar\|_{L^2(\Sigma_t)}\leq \|\mathcal H-\hbar\|_{L^2(\Sigma_0)} e^{-\frac{5m}{9\sigma^2}t}\leq C\sigma^{-3}e^{-\frac{5m}{9\sigma^2}t}.
\end{equation*}
\end{proof}
\subsection{The ST-Center of Mass} We combine the results of the previous Section in order to straightforwardly obtain the following Proposition.
\begin{prp}\label{prp55} Let $\Sigma_t^\sigma$ be a solution to \eqref{generalflow513} satisfying the properties constructed in Theorem \textup{\ref{proofmainthm}}, with initial datum $\Sigma^\sigma$. Set moreover
\begin{equation*}
\Sigma_\textup{st}^\sigma:=\lim_{t\to\infty} \Sigma_t^\sigma.
\end{equation*}
Then, if $\boldsymbol z_{\Sigma^\sigma_\textup{st}}$ is defined as in Definition \textup{\ref{definitionbarycenter}},
\begin{equation*}
\lim_{\sigma\to\infty} d_{\mathbb H^3}(\boldsymbol z_{\Sigma^\sigma_\textup{st}},\boldsymbol 0)=0.
\end{equation*}
\end{prp}
\begin{proof}
By the proof of \cite[Prp. 3.6]{cederbaumcortiersakovich}, we have that $\hat C_{\Sigma^\sigma}=\hat N+O(\sigma^{-1})$, component-wise, that is 
\begin{equation*}
d_{\mathscr{H}}(\hat C_{\Sigma^\sigma},\hat N)\leq C\sigma^{-1}.
\end{equation*}
Combining this with Corollary \ref{cor56} and integrating, we get that for every $t>0$,
\begin{equation*}
d_{\mathscr{H}}(\hat C_{\Sigma_t^\sigma},\hat N)\leq d_{\mathscr{H}}(\hat C_{\Sigma^\sigma},\hat N)+C\sigma^{-2}\leq C\sigma^{-1}.
\end{equation*}
Taking the limit $t\to\infty$, since $\Sigma_t$ is converging to a STCMC-surface $\Sigma_\infty^\Sigma$ and the right had side is uniform in $t$, 
\begin{equation*}
d_{\mathscr{H}}(\hat C_{\Sigma_\infty^\sigma},\hat N)\leq C\sigma^{-1},
\end{equation*}
and thus 
\begin{equation*}
\lim_{\sigma\to\infty}z_{\Sigma_{\textup{st}}^\sigma}=\lim_{\sigma\to\infty} z_{\Sigma_\infty^\sigma}=0,
\end{equation*}
with respect to the distance $d_{\mathbb H^3}$, thanks to the isometry in \eqref{isometricimmersion}.
\end{proof}
\appendix
\section{The isometric immersion}\label{appendixisoimm}  In \cite{cederbaumcortiersakovich}, in order to deal with the fact that the hyperbolic space is not an affine space, it has been considered the isometric immersion 
\begin{equation}\label{isometricimmersion}
\begin{aligned} I \colon \mathbb{H}^3 &\to \mathbb{R}^{1,3} \\ x &\mapsto I(x):=\left(\cosh(r),\sinh(r)\frac{x}{r}\right),
\end{aligned} 
\end{equation}
where $x=(x_1,x_2,x_3)$ and $r=|x|=\sqrt{x_1^2+x_2^2+x_3^2}$, between the model space $(\mathbb H^3,\overline g^\hp)$ and the Minkowski space $(\R^{1,3},\eta)$, i.e. $\R^4$ equipped with the Minkowski metric 
\begin{equation*}
\left|(y_0,y_1,y_2,y_3)\right|_\eta^2:=-y_0^2+y_1^2+y_2^2+y_3^2.
\end{equation*}
We remark that, despite we use the notation $|\cdot|_\eta^2$, this quantity has not a sign, in general. We also remark that 
\begin{equation*}
\left|I(x)\right|_\eta^2=-\cosh^2r+\sinh^2(r)=-1.
\end{equation*}
We now clarify why the notion of barycenter in Definition \ref{definitionbarycenter} is good for spheres. 
\begin{rem} If $\Sigma=\mathbb S_{\hat r}(0)$, then $C_\Sigma=(\cosh \hat r,0,0,0)$, and thus $\hat C_\Sigma=(1,0,0,0)\in H$. Consequently, $z_\Sigma=0\in\mathbb H^3$.
\end{rem}
\begin{rem}\label{rem44} Consider now $\mathbb S_{\hat r}(a_0)$, with an arbitrary $a_0\in \mathbb H^3$. By transitivity of the hyperbolic space, we find an isometry $B:\mathbb H^3\to \mathbb H^3$ such that $B(0)=a_0$. Then, by definition of isometry and induced metric we have
\begin{equation*}
\int_{\mathbb S_{\hat r}(a_0)} I_\alpha(j) \ d\mu_{\mathbb S_{\hat r}(a_0)}=\int_{\mathbb S_{\hat r}(0)} I_\alpha\left(B\circ \iota\right) \ d\mu_{\mathbb S_{\hat r}(0)},
\end{equation*}
where $\iota:\mathbb S_{\hat r}(0)\hookrightarrow \mathbb H^3$ and $j:\mathbb S_{\hat r}(a_0)\hookrightarrow \mathbb H^3$ are the immersions of the spheres and $I_\alpha$ is a component of the map $I$. In the following, we omit the presence of the immersion, which is clear from the context.\\
\indent For every $y\in I(\mathbb H^3)$, we define $\hat B(y):=I(B(x))$, where $x$ is the unique (by injectivity) element of $\mathbb H^3$ such that $I(x)=y$. Thus, $\hat B$ is an isometry of the hyperboloid $H\subset \mathbb R^{1,3}$, and 
\begin{equation*}
I\circ B=\hat B\circ I
\end{equation*}
by definition. Since the isometries of $H\subset \mathbb R^{1,3}$ are those linear maps which preserve the product $\eta$, and act linearly as a Lorentz transform, we find 
\begin{equation*}
\begin{aligned}
\int_{\mathbb S_{\hat r}(a_0)} I \ d\mu_{\mathbb S_{\hat r}(a_0)}=\int_{\mathbb S_{\hat r}(0)} I\circ B \ d\mu_{\mathbb S_{\hat r}(0)}&=\int_{\mathbb S_{\hat r}(0)} \hat B\circ I \ d\mu_{\mathbb S_{\hat r}(0)}\\
&=\hat B\int_{\mathbb S_{\hat r}(0)} I \ d\mu_{\mathbb S_{\hat r}(0)}
\end{aligned}
\end{equation*}
by the linearity of the integral. Thus, dividing by the area (which is the same for each sphere), 
\begin{equation*}
C_{\mathbb S_{\hat r}(a_0)}=\hat B(C_{\mathbb S_{\hat r}(0)}),
\end{equation*}
using again the linearity. By definition of isometry of the hyperboloid, also $-|C_{\mathbb S_{\hat r}(a_0)}|_\eta^2=-|C_{\mathbb S_{\hat r}(0)}|_\eta^2$. Thus 
\begin{equation*}
\hat C_{\mathbb S_{\hat r}(a_0)}=\hat B(\hat C_{\mathbb S_{\hat r}(0)})=\hat B(I(0))=I(a_0).
\end{equation*}
In general, if $\Sigma$ is a surface away from the origin and $B^{-1}(\Sigma)$ is its "centered version", then $z_\Sigma=B(z_{B^{-1}(\Sigma)})$.
\end{rem}
\section{Some additional computations}\label{appendixAA}
We collect in this Section some computations employed in Section \ref{sect3round}, which are postponed here for sake of readability.
\subsection{An ODI for $|\partial_r^\top|$}
\begin{lem}\label{appendixA}
Let $\Sigma$ be a surface immersed in an asymptotically hyperboloidal initial data set. Suppose that the following holds. There exists $\hat r_0>1$ and $B>0$ such that, similarly to the hypotheses of Proposition \ref{prp311partI}, 
\begin{enumerate}[label=\textup{(\roman*)}]
\item $\frac\sigma2\leq \sinh r_\Sigma\leq \sinh R_\Sigma \leq \frac32\sigma$;
\item $\|\overset{\circ}{A}\|_{L^\infty(\Sigma)}^2+\|H-h\|_{L^\infty(\Sigma)}\leq B\sigma^{-4}$;
\item $|\partial_r^\top|\leq B\sigma^{-1}$ on $\Sigma$;
\item $|\sinh \hat r_0-\sigma|\leq B$;
\item it holds \begin{equation*}
h=2+\frac1{\sinh^2 \hat r_0}+O(\sigma^{-2-\delta}),
\end{equation*}
for some $\delta>0$.
\end{enumerate}
Then, there exists a universal constant $C>0$ such that for every $V\in T\Sigma$ such that $|V|=1$ it holds
\begin{equation*}
\left|\overline\nabla_V|\partial_r^\top|^2\right|\leq C\sigma^{-3}|\partial_r^\top|+C|\partial_r^\top|^3+O(\sigma^{-2})|\partial_r^\top|.
\end{equation*}
In particular, since $|\partial_r^\top|=O(\sigma^{-1})$, then $|\nabla |\partial_r^\top|^2|=O(\sigma^{-3})$.
\end{lem}
\begin{proof}
Let $V$ be a tangent unit vector to $\Sigma$, and decompose $V=\overline V+V_0 \partial_r$, where $\langle \overline V,\partial_r\rangle_\hp=0$. Then
\begin{equation*}
|V_0|=\left|\langle V,\partial_r\rangle_{\overline g}+O(e^{-3r})\right|\leq |\partial_r^\top|+O(e^{-3r}).
\end{equation*}
Note that $\nabla r=\textup{proj}_\Sigma \left(\overline \nabla r\right)=\partial_r^\top+O(e^{-3r})$. Thus, 
\begin{equation}\label{stimaA2}
\left|\langle V-\overline V,\partial_r^\top\rangle\right|\leq |V_0||\langle \partial_r,\partial_r^\top\rangle|\leq |V_0||\partial_r^\top|^2\leq |\partial_r^\top|^3+|\partial_r^\top|^2O(\sigma^{-3r}).
\end{equation}
Set $|\cdot|=|\cdot|_{\overline g}$. Since $|\partial_r^\top|^2=|\partial_r|^2-|\langle \partial_r,\nu\rangle|^2$, we find
\begin{equation*}
\begin{aligned}
\overline \nabla_V|\partial_r^\top|^2&=2\langle \overline\nabla_V\partial_r,\partial_r\rangle-2\langle \partial_r,\nu\rangle\left(\langle\overline\nabla_V \partial_r,\nu\rangle+\langle \partial_r,\overline\nabla_V \nu\rangle\right)\\
&=2\langle\overline\nabla_V\partial_r,\left(\partial_r-\langle \partial_r,\nu\rangle \nu\right)\rangle-2\langle\partial_r,\nu\rangle\langle\partial_r,\overline\nabla_V\nu\rangle\\
&=2\langle\overline\nabla_V\partial_r,\partial_r^\top\rangle-2\langle\partial_r,\nu\rangle\langle\partial_r,\overline\nabla_V\nu\rangle\\
&=2\langle\overline\nabla_V^\hp\partial_r,\partial_r^\top\rangle-2\langle\partial_r,\nu\rangle\langle\partial_r,\overline\nabla_V\nu\rangle+O(e^{-3r})|\partial_r^\top|\\
&=2\langle\overline\nabla_V^\hp\partial_r,\partial_r^\top\rangle-2\langle\partial_r,\nu\rangle\langle\partial_r^\top,\overline\nabla_V\nu\rangle+O(e^{-3r})|\partial_r^\top|.
\end{aligned}
\end{equation*}
since moreover $\langle \nu,\overline\nabla_V\nu\rangle=0$. By definition we also have 
\begin{equation*}
\overline\nabla_V^\hp \partial_r=\frac{\cosh r}{\sinh r} \overline V, \qquad  \langle \overline\nabla_V\nu,\partial_r^\top\rangle=\langle A^\#V,\partial_r^\top\rangle=\frac{H}2\langle V,\partial_r^\top\rangle+\langle \overset{\circ}{A}^\#V,\partial_r^\top\rangle,
\end{equation*}
and thus, since $|\overset{\circ}{A}|=O(\sigma^{-2})$ and $h$ is uniformly bounded,
\begin{equation*}
\begin{aligned}
\left|\overline\nabla_V|\partial_r^\top|^2\right|&=\left|2\frac{\cosh r}{\sinh r}\langle \overline V,\partial_r^\top\rangle-H\langle V,\partial_r^\top\rangle\langle \partial_r,\nu\rangle +O(\sigma^{-2})|\partial_r^\top|\right|\\
&=\left|\left(2\frac{\cosh r}{\sinh r}-h\right)\langle V,\partial_r^\top\rangle-h\langle V,\partial_r^\top\rangle\left(\langle \partial_r,\nu\rangle-1\right)+O(\sigma^{-2})|\partial_r^\top|\right|\\
& \quad +2\frac{\cosh r}{\sinh r}\left|\langle \overline V-V,\partial_r^\top\rangle\right|\\
&\leq \left|\frac1{\sinh^2 r}-\frac1{\sinh^2 \hat r_0}\right||\partial_r^\top|+C|\partial_r^\top|^3+O(\sigma^{-2})|\partial_r^\top|\\
&\leq C\sigma^{-3}|\partial_r^\top|+C|\partial_r^\top|^3+O(\sigma^{-2})|\partial_r^\top|,
\end{aligned}
\end{equation*}
using \eqref{stimaA2} in order to estimate $\overline V-V$ and using $1-\langle\partial_r,\nu\rangle^2=|\partial_r^\top|^2+O(\sigma^{-5})$ combined with (iii) in order to bound $\langle \partial_r,\nu\rangle-1$.
\end{proof}
\subsection{Conformal metrics on $\mathbb S^2$}\label{appendixschauder}
\begin{lem}\label{kwdecomp} In the context of \textup{Proposition \ref{prp311}}, setting $\hat g=\sigma^{-2}g$, there exist a function $\beta:\mathbb S^2\to \R$, and smooth functions $u,v$, such that 
\begin{equation*}
\hat g=e^{2\beta}g_0, \qquad w=u+v,
\end{equation*}
where $g_0$ is the round metric on $\mathbb S^2$ and $u\in \textup{Ker}\left(\Delta_{g_0}+2\right)$ and $v$ orthogonal to $u$ in $L^2(\mathbb S^2)$. Moreover, 
\begin{equation*}
\|u\|_{C^{2,\frac12}(\mathbb S^2)}\leq c\sigma^{-1},
\end{equation*}
for some $c>0$ as in the statement of \textup{Proposition \ref{prp311}}.
\end{lem}
\begin{proof}
We compute the Gaussian curvature of $\hat g$, which is given trough identity \eqref{gaussianid24} by 
\begin{equation}\label{hatkappaidentity}
\begin{aligned}
\hat \kappa=\frac{|\Sigma|}{4\pi}\kappa &=\frac{|\Sigma|}{4\pi}\left(-1+\frac{H^2}{4}+\frac{m}{\sinh^3 r}-\frac{3m|\nu^\top|^2}{2\sinh^3 r}-\frac{|\overset{\circ}{A}|^2}{2}+O(e^{-5r})\right)\\
&=\frac{|\Sigma|}{4\pi}\left(-1+\frac{h^2}{4}+\frac{m}{\sinh^3 r}-\frac{3m|\nu^\top|^2}{2\sinh^3 r}-\frac{|\overset{\circ}{A}|^2}{2}+\frac{H^2-h^2}{4}+O(e^{-5r})\right)\\
&=\sigma^2\left(\frac1{\sinh^2 \hat r}+ O(\sigma^{-3})\right)=1+O(\sigma^{-1}),
\end{aligned}
\end{equation}
where in the latter line we used \eqref{hexpldelta} combined with (i), (ii) and the uniform estimate $|\nu^\top|\leq c_0$. Note that $\left|O(\sigma^{-3})\right|\leq c\sigma^{-3}$, with $c=c(B)$. Thus, by \cite[Lemma 3.7]{christodoulouyau}, we find $\beta:\mathbb S^2\to \R$ (going to zero uniformly as $\sigma\to\infty$) such that $\hat g=e^{2\beta} g_0$, where $g_0$ is the round metric on $\mathbb S^2$ and 
\begin{equation}\label{meanmetric313}
\int_{\mathbb S^2} x_j e^{2\beta} \ d\mu_{g_0}=0, \qquad \forall j\in\{1,2,3\}.
\end{equation}
Moreover, proceeding as in the proof of \cite[Thm. 5.1]{nevestian} we also obtain 
\begin{equation}\label{decaybeta314}
\beta=O(\sigma^{-1}), \qquad \int_{\mathbb S^2}|\nabla_{g_0} \beta|^2 \ d\mu_{g_0}=O(\sigma^{-2}),
\end{equation}
and, by the Kazdan-Warner identity (see \cite{kazdanwarner} and also \cite[Lemma 6.3]{nevestian}),
\begin{equation}\label{KWineq}
\int_{\mathbb S^2}\langle \nabla_{g_0} \hat \kappa,\nabla_{g_0} x_i\rangle e^{2\beta} \ d\mu_{g_0}=0.
\end{equation}
Since $\Delta_{g_0} x_i=-2x_i$, integrating by parts we find
\begin{equation*}
-2\int_{\mathbb S^2} x_i\hat \kappa e^{2\beta} \ d\mu_{g_0}+2\int_{\mathbb S^2} \hat \kappa \langle \nabla_{g_0}\beta,\nabla_{g_0}x_i\rangle e^{2\beta} \ d\mu_{g_0}=0.
\end{equation*}
Combining \eqref{decaybeta314} and $\hat \kappa=1+O(\sigma^{-1})$, and using Hölder's inequality, we find 
\begin{equation*}
\begin{aligned}
\int_{\mathbb S^2}\hat \kappa\langle \nabla_{g_0}\beta,\nabla_{g_0}x_i\rangle e^{2\beta} \ d\mu_{g_0}&=\int_{\mathbb S^2}\langle \nabla_{g_0}\beta,\nabla_{g_0}x_i\rangle \ d\mu_{g_0}+O(\sigma^{-2})\\
&=-2\int_{\mathbb S^2} \beta x_i \ d\mu_{g_0}+O(\sigma^{-2})=O(\sigma^{-2}),
\end{aligned}
\end{equation*}
also using \eqref{meanmetric313}. Thus the Kazdan-Warner identity becomes 
\begin{equation}\label{KW2318}
\int_{\mathbb S^2} x_i\hat \kappa e^{2\beta} \ d\mu_{g_0}=O(\sigma^{-2}).
\end{equation}
Multiplying \eqref{KW2318} by $|\Sigma|^\frac12$ and combining this with the second line in \eqref{hatkappaidentity}, we find
\begin{equation*}
m\int_{\mathbb S^2}x_i|\Sigma|^\frac32 \left(\sinh r\right)^{-3}e^{2\beta} \ d\mu_{g_0}=\int_{\mathbb S^2} \frac{3m|\nu^\top|^2|\Sigma|^\frac32}{8\pi \sinh^3 r} e^{2\beta} \ d\mu_{g_0}+\int_{\mathbb S^2}\frac{|\overset{\circ}{A}|^2}2|\Sigma|^\frac32 e^{2\beta} \ d\mu_{g_0} +O(\sigma^{-1}).
\end{equation*}
Since $|\sigma^3-\sinh^3 \hat r|\leq C\sigma^2$ and $|\sinh r-\frac{e^r}2|=\frac{e^{-r}}{2}$, the left-hand-side becomes 
\begin{equation*}
(4\pi)^\frac32 m\int_{\mathbb S^2} x_i e^{-3(r-\hat r)} \ d\mu_{g_0}+O(\sigma^{-1}).
\end{equation*}
Since $|\overset{\circ}{A}|^2\leq c\sigma^{-4}$, $|\nu^\top|\leq c_0$ with $c_0$ suitably small, by the definition $w=r-\hat r$ we conclude that 
\begin{equation*}
\int_{\mathbb S^2} x_i e^{-3w} \, d\mu_{g_0}=O(\sigma^{-1}).
\end{equation*}
Since by (i) we have that $|w|\leq B\sigma^{-1}$, i.e. it is uniformly small for $\sigma$ suitably large, there exists $C>0$ such that
\begin{equation*}
\left|e^{-3w}-1+3w\right|\leq Cw^2,
\end{equation*}
and thus 
\begin{equation*}
\langle w,x_i\rangle_{L^2(\mathbb S^2)}=\int_{\mathbb S^2} x_i w \, d\mu_{g_0}=O(\sigma^{-1})+O\left(\int_{\mathbb S^2}w^2 d\mu_{g_0}\right).
\end{equation*}
We now decompose $w=u+v$, with $u\in \textup{Ker}(\Delta_{g_0}+2)$ and $v$ orthogonal to $u$. Note that $\dim \textup{Ker}(\Delta_{g_0}+2)=3$ and this kernel is spanned by three eigenfunctions $f_1,f_2,f_3$ which, except for a multiplicative normalization constant, coincide with $x_1,x_2,x_3$. Thus there exists $C>0$ such that 
\begin{equation*}
u=C\sum_{i=1}^3 \langle w,x_i\rangle_{L^2(\mathbb S^2)}x_i. 
\end{equation*}
Thus, since $w^2\leq B^2\sigma^{-2}$ and $x_i$ and its derivative are bounded on $\mathbb S^2$, 
\begin{equation}\label{inequC012}
\|u\|_{C^{0,\frac12}(\mathbb S^2)}\leq C\sum_{i=1}^3|\langle w,x_i\rangle_{L^2(\mathbb S^2)}|\leq c\sigma^{-1}
\end{equation}
Since moreover by definition $-\Delta_{g_0} u=2u$, Schauder's estimate implies  
\begin{equation}\label{C2ofu323}
\|u\|_{C^{2,\frac12}(\mathbb S^2)}\leq c\sigma^{-1}.
\end{equation}
\end{proof}
\section{A direct computation}\label{directcomputation}
We start with the identity
\begin{equation*}
\nabla_X\langle \partial_r,\nu\rangle=\textup{Hess}_r(X,\nu)+A(X,\partial_r^\top).
\end{equation*}
Since $\overline g$ can be directly compared to $\overline g^m$, we find that 
\begin{equation*}
\textup{Hess}_r=\textup{coth} r \ (\overline g-dr^2)+O(\sigma^{-5}).
\end{equation*}
Thus, since $\nu=\langle \nu,\partial_r\rangle^{-1}(\partial_r-\partial_r^\top)$, we have that 
\begin{equation*}
\nabla_X\langle \partial_r,\nu\rangle=(A-\textup{coth} r \ g)(X,\partial_r^\top)+O(\sigma^{-5}).
\end{equation*}
By the fact that $\left|\frac{H}2-\textup{coth} r\right|\leq C\sigma^{-2}$, we conclude that $|\nabla\langle \partial_r,\nu\rangle|\leq C\sigma^{-4}$.\\
\indent To estimate the second derivative of the angle, the computation is more involved but the idea is similar. Defining $\alpha:=\langle \partial_r,\nu\rangle$, 
\begin{equation*}
\nabla^2 \alpha=\left(\overline\nabla \textup{Hess}_r\right)(\cdot,\nu)-A*\textup{Hess}_r(\nu,\nu)+\textup{Hess}_r(\cdot,A^\#)+(\nabla A)(\cdot,\partial_r^\top)+A(\cdot,\nabla \partial_r^\top).
\end{equation*}
Since $dr(\nu)=\alpha$ by definition, we have 
\begin{equation*}
\begin{aligned}
\overline\nabla_\cdot\textup{Hess}_r(\cdot,\nu)&=-\alpha\textup{coth}'r \ dr^2-\textup{coth} r \left(\alpha\textup{Hess}_r+\textup{Hess}_r(\cdot,\nu)dr\right)+O(\sigma^{-5})\\
&=-\alpha\textup{coth}'r \ dr^2-\textup{coth}^2 r \ \alpha(g-dr^2)+\textup{coth}^2 r \ \alpha dr^2+O(\sigma^{-5})\\
&=-\textup{coth}^2 r \ \alpha g+(2\textup{coth}^2 r-\textup{coth}'r)\alpha dr^2+O(\sigma^{-5}),
\end{aligned}
\end{equation*}
using that 
\begin{equation}\label{useful2532}
\begin{aligned}
\textup{Hess}_r=\textup{coth}r \ (g-dr^2)+O(\sigma^{-5}), \qquad \textup{Hess}_r(\cdot,\nu)=-\textup{coth}r \ \alpha dr+O(\sigma^{-5}).
\end{aligned}
\end{equation}
Moreover $\textup{Hess}_r(\nu,\nu)=\textup{coth}r \ (1-\alpha^2)+O(\sigma^{-5})=O(\sigma^{-4})$. We conclude since, by \eqref{useful2532} again, $\textup{Hess}_r(\cdot,A^\#)=\textup{coth}r \ (g(\cdot,A^\#)-dr\otimes dr(A^\#))+O(\sigma^{-5})=\textup{coth}r \ A+O(\sigma^{-5})$ and since, by $\langle \overline \nabla _X\partial_r^\top,Y\rangle=\textup{Hess}_r(X,Y)=\textup{coth} r \ g(X,Y)+O(\sigma^{-4})$ using $dr^2=O(\sigma^{-4})$ as before, 
\begin{equation*}
A(\cdot,\nabla\partial_r^\top)=A(\cdot,\overline\nabla\partial_r^\top-\alpha A^\#)=\textup{coth} r \ A-\alpha A^2+O(\sigma^{-4}).
\end{equation*}
Putting the pieces together, we conclude that 
\begin{equation}\label{inequalityalphaangle}
|\nabla^2\alpha|\leq C\sigma^{-4}+C\sigma^{-2}|\nabla A|.
\end{equation}
\noindent \textbf{Conflict of interest:} The author has no relevant financial or non-financial interests to disclose.
\\
\\
\textbf{Data availability:} Not applicable because no datasets where used during the current study.
\\
\\
\textbf{AI disclosure:} Claude Sonnet 5.5 was used to assist the author in organizing the proofs of Section \ref{thefoliationproof} and in verifying the routine computations sketched therein.

\end{document}